\documentclass[english,12pt,a4paper,table]{article}

\usepackage{makecell}
\usepackage{color}
\usepackage[table]{xcolor}
\usepackage{pgf,tikz}
\usetikzlibrary{arrows,patterns}

\usepackage{amsmath}
\usepackage{amsfonts}
\usepackage{amssymb}
\usepackage{amsthm}
\usepackage{mathtools}
\usepackage{lscape}

\newcommand{\N}{\ensuremath{\mathbb{N}}}
\newcommand{\Bernshtein}{Bernshtein}
\newcommand{\Bezout}{B\'ezout}

\newcommand{\C}{\ensuremath{\mathbb{C}}}
\newcommand{\R}{\ensuremath{\mathbb{R}}}
\newcommand{\V}{\ensuremath{\mathbf{V}}}
\newcommand{\I}{\ensuremath{\mathbf{I}}}

\newcommand{\newton}[1][f]{\ensuremath{\mathcal{N}(#1)}}
\newcommand{\conv}{\ensuremath{\mathrm{conv}}}
\newcommand{\cone}{\ensuremath{\mathrm{cone}}}
\newcommand{\LC}{\ensuremath{\mathbf{LC}}}
\newcommand{\LM}{\ensuremath{\mathbf{LM}}}
\newcommand{\Vol}{\ensuremath{\mathrm{Vol}}}

\newcommand{\cornerideal}{\ensuremath{\langle f(a + c, b + d), f(a + d, b - c), f(a - c, b - d), f(a - d, b  + c) \rangle}}

\usepackage{fancyref}
\newcommand*{\fancyreflstlabelprefix}{lst}
\newcommand*{\fancyrefexlabelprefix}{ex}
\newcommand*{\fancyrefdeflabelprefix}{def}
\newcommand*{\fancyrefthmlabelprefix}{thm}
\newcommand*{\fancyreflemlabelprefix}{lem}
\newcommand*{\fancyrefcorlabelprefix}{cor}
\Frefformat{vario}{\fancyreflstlabelprefix}{Listing~#1 on page #2}
\Frefformat{vario}{\fancyrefexlabelprefix}{Example~#1 on page #2}
\Frefformat{vario}{\fancyrefthmlabelprefix}{Theorem~#1}
\Frefformat{vario}{\fancyreflemlabelprefix}{Lemma~#1}
\Frefformat{vario}{\fancyrefcorlabelprefix}{Corollary~#1}
\Frefformat{vario}{\fancyrefdeflabelprefix}{Definition~#1}

\usepackage{multirow}
\makeatletter
\newtheorem*{rep@theorem}{\rep@title}
\newcommand{\newreptheorem}[2]{%
\newenvironment{rep#1}[1]{%
 \def\rep@title{#2 \ref{##1}}%
 \begin{rep@theorem}}%
 {\end{rep@theorem}}}
\makeatother

\newtheorem{definition}{Definition}
\newtheorem{theorem}{Theorem}[section]
\newreptheorem{theorem}{Theorem}
\newtheorem{lemma}[theorem]{Lemma}

\newtheorem{corollary}[theorem]{Corollary}

\newtheorem{example}[theorem]{Example}

\usepackage[nottoc]{tocbibind}

\usepackage{fancyvrb}
\usepackage{verbatim}
\usepackage{dsfont}
\usepackage{bbm}
\usepackage{listings}

\usepackage{caption}
\usepackage{subcaption}

\DeclareCaptionSubType*[alph]{figure}

\usepackage{float}

\usepackage{longtable}

\usepackage{imakeidx}

\usepackage{graphicx}
\usepackage{hyperref}

\newcommand{\lqindex}[1]{\emph{#1}\index{#1}}
\newcommand{\inscribeOne}[2]{#1 square inscribed on
\hyperlink{poly:f#2}{$f_{#2}$} in \Fref{tab:long-polys}}
\newcommand{\inscribe}[2]{#1 squares inscribed on
\hyperlink{poly:f#2}{$f_{#2}$} in \Fref{tab:long-polys}}

\makeatletter
\providecommand\phantomcaption{\caption@refstepcounter\@captype}
\makeatother

\begin{document}

\author{
Wouter van Heijst}
\title{The algebraic square peg problem \\
{\small Master's thesis in mathematics, Aalto University, March 2014.}
}
\date{}

\maketitle

\newpage

\tableofcontents

\section{Introduction}

Toeplitz conjectured in 1911 that every continuous closed curve in the plane
that does not self-intersect, also known as a Jordan curve, contains all four corners of some square. More
than a hundred years have passed since the statement of Toeplitz's conjecture;
various partial results assuming the curve satisfies additional smoothness properties have been
proven, but in full generality the problem remains unsolved.

Why look at squares?  The conjecture does not hold if squares are
replaced with regular polygons with more than four vertices;
Eggleston~\cite{eggleston} gave an example of a convex curve, a curve
that is the boundary of a convex region of the plane, that does not
inscribe any regular polygon with more than four vertices.
On the other hand, the conjecture does hold if squares are replaced by triangles or
rectangles; Nielsen~\cite{triangles} showed that any Jordan curve
inscribes a triangle and Vaughan, by way of Meyerson~\cite{meyerson},
proved that every Jordan curve inscribes some rectangle. Vaughan's proof
has no control over the aspect ratio of the inscribed rectangle.  Both
these cases are discussed in Igor Pak's online book ``Lectures on Discrete and
Polyhedral Geometry''~\cite[Section~5,  ``Inscribed and
circumscribed polgons'']{pak-book}.
   We shall concern
ourselves in this thesis with the special case of inscribing a
rectangle with prescribed equal aspect ratio, otherwise known as a
square. See
Matschke's survey paper~\cite[Section~4]{Matschke} for further
problems related to the square peg problem.

Initial publications on the square peg problem, as Toeplitz's conjecture has
become known, were made by Emch; who proved the
existence of an inscribed square on convex curves~\cite{1913} in 1913 and
three years later for piecewise analytic curves with a finite number of singularieties~\cite{1916}.
According to Matschke~\cite[Emch's~proof]{Matschke}, implicit in Emch's work is
the understanding that a generic curve inscribes an odd number of squares. Since
zero is not an odd number, such a parity argument implies the existence of at
least one inscribed square, thereby proving Toeplitz's conjecture for these
restricted classes of curves. The sense of genericity is important; Popvassilev
showed that for any natural number $n$, there exists a continuous curve
that inscribes exactly $n$ squares~\cite{popvassilev}.

Further work on the square peg problem came from, among others, the
hands of Jerrard~\cite{jerrard}, and
Stromquist~\cite{stromquist}. Jerrard's proof for analytic curves and \linebreak Stromquist's proof for locally monotone curves both show
show that generically the number of squares inscribed on a smooth enough curve
is odd.  Stromquist's locally monotone curves is one of the largest
classes for which Toeplitz's conjecture is known to hold.  In
more recent years Pak~\cite{pak} has given an elementary proof for
piecewise linear curves while Matschke~\cite[Theorem~3.3]{Matschke}
has generalized the square peg problem to arbitrary metric spaces.

\hspace{-0.5em}We refer readers interested in the history of the square peg problem
to Matschke's survey paper~\cite{Matschke} or the papers of
Sagols and
Mar\'{\i}n~\cite[Section~1]{sagols} and
Pak~\cite[Section~3]{pak}.
\\

In this thesis we shall employ algebra, rather than the analytical and
topological methods of the above approaches, to count the number of
squares that may be inscribed on a curve.  Thus the class of curves we
consider is that of the algebraic plane curves, which are curves defined by the vanishing
of a polynomial in two variables. These are no longer neccessarily
Jordan curves, but exhibit interesting behaviour
nonetheless. The main result of this thesis, \Fref{thm:mine}, states
that an algebraic plane curve of degree $m$ inscribes at most $(m^4 -
5m^2 + 4m)/4$ isolated squares. \Fref{sec:experimental} provides some
evidence for the claim that a generic complex algebraic plane curve
inscribes exactly $(m^4 - 5m^2 + 4m)/4$ squares. The behaviour of real
algebraic plane curves is less clear, examples of real algebraic plane
curves of different topological types inscribing various numbers of
squares are listed in \Fref{sec:illustrative}.  Those examples form
the basis for three conjectures in \Fref{sec:conclusions}, similar to
the results from Emch, Jerrard, and Stromquist that a generic Jordan curve inscribes
an odd number of squares. The most striking of these, to the author's
eyes at least, is the conjecture that an algebraic plane curve
homeomorphic to the real line inscribes an even number of squares.
\\

The outline of this thesis is as follows: In \Fref{sec:background} we
recall some algebra, polytope theory, and algebraic geometry to support understanding of the statement of \Bernshtein's Theorem,
\Fref{thm:Bernshtein}.  In \Fref{sec:formulation} we formulate the
algebraic square peg problem; we parametrize a complex square in
\Fref{def:complex-square} as a $4$-tuple $(a, b, c, d)$ where $(a, b)$
is the center of the square and the four corners are offset from the
center by $(c, d)$, $(-d, c)$, $(-c, -d)$ and $(d, -c)$.  Evaluating a
polynomial $f$ at these four corners gives the four generators of the
corner ideal that describes all squares inscribed on the algebraic plane
curve defined by $f$.  \Bernshtein's Theorem provides an estimate on
the number of isolated solutions to this system of four polynomials.
While the immediate estimate is no better than \Bezout's bound, in
\Fref{sec:upper-bound} we show that a different choice of generators
yields Newton polytopes whose mixed volume gives exactly the bound
$(m^4 - 5m^2 + 4m)/4$ on the number of inscribed isolated squares.
That this bound is tight, at least for low degrees, is exhibited by
experimental data in \fref{sec:experimental}.    In
\Fref{sec:illustrative} we picture simple real algebraic plane curves of
degrees three to eight inscribing varying numbers of squares.  Finally
we discuss some directions for future work in \Fref{sec:conclusions}.

\section{Background}
\label{sec:background}

The square peg problem is inherently a geometric problem: Whether a
curve inscribes a square depends on the lengths of and angles between
line segments connecting pairs of points on the curve.  Considering
squares inscribed on algebraic curves allows us to view the square peg
problem as an an algebraic problem as well.  The gain of this approach
is that we can use algebraic tools, such as \Bernshtein's Theorem, to
make definite statements about the set of inscribed squares.

The main result of this thesis, \Fref{thm:mine}, states that the
number of isolated squares inscribed on an algebraic curve of degree
$m$ is at most $(m^4 - 5m^2 + 4m)/4$.  The proof of this result
depends on \Bernshtein's Theorem, \Fref{thm:Bernshtein}, which bounds the
number of solutions to a polynomial system of equations by the mixed
volume of the Newton polytopes of the generators of that polynomial
system. The purpose of this background section is to present enough
knowledge about these concepts such that readers who were not
previously familiar with them can understand the statement of
\Bernshtein's Theorem.
\\

In \Fref{sec:background-algebra} we will recall some basic facts about
polynomials and
ideals of polynomial rings. The fact that each ideal is
finitely generated is known as Hilbert's Basis Theorem
(\Fref{lem:Noetherian}).

We discuss convexity, polytopes, simplices, Minkowski sums, Schlegel diagrams,
normal fans, Newton polytopes and the definition of the mixed volume
in \Fref{sec:background-polytopes}.

In \Fref{sec:background-varieties} we mention the Nullstellensatz,
which states that over an algebraically closed field, the radical of
any ideal defining a variety is exactly the ideal of polynomials
vanishing on that variety.  We also show that varieties consist of a
finite number of irreducible components (\Fref{lem:finite-components}),
and the fact that the saturation of an ideal $I$ with respect to an
ideal $J$ corresponds to the difference in varieties of $I$ and $J$ (\Fref{lem:saturation}).
These two results will be used in
\Fref{sec:upper-bound} and \Fref{sec:experimental} to ensure that we
are counting all the non-degenerate squares inscribed on an algebraic plane curve.

The algebra and results on varieties follow the expositions of
Cox~\cite{IVA} and Eisenbud~\cite{view}.   The polytope theory derives
from Ziegler's book on polytopes~\cite[Chapters~0, 1,
2, 5 and 7]{ziegler}.  \Fref{def:mixed-volume} of the mixed volume is
taken from Schneider's book on convex bodies~\cite{schneider}.

Readers familiar with these topics can safely skip this background
section and proceed immediately to \Fref{sec:formulation}.

\subsection{Algebra}
\label{sec:background-algebra}
Algebraic plane curves are a special case of geometric objects called
varieties. Varieties are defined by the vanishing of a set of
polynomials; in the case of plane curves these are polynomials in two
variables.   Before we discuss these algebraic geometric objects in
\Fref{sec:background-varieties}, we define some basic notions
concerning polynomials and their natural environments, polynomial rings.

Let $x_1, \dots, x_n$ be $n$ independent variables and $\alpha \in
\N^n$ a tuple of nonnegative integers.
A \lqindex{monomial}
$x^\alpha = x_1^{\alpha_1} \dots x_n^{\alpha_n}$ is a product of
powers of the variables $x_i$.  The degree of a monomial $x^\alpha$ is the
sum  $\alpha_1 + \dots + \alpha_n$ of the entries of its exponent.
A \lqindex{polynomial} over
a field $\mathds{k}$ in $x_1, \dots, x_n$ is a finite sum
$\sum_{\alpha \in \N^n} c_\alpha x^\alpha$ of monomials where the
coefficients $c_\alpha$ are elements of the field $\mathds{k}$.
The \lqindex{total degree} (or simply degree) $\deg f$ of
a polynomial is the maximal degree of its monomials; the degree of $3xy^2 - xy$ is three due to the exponent $(1, 2)$ of the monomial $xy^2$.

The collection of all polynomials in $x_1, \dots, x_n$ over
$\mathds{k}$, denoted $\mathds{k}[x_1,
\dots, x_n]$, is called a \lqindex{polynomial ring}.
This terminology is justified, as multiplication and addition of polynomials equip
$\mathds{k}[x_1, \dots, x_n]$ with the structure of a ring.
A \lqindex{monomial ordering}  $<$ on a polynomial ring is a
binary relation with the following properties for any distinct exponents $\alpha, \beta
\in \N^n$,
\begin{enumerate}
\item either $x^\alpha < x^\beta$ or $x^\beta < x^\alpha$ (linear ordering)
\item $x^\alpha < x^\beta$ implies $x^{\alpha + \gamma} < x^{\beta +
    \gamma}$ for any $\gamma \in \N^n$.
\item $1 < x^\gamma$ for any nonzero $\gamma \in \N^n$ (well-ordering).
\end{enumerate}
As usual with orderings we write $x^\alpha \leq x^\beta$ if either
$x^\alpha = x^\beta$ or $x^\alpha < x^\beta$.
The leading monomial $\LM_<(f)$ of a polynomial $f$
compares greater than any other monomial of $f$ with respect to the
ordering $<$.  The coefficient of
the leading monomial is denoted $\LC_<(f)$. The explicit dependence on
the particular ordering $<$
is suppressed if no confusion is likely to arise.  There is only one monomial ordering on univariate
polynomials, $x^d < x^e$ if $d < e$, but multivariate polynomials admit
many different monomial orderings.

Certain subsets of $\mathds{k}[x_1, \dots, x_n]$ hold special interest
  for us.  A subset $I \subset \mathds{k}[x_1, \dots, x_n]$ is called an
  \lqindex{ideal} if it is closed under multiplication by elements of the
  polynomial ring and closed under addition by elements of $I$.  These
  conditions can be compactly stated with set-wise addition and
  multiplication notation, respectively $\mathds{k}[x_1, \dots, x_n]I
  \subset I$ and $I + I \subset I$.

 The
  set $\{ 0 \}$ is an ideal as $0 + 0 = 0$ and $f \cdot 0 = 0$ for any
  polynomial $f \in \mathds{k}[x_1, \dots, x_n]$. The set $\{x, y\} \subset
  \mathds{k}[x, y, z]$ on the other hand is not an
  ideal; neither $x + y$ nor $xz$ are contained in $\{x, y\}$, so
  $\{x, y\}$
  violates both closedness properties of an ideal.  The set $\{ xf
\mid f \in \mathds{k}[x, y] \}$ of ``polynomial consequences of $x$'' is
again an ideal of $\mathds{k}[x,y]$; both the addition of elements
$xg + xg' = x(g + g')$ and the multiplication of an element $xg$ with an
arbitrary polynomial $g'$ are of the form $xf$ required to be an element of the set.

Any ideal $I$ can be expressed as the consequence of an, a priori
possibily infinite, set of
\index{ideal!generators of}generators $B_I$ called a \emph{basis}\index{ideal!basis} for $I$,
\[
   I = \langle B_I \rangle = \left\{ \sum_{i=1}^r h_i g_i
     \mid r \in \N, g_i \in B_I, h_i \in \mathds{k}[x_1, \dots, x_n] \right\}.
\]
The ideals $\{ 0 \}$ and $\{ xf \mid f \in \mathds{k}[x, y] \}$ are
generated by single polynomials, $0$ and $x$ respectively.
Bases are not unique, as the examples $\langle x, y \rangle =
\langle x + y, x - y \rangle$ and $\langle x, xy, y\rangle = \langle
x, y \rangle$ show.  If $I$ has a finite, basis $I$
is \emph{finitely generated}\index{ideal!finitely generated}.

A ring with the property that every ideal is finitely generated is
called \lqindex{Noetherian}.  It is easy to see that all fields are
Noetherian; any ideal $I \subset \mathds{k}$ other than $\langle 0 \rangle$
contains some nonzero element $u$.  Since all nonzero elements of
$\mathds{k}$ are invertible and $I$ is closed under multiplication by
field elements, $r = ru^{-1}u \in I$ for all $r \in \mathds{k}$.  But
then $I$ is the entire field itself, $I = \langle 1 \rangle$.
As all ideals of a field are generated by a single element, any field
is clearly Noetherian.

  As a consequence of the next lemma, polynomial rings over a field are Noetherian
as well.
\begin{lemma}[{Hilbert's Basis Theorem~\cite[Theorem~1.2]{view}}]
  \label{lem:Noetherian}
  Let $R$ be a Noetherian ring. Then $R[x]$ is Noetherian.
\end{lemma}
\begin{proof}
Let $I \subset R[x]$ be an ideal. Select elements $f_i \in I$ as
follows. If $I = \langle f_1, \dots, f_i \rangle$, stop.  Otherwise
choose $f_{i + 1} \in I \setminus \langle f_1, \dots, f_i \rangle$ of
minimal degree.

The leading coefficients of the $f_i$ generate an
ideal $\langle \LC(f_1), \LC(f_2), \dots
\rangle$ of $R$. This ideal is finitely generated since $R$ is Noetherian.  Let
$m$ be the smallest index such that the first $m$ leading coefficients
generate the entire ideal of leading coefficients, $\langle \LC(f_1), \dots, \LC(f_m)
\rangle = \langle \LC(f_1), \dots \rangle$.  We claim that our process must have stopped at $f_m$,
that is, $I = \langle f_1, \dots, f_m \rangle$.

Suppose we had picked an
$f_{m + 1}$.  By assumption on $m$ the leading coefficient
$\LC(f_{m+1})$ can be expressed as a linear combination $\sum_{j=1}^m u_j \LC(f_j)$ of the earlier leading coefficients.  The
polynomial $g = \sum_{j=1}^m u_j f_j x^{\deg f_{m+1} - \deg f_j}$ has
the same degree and leading term as $f_{m + 1}$ by construction.   Their difference, $f_{m + 1} -
g$, is of strictly smaller degree than $f_{m+1}$. By minimality of
$f_{m+1}$, the difference $f_{m+1} - g$ must be an element of $\langle
f_1, \dots, f_m \rangle$.  As $f_{m+1}$ is the sum of two elements of
$\langle f_1, \dots, f_m \rangle$, it must itself be an element of this
ideal, which contradicts the choice of $f_{m+1}$.
\end{proof}

Hilbert's Basis Theorem is stated for univariate polynomials with coefficients in a Noetherian ring;
as we can rewrite a polynomial $\sum c_\gamma x_1^{\gamma_1} \dots
x_n^{\gamma_n}$ as a sum $\sum_{i=0}^r ( \sum_{\gamma_n = i}
  c_\gamma x^{\gamma_1} \dots x_{n-1}^{\gamma_{n - 1}}) x_n^i$
of monomials in $x_n$ with coefficients in $\mathds{k}[x_1, \dots,
x_{n-1}]$, the polynomial ring $\mathds{k}[x_1, \dots, x_n] =
\mathds{k}[x_1, \dots, x_{n-1}][x_n]$ is Noetherian as well.

A sequence $(A_1, A_2, \dots)$ of nested sets is called
\lqindex{ascending} if $A_i \subset A_{i + 1}$ and \lqindex{descending} if $A_i
\supset A_{i + 1}$.  Such a sequence terminates, or stabilizes, if the
tail of the sequence is constant, that is,
$A_n = A_N$ for some $N \in \N$ and all $n \geq N$.
If every ascending chain of ideals of a ring $R$ terminates, $R$ is said to
satisfy the \lqindex{Ascending Chain Condition} (ACC).  The Ascending Chain Condition on a ring and a ring being Noetherian are two different ways of looking at the same property.
\begin{lemma}
  The Ascending Chain Condition and being Noetherian are equivalent.
\end{lemma}
\begin{proof}
Let $R$ be a Noetherian ring and let $I_1 \subset I_2 \subset \dots$ be an ascending chain of ideals. The union
$I = \cup_{1}^\infty I_i$ is again an ideal, since $f, g \in I$
implies that $f, g \in I_r$ for some $r$ large enough.    By
assumption $I$ is finitely generated, say $I = \langle f_1, \dots, f_m
\rangle$.  The chain terminates at the smallest index $j$ such that $f_1,
\dots, f_m \in I_j$.
\\

Assume that a ring $R$ has the Ascending Chain Condition and let $I$
be an ideal of $R$. Pick $f_1 \in I$ and $f_{i + 1} \in I \setminus \langle
f_1, \dots, f_i \rangle$. The ideals $I_i = \langle f_1, \dots, f_i
\rangle$ so constructed form an ascending chain. By the ACC the chain terminates,
providing a finite set of generators for $I$.
\end{proof}

In the sequel we separate non-degenerate squares from degenerate
squares inscribed  on a curve by taking the difference of varieties.
The corresponding algebraic operation is called saturation, which is
phrased in terms of colon ideals.
Let $I, J \subset \mathds{k}[x_1, \dots, x_n] = R$ be ideals.  The
\lqindex{colon ideal} $I : J$ is the set
   $\{ f \in R : fJ \subset I \}$.    The colon ideal $\langle xy
   \rangle : \langle y \rangle$ contains all polynomials $f$ such that
   $fy \in \langle xy \rangle$. It does not contain the polynomial
   $1$, as $y$ is not an element of $\langle xy \rangle$. It does
   contain $x$, and it is not hard to show that $\langle xy \rangle :
   \langle y \rangle = \langle x \rangle$.

Recall that the notation $J^m$ denotes the set of all products
$\prod_{i=1}^m j_i$ with $m$ factors from $J$.
The \lqindex{saturation} $I : J^\infty$ of $I$
   with respect to $J$ is the ideal $\bigcup_{m=0}^\infty I : J^m$.
The colon ideals $I : J^m$ form an ascending chain; as $I$ is an ideal and thus closed under multiplication by the ring, the condition $fJ
\subset I$ implies that $fJ^2
   \subset I$. The ascending chain $I
   \subset I : J \subset I : J^2 \subset \dots$ terminates because
   polynomial rings are Noetherian, and thus the saturation $I :
   J^\infty = I : J^M$ for some $M \in \N$.

For multivariate
polynomials it is often convenient to think about all the monomials of
a certain degree separately. The monomials of a fixed degree form a basis for the vector space of all homogenenous polynomials of that degree.
 A general approach for grouping objects
with the same properties together is to work with a grading.
A \lqindex{grading} of a ring $R$ is a decomposition of $R$ as a direct
sum $R_0 \oplus R_1 \oplus \dots$ into abelian groups $R_i$ with the
property that $R_i R_j \subset R_{i + j}$.  An element $f \in R_k$ is
called a \lqindex{homogeneous} element, or a \lqindex{form}, of degree
$k$.
A polynomial ring has a grading by total degree where the homogeneous polynomials
of degree $k$ are sums of monomials of total degree $k$.
The
homogeneous parts of a polynomial $f$ are homogeneous elements $f_i \in R_i$ such
that $f_1 + \dots + f_{\deg f} = f$.   The three homogeneous parts of $f =
3x^3y^3 +
xy + 2x^2 + 1$ are the forms $3x^3y^3$, $xy + 2x^2$ and $1$.

\subsection{Polytopes}
\label{sec:background-polytopes}
\Bernshtein's Theorem is stated in terms of polynomials, varieties,
Newton polytopes and mixed volumes.  We discussed polynomials in the
previous section and will discuss varieties in the next section.  The
current section contains the definition of mixed volume and enough polytope theory to understand the
statement of \Bernshtein's Theorem, as well as the proofs in \Fref{sec:upper-bound}.

Throughout this section $V$ denotes the ambient vector space
containing the geometric objects of interest. Its dual space $V^*$
consists of all linear functionals $\alpha\colon V \to \mathds{k}$.
The notation $\langle \alpha, v \rangle$ denotes the functional
pairing $\langle \alpha, v \rangle = \alpha(v)$ as well as the inner
product on $V$ by identifying the functional $\alpha \in V^*$ with a
suitable vector $\alpha \in V$.  As $V$ will always be
finite-dimensional in this thesis, no confusion is likely to arise.
The standard basis vectors $e_i$ of $V$ are unit vectors whose $i$-th
coordinate is one. The standard basis vectors of the plane are $e_1
= (1, 0)$ and $e_2 = (0, 1)$.
\\

Polytopes are a particular nice class of convex geometric objects.
A set $S$ is \lqindex{convex} if it contains all line segments between
its constituent points.
Equivalently, convexity of $S$ can be expressed as the property
that $S$ contains all the convex combinations of its elements. A
finite sum $\sum_{i=1}^r t_i s_i$ is a \lqindex{convex combination} of
elements $s_i$ of $S$
if all
the $t_i$ are non-negative and sum to one.
This leads us to the definition of the
\lqindex{convex hull} of $S$, the set of all convex combinations of elements of $S$,
\[
   \conv S = \left \{ \sum_{i=1}^r t_i s_i \mid r \in \N, s_i \in S,
     t_i \geq 0, \sum_1^r t_i = 1 \right\}.
\]

If we do not require that the $t_i$ are non-negative, a finite sum
$\sum_{i=1}^r t_i s_i$ such that the $t_i$ sum to one is an
\lqindex{affine combination} of the elements $s_i$.  The
\lqindex{affine hull} is defined analogously to the convex hull.  The
affine hull of a subset $S$ of $V$ is the smallest \lqindex{affine} subspace of
$V$ that contains $S$.  If the affine subspace contains the element $0$ it is also a
linear subspace of $V$. If $0$ is not contained in an affine subspace $A$, then $A$ is the translation of some linear subspace of $V$. The dimension of an affine subspace is the dimension of the linear subspace it is a translate of. Consider affine space a linear space where we have forgotten how to distinguish the zero element.

\begin{figure}[H]
\centering
  \includegraphics[width=0.6\textwidth]{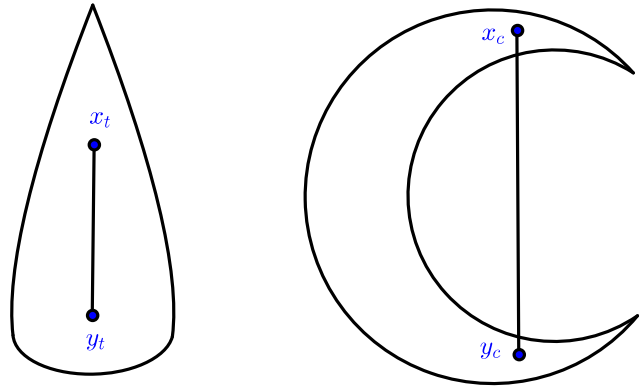}
\captionsetup{width=0.7\textwidth}
\caption{The teardrop is convex because it contains every line segment between two of its points. The crescent is not convex.}
\end{figure}

The line $y = x + 1$ is not a linear subspace of $\R^2$
since it does not contain the origin, but it is an affine subspace.
For linear subspaces we are used to the concept of linear
independence, affine subspaces have a similar concept of affine independence.
A set $\{p_1, \dots, p_r\} \subset V$ of points is \lqindex{affinely
  independent} if no $p_i$ is contained in affine hull spanned by the other $p_j$. Linear independence
implies affine independence, but not vice versa.
 The set $
\{(1, 0), (0, 1), (1, 1)\}$ is affinely independent since a line
  through two of the points does not contain the third. The set $\{(1,
    0), (0, 1), (1/2, 1/2)\}$ is affinely dependent as the three points are collinear.
These affine hulls are depicted in
  Figure  \ref{fig:affine-hulls}.
\begin{figure}[H]
\centering
  \includegraphics[width=0.4\textwidth]{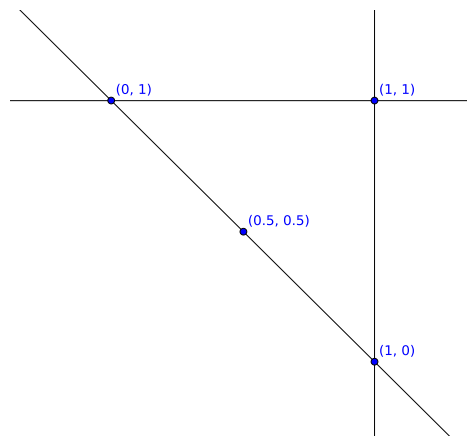}
  \caption{The affine hull of a pairs of points, or collinear points, is a line.}
  \label{fig:affine-hulls}
\end{figure}

Points, edges, triangles, tetrahedra and their higher-dimensional
generalizations have the property that their vertices are affinely
independent; an $n$-\lqindex{simplex} is the convex hull of $n + 1$
affinely independent points.  The convex hull of the origin and the
$n$ standard basis vectors $e_i$ of an $n$-dimensional vector space is
an $n$-simplex.  In one, two and three dimensions the volumes of such
simplices are $1$, $1/2$ and $1/6$.  Volume is invariant under
translation, so the volume of an $n$-simplex with vertices $v_0$,
$v_1$, \dots, $v_n$ is the same as that of the $n$-simplex with
vertices $0$, $v_1 - v_0$, \dots, $v_n - v_0$.  The matrix with colum
vectors $v_i - v_0$ maps the vertices of $\conv(0, e_1, \dots, e_n)$
to the vertices $\conv(0, v_1 - v_0, \dots, v_n - v_0)$.  The volume
of the second simplex is proportional to the volume of the first
simplex, as the determinant of a matrix can be interpreted as a
scaling factor in volume.
According to Stein~\cite{stein-simplex}, the volume of a general $n$-simplex with vertices
$v_0$, $v_1$, \dots, $v_n$ is
\[
   \frac{1}{n!}\left|\det \begin{pmatrix} v_1 - v_0 & v_2 - v_0 &
       \dots & v_n - v_0 \end{pmatrix} \right|.
\]

A \lqindex{polytope} is the convex hull of a \emph{finite} set of
points, not necessarily affinely independent. \Fref{fig:(non)-polytopes} depicts some examples and non-examples of polytopes.
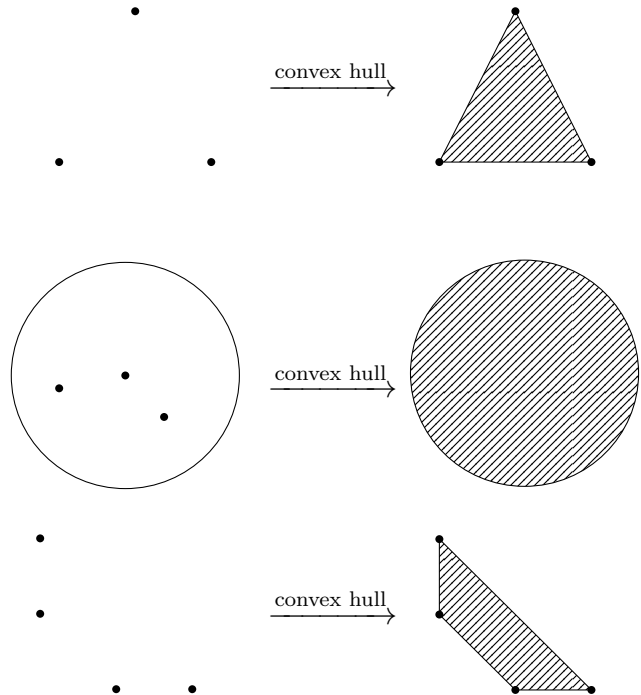
\begin{figure}
\centering
\begin{tikzpicture}[line cap=round,line join=round,>=triangle 45,x=1.0cm,y=1.0cm]
\clip(-3.98,-5.38) rectangle (6.48,4.22);
\fill[fill=black,pattern=north east lines] (2,2) -- (3,4) -- (4,2) -- cycle;
\fill[fill=black,pattern=north east lines] (2,-4) -- (2,-3) -- (4,-5) -- (3,-5) -- cycle;
\draw (2,2)-- (3,4);
\draw (3,4)-- (4,2);
\draw (4,2)-- (2,2);
\draw(-2.13,-0.83) circle (1.5cm);
\draw [fill=black,pattern=north east lines] (3.12,-0.8) circle (1.5cm);
\draw (-0.4,3.5) node[anchor=north west] {$\xrightarrow{\text{convex hull}}$};
\draw (-0.4,-0.5) node[anchor=north west] {$\xrightarrow{\text{convex hull}}$};
\draw (-0.4,-3.5) node[anchor=north west] {$\xrightarrow{\text{convex hull}}$};
\draw (2,-4)-- (2,-3);
\draw (2,-3)-- (4,-5);
\draw (4,-5)-- (3,-5);
\draw (3,-5)-- (2,-4);
\begin{scriptsize}
\fill [color=black] (-3,2) circle (1.5pt);
\fill [color=black] (-2,4) circle (1.5pt);
\fill [color=black] (-1,2) circle (1.5pt);
\fill [color=black] (2,2) circle (1.5pt);
\fill [color=black] (3,4) circle (1.5pt);
\fill [color=black] (4,2) circle (1.5pt);
\fill [color=black] (-3,-1) circle (1.5pt);
\fill [color=black] (-1.62,-1.38) circle (1.5pt);
\fill [color=black] (-2.13,-0.83) circle (1.5pt);
\fill [color=black] (-3.25,-3.99) circle (1.5pt);
\fill [color=black] (-2.25,-4.99) circle (1.5pt);
\fill [color=black] (-1.25,-4.99) circle (1.5pt);
\fill [color=black] (-3.25,-2.99) circle (1.5pt);
\fill [color=black] (2,-4) circle (1.5pt);
\fill [color=black] (3,-5) circle (1.5pt);
\fill [color=black] (4,-5) circle (1.5pt);
\fill [color=black] (2,-3) circle (1.5pt);
\end{scriptsize}
\end{tikzpicture}
  \captionsetup{width=0.7\linewidth}
  \caption{Subsets of the plane and their convex hulls. The disc is not a polytope, the other two convex hulls are polytopes.}
  \label{fig:(non)-polytopes}
\end{figure}

The \lqindex{Minkowski (or vector) sum} of two sets $S$ and $T$ is the set $S + T =
\{ s + t : s \in S, t \in T \}$ of sums of their elements.  The Minkowski sum is a well-defined binary operation on
the space of convex objects as well as the space of polytopes.
Let $S$ and $T$ be two convex
sets.   The cartesian product $S \times T$ is again convex and the map
$(s, t) \mapsto s + t$ is linear so in particular it preserves convex
combinations.
  Assume furthermore that $S$ and $T$ are the convex
hulls of finite sets of points $\{s_1, \dots, s_p\}$ and $\{t_1,
\dots, t_q \}$. An arbitrary point $s + t = \sum_1^p \lambda_i s_i  +
\sum_1^q \mu_j s_j$ is the convex combination $\sum_{i,j}
\lambda_i\mu_j (s_i + t_j)$ so $S + T$ is the convex hull of the
finite set $\{s_1, \dots, s_p\} + \{t_1, \dots, t_q \}$ and hence a polytope.

A different viewpoint defines a polytope as the bounded intersection
of a finite number of halfspaces.  The equivalence between these two
viewpoints is a fundamental result in polytope theory, see
Ziegler~\cite[Theorem~1.1]{ziegler}. Obtaining a vertex description
from a halfspaces description and vice-versa is a hard problem in
general. For the specific polytopes occurring in this thesis both
descriptions are at hand.

A hyperplane $H_{\alpha, c} = \{ x \in V : \langle \alpha, x \rangle = c \}
\subset V$ is an affine subspace of codimension one with normal vector $\alpha$. The
closed halfspaces $H_{\alpha, c}^- = \{ x \in V : \langle \alpha, x
\rangle \leq c \}$ and $H_{\alpha, c}^+ = \{x \in V : \langle
\alpha, x \rangle \geq c \}$ contain all the points to one side of $H_{\alpha, c}$ in addition to the hyperplane itself.

A hyperplane $H$ \emph{supports}\index{hyperplane!supporting} a convex
set $S$ at the point $v$ if $H$ touches $S$ at the point $v$ and $S$
lies on one side of $H$, that is,
 $v \in
H \cap S$ and
either $S \subset H^-$ or $S \subset H^+$.   It is allowed for $S$ to lie
within $H$, the line segment $\{ (x, y) \mid x \geq 0, y \geq 0, x + y
= 1 \}$ is supported by the hyperplane $x + y = 1$ at any of its points.

 If $H_{\alpha, c}$
supports $S$ and $S \subset H_{\alpha, c}^-$ then $H_{\alpha,c}^-$ is a
supporting halfspace of $S$ with outward normal vector $\alpha$.
If the convex set $S$ is also
closed, then for any $x$ outside of $S$ there is a unique point $y \in
S$ that is closest to $x$.  The hyperplane through $y$ that is
perpendicular to the line segment between $x$ and $y$ supports $S$ at
$y$.  This construction, depicted in \Fref{fig:supporting-halfspace}, shows that for each point $x$ outside of $S$
there is a halfspace $H^-$ that contains $S$ but not $x$, and thus
every nonempty closed convex set is the intersection of its supporting
halfspaces~\cite[Corollary~1.3.5]{schneider}.
\begin{figure}
\centering
\begin{tikzpicture}[line cap=round,line join=round,>=triangle 45,x=1.0cm,y=1.0cm]
\clip(-0.48,0.74) rectangle (4.59,5.69);
\draw [fill=black,pattern=north east lines] (1.26,3.75) circle (1.57cm);
\draw [->] (2.32,2.59) -- (4,1);
\draw [domain=-0.48:4.59] plot(\x,{(--0.21--1.68*\x)/1.59});
\begin{scriptsize}
\fill [color=black] (2.32,2.59) circle (1.5pt);
\draw[color=black] (2.35,2.34) node {$y$};
\draw[color=black] (0.8,5.4) node {$S$};
\fill [color=black] (4,1) circle (1.5pt);
\draw[color=black] (4.2,1) node {$x$};
\draw[color=black] (4.12,4.09) node {$H$};
\end{scriptsize}
\end{tikzpicture}
\captionsetup{width=0.7\textwidth}
   \caption{The supporting hyperplane $H$ separates the closed convex set $S$ from
  any point $x$ outside of $S$.}
  \label{fig:supporting-halfspace}
\end{figure}
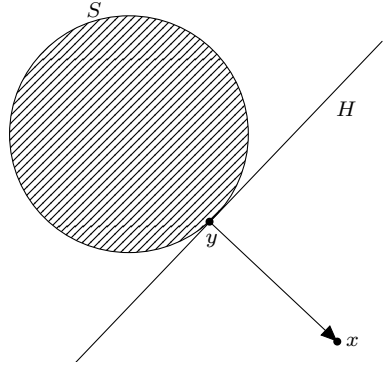
Let $P = H_1^- \cap \dots \cap H_r^-$ be a polytope defined as the
intersection of $r$ halfspaces, where $r$ is minimal.  An intersection
of $P$ with multiple halfplanes $H_i$ yields a subset of $P$ called a
\emph{face}\index{polytope!face}.  A face of dimension $i$ is called
an $i$-face.  Every polytope trivially has itself as a face. Faces
that are strict subsets of the polytope are
\emph{proper}\index{polytope!face!proper} faces.  Special terminology
is used for $0$-faces (\emph{vertices}\index{polytope!vertex}),
$1$-faces (\emph{edges}\index{polytope!edge}) and the proper faces of
largest dimension (\emph{facets}\index{polytope!facet}).  An
$n$-dimensional polytope is \emph{simple}\index{polytope!simple} if
all its vertices are contained in the minimum of $n$ facets.  A
three-dimensional cube is simple, since each vertex is contained in
three facets, but a pyramid with a square base is not simple as the
apex is contained in four facets.

There is a dual way of thinking of the faces of a polytope, for a
functional $\alpha \in V^*$ let $M_P(\alpha) = \max_{v\in P} \langle
\alpha, v \rangle$ denote the maximum value that $\alpha$ attains on
$P$. The \lqindex{maximizer} $F_P(\alpha)$ of $P$ with respect to
$\alpha$ is the subset of $P$ where $\alpha$ attains the maximal value
$M_P(\alpha)$,
\[
   F_P(\alpha) = \left\{ v \in P \mid \langle \alpha, v \rangle =
     \max_{w \in P} \langle \alpha, w \rangle \right\}.
\]
One way to envision the maximizer of $P$ with respect to $\alpha$ is
to picture sliding the halfplane
perpendicular to $\alpha$ along its normal in the positive direction,
see \Fref{fig:maximizer}.  As the hyperplane progresses along $\alpha$ there is a critical point where the intersection with $P$ becomes empty. The last non-empty intersection is
the set $F_P(\alpha)$.

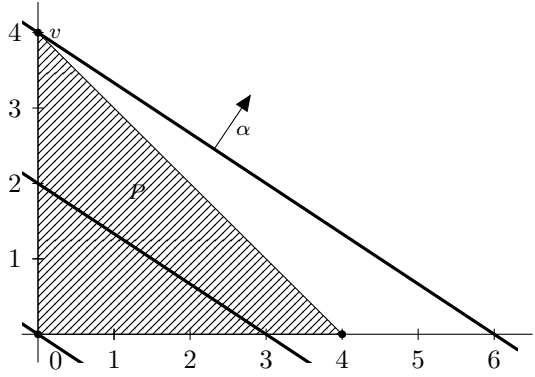
\begin{figure}
 \centering
\begin{tikzpicture}[line cap=round,line join=round,>=triangle 45,x=1.0cm,y=1.0cm]
\draw (-0.21,0) -- (6.5,0);
\foreach \x in {,1,2,3,4,5,6}
\draw[shift={(\x,0)},color=black] (0pt,2pt) -- (0pt,-2pt) node[below] {\footnotesize $\x$};
\draw (0,-0.37) -- (0,4.4);
\foreach \y in {,1,2,3,4}
\draw[shift={(0,\y)},color=black] (2pt,0pt) -- (-2pt,0pt) node[left] {\footnotesize $\y$};
\draw (0pt,-10pt) node[right] {\footnotesize $0$};
\clip(-0.21,-0.37) rectangle (6.31,4.4);
\fill[fill=black,pattern=north east lines] (0,0) -- (0,4) -- (4,0) -- cycle;
\draw [domain=-0.21:6.31] plot(\x,{(--24-4*\x)/6});
\draw [domain=-0.21:6.31] plot(\x,{(--24-4*\x)/6});
\draw [line width=1.2pt,domain=-0.21:6.31] plot(\x,{(--24-4*\x)/6});
\draw [line width=1.2pt,domain=-0.21:6.31] plot(\x,{(--12-4*\x)/6});
\draw [line width=1.2pt,domain=-0.21:6.31] plot(\x,{(-0-4*\x)/6});
\draw [->] (2.32,2.46) -- (2.8,3.18);
\draw (0,0)-- (0,4);
\draw (0,4)-- (4,0);
\draw (4,0)-- (0,0);
\begin{scriptsize}
\fill (0,0) circle (1.5pt);
\fill (4,0) circle (1.5pt);
\fill (0,4) circle (1.5pt);
\draw (0.23,3.99) node {$v$};
\draw (2.7,2.7) node {$\alpha$};
\draw (1.3,1.89) node {$\displaystyle P$};
\end{scriptsize}
\end{tikzpicture}
\captionsetup{width=0.7\textwidth}
  \caption{The face $v$ of the triangle $P$ is the maxmizer $F_P(\alpha)$ of $P$ with respect to $\alpha$.}
  \label{fig:maximizer}
\end{figure}

\begin{lemma}
The faces of a full-dimensional polytope $P$ are exactly the sets of maximizers $\{ v \in P
\mid \langle v, \alpha \rangle = \max_{w \in P} \langle w, \alpha
\rangle \}$ where $\alpha$ ranges over all functionals on the ambient
vector space containing the polytope.
\end{lemma}
\begin{proof}
Let $H_1, \dots, H_r$ be a set of facet-defining hyperplanes of $P$ with
outward normals $n_1, \dots, n_r$.
  The polytope itself maximizes the zero functional.  Facets are the
  maximizers with respect to their facet normals.  Any lower
  dimensional faces are intersections of multiple facets.

  Assume that the intersection $H_1
\cap \dots \cap H_n$ is a face $F$ of $P$.
Then for $\alpha \in \cone \{ n_1, \dots, n_r \} = \{ \sum t_i n_i
\mid t_i \geq 0 \}$ the face $F$ is a subset of the maximizer
$F_P(\alpha)$.  If one of the $t_i$ is zero, the containment is strict,
but if all $t_i$ are positive then any point $x$ outside of any of the
$H_i$ is not an element of the maximizer $F_P(\alpha)$.  Hence the
face $F$ is equal to $F_P(\alpha)$.
\end{proof}
The \emph{normal cone}\index{polytope!normal cone} of a face $F$ is the set of
functionals $\{ \alpha \in V^* \mid F_P(\alpha) = F \}$ that attain
their maximal value precisely on $F$.  Identifying the functionals
$\alpha \in V^*$ with vectors $\alpha \in V$ such that $\alpha(v) =
\langle \alpha, v \rangle$ for every $v \in V$, these normal cones can
be thought of as geometric objects living in the same space as $F$.

The \emph{normal fan}\index{polytope!normal fan} of the polytope $P$ is the collection of the
normal cones of all faces of $P$; it partitions $V^*$ into cones, see \Fref{fig:normal-fan}.
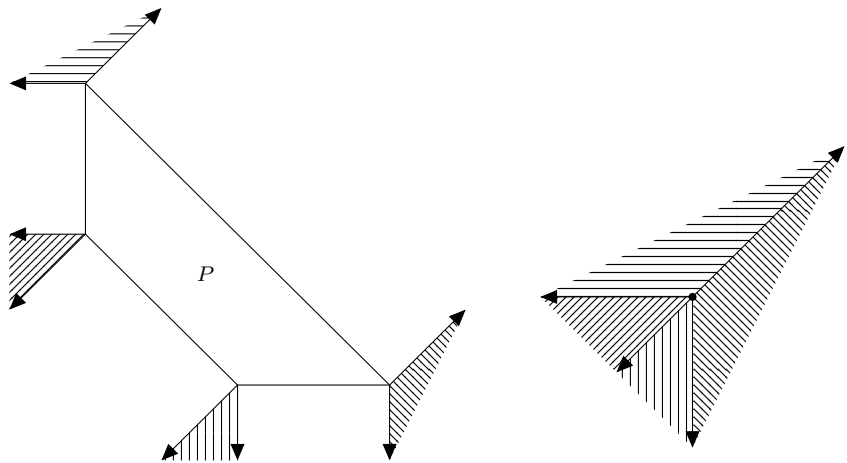
\begin{figure}
\centering
\begin{subfigure}[b]{0.45\textwidth}
\begin{tikzpicture}[line cap=round,line join=round,>=triangle 45,x=1.0cm,y=1.0cm]
    \clip(-1.32,-1.41) rectangle (5.31,5.48);
    \fill[line width=0pt,fill=black,pattern=horizontal lines] (-1,4) -- (0,4) -- (1,5) -- cycle;
    \fill[line width=0pt,fill=black,pattern=north east lines] (0,2) -- (-1,2) -- (-1,1) -- cycle;
    \fill[line width=0pt,fill=black,pattern=vertical lines] (1,-1) -- (2,0) -- (2,-1) -- cycle;
    \fill[line width=0pt,fill=black,pattern=north west lines] (4,0) -- (4,-1) -- (5,1) -- cycle;
    \draw (0,2)-- (0,4);
    \draw (0,4)-- (4,0);
    \draw (4,0)-- (2,0);
    \draw (2,0)-- (0,2);
    \draw [->] (0,2) -- (-1,2);
    \draw [->] (0,2) -- (-1,1);
    \draw [->] (2,0) -- (1,-1);
    \draw [->] (2,0) -- (2,-1);
    \draw [->] (4,0) -- (4,-1);
    \draw [->] (4,0) -- (5,1);
    \draw [->] (0,4) -- (1,5);
    \draw [->] (0,4) -- (-1,4);
    \begin{scriptsize}
        \draw (1.58,1.48) node {$P$};
    \end{scriptsize}
\end{tikzpicture}
\end{subfigure}
\begin{subfigure}[b]{0.45\textwidth}
\begin{tikzpicture}[line cap=round,line join=round,>=triangle 45,x=2.0cm,y=2.0cm]
    \clip(-1.21,-1.29) rectangle (1.66,1.23);
    \fill[line width=0pt,fill=black,pattern=horizontal lines] (0,0) -- (-1,0) -- (1,1) -- cycle;
    \fill[line width=0pt,fill=black,pattern=north west lines] (0,0) -- (0,-1) -- (1,1) -- cycle;
    \fill[line width=0pt,fill=black,pattern=north east lines] (-0.5,-0.5) -- (0,0) -- (-1,0) -- cycle;
    \fill[line width=0pt,fill=black,pattern=vertical lines] (0,-1) -- (0,0) -- (-0.5,-0.5) -- cycle;
    \draw [->] (0,0) -- (-1,0);
    \draw [->] (0,0) -- (0,-1);
    \draw [->] (0,0) -- (1,1);
    \draw [->] (0,0) -- (-0.5,-0.5);
    \begin{scriptsize}
        \fill [color=black] (0,0) circle (1.5pt);
    \end{scriptsize}
\end{tikzpicture}
\end{subfigure}
\captionsetup{width=0.8\textwidth}
   \caption{The normal fan of $P$ partitions the plane into normal cones of all the faces of $P$.}
   \label{fig:normal-fan}
\end{figure}
Scaling a polytope by a positive scalar does not change the normal
fans, as is clear from the equality $\lambda P = \{ \lambda x : Ax \leq
b \} = \{ x : Ax \leq \lambda b\}$.

Let $P$ be an $n$-dimensional polytope.  A \emph{triangulation}\index{polytope!triangulation} $S$ of
$P$ is a decomposition of $P$ into simplices of dimension $n$ with
mutually disjoint interiors, \Fref{fig:polytope-triangulation} shows
triangulations for a square and a triangular prism.
\begin{figure}
  \includegraphics[width=\linewidth]{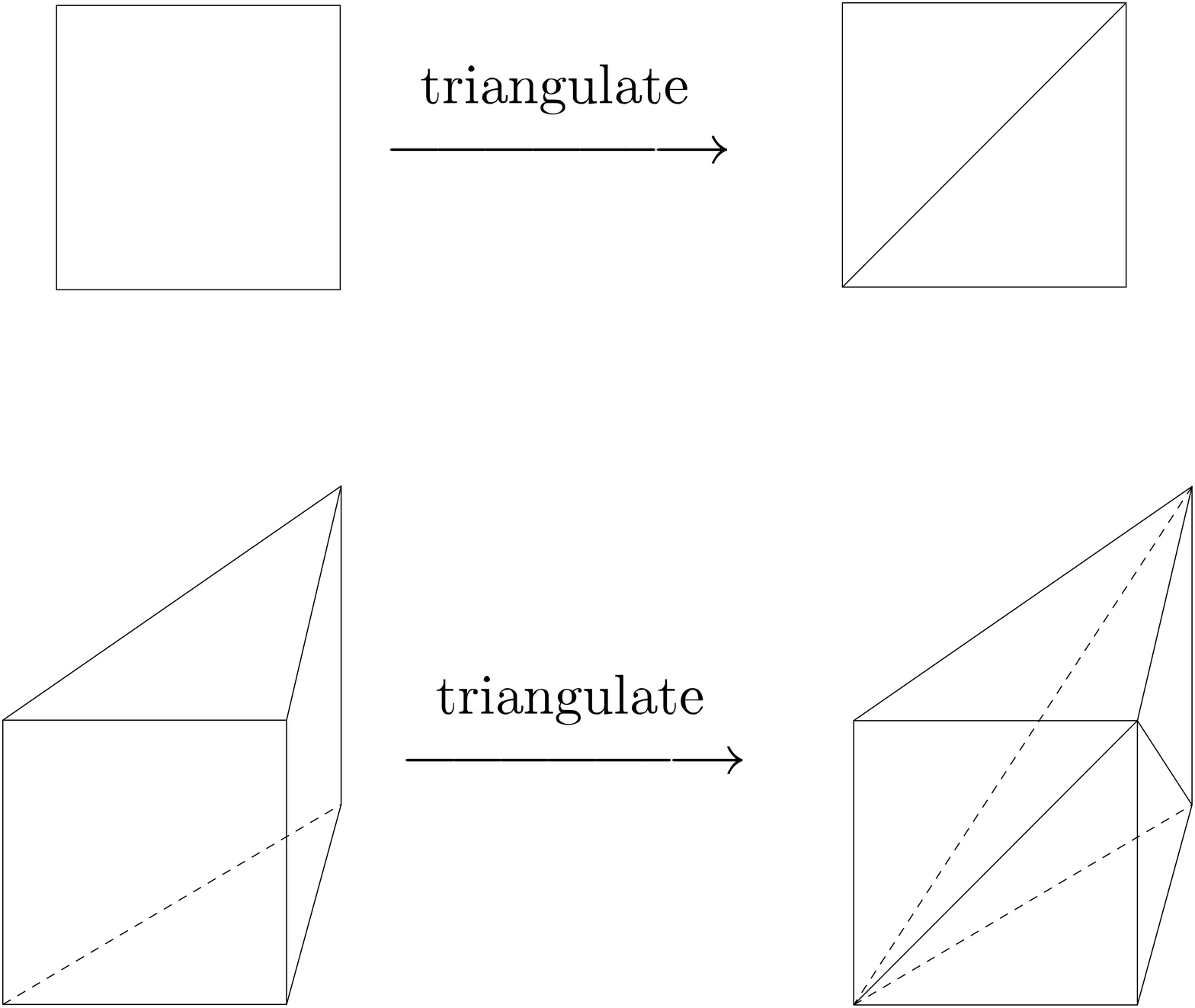}
   \caption{Triangulations of a square and a triangular prism.}
  \label{fig:polytope-triangulation}
\end{figure}

\begin{lemma}
\label{lem:triangulation}
Let $v$ be a vertex of a polytope $P$ and for $F$ a facet of $P$ not
containing $v$ let
$S_F$ be a triangulation of $F$.  Then the union
\[
   \bigcup_F \{ \conv(v, S) \mid S \in S_F \}
\]
of the convex hulls of $v$ with each simplex in a triangulation of a
face of $F$ not containing $v$,
is a triangulation of $P$.
\end{lemma}
\begin{proof}
Let $x \in P$ be distinct from $v$. The ray from $v$ to $x$ exits $P$ in
some face $F$ not containing $v$ and thus intersects some simplex $\sigma \in
S_F$.  The convex hull $\conv(v, \sigma)$ of $v$ and $\sigma$ contains $x$ by
convexity. As $v$ is affinely independent from $\sigma$, the simplex
$\conv(v, \sigma)$ is full-dimensional.

Suppose the ray through $x$ intersects two distinct simplices $\sigma$ and
$\tau$. Then $x$ is contained in $\conv(v, \sigma \cap \tau)$. Since
$\sigma$ and $\tau$ share no interior points, the dimension of the intersection $\sigma
\cap \tau$ is at most $n - 2$.  The dimension of $\conv(v, \sigma \cap
\tau)$ is then at most $n - 1$, so $\conv(v, \sigma)$ and $\conv(v,
\tau)$ have disjoint interiors.
\end{proof}

\Fref{lem:triangulation} suggests an algorithm for triangulating a
polytope.  Starting out with a pair $(P, v)$, recursively triangulate
the facets of $P$ not containing $v$ to obtain the triangulations
$S_F$.  This algorithm is known as the Cohen \& Hickey
algorithm~\cite[Section 3.1]{exact-volumes} and will be used in
\Fref{cor:volume} to calculate the volume of a Minkowski sum.

So far we have pictured polytopes of dimension zero, one, two and
three.  The polytopes playing a main role in this thesis are four-dimensional. One way to visualize four-dimensional polytopes is by
using Schlegel diagrams.  The idea is to project a polytope onto one
of its facets, see \Fref{fig:schlegel-projection}.
\begin{figure}
 \centering
  \includegraphics[width=0.5\linewidth]{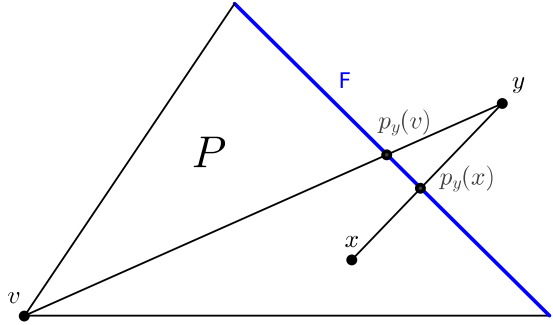}
 \captionsetup{width=0.9\textwidth}
  \caption{A Schlegel diagram of a polytope $P$ is obtained by projecting
           $P$ onto a facet $F$ using the projection $p_y$.}
  \label{fig:schlegel-projection}
\end{figure}

Let $y$ lie beyond a facet $F$ of a polytope
$P$. The projection $p_y(x)$ of $x \in P$ onto $F$ is the intersection
of the line segment between $x$ and $y$ with $F$. The
\emph{Schlegel diagram}\index{polytope!Schlegel diagram} $\mathcal{D}(P, F)$ of $P$ based at the
facet $F$ is the image of all the proper faces of $P$, other than $F$,
under the projection map $p$.  Its usefulness comes from the
fact~\cite[Proposition 5.6]{ziegler} that although $\mathcal{D}(P, F)$
is of smaller dimension than the original polytope, the combinatorial
structures of $P$ and the Schlegel diagram are equivalent.  This
allows one to read off the face structure of a four-dimensional
polytope from a three-dimensional picture.  The Schlegel diagrams in
this thesis are \Fref{fig:schlegel1} and \Fref{fig:schlegel2}.
\\

The concept of mixed volume was introduced by Minkowski in the early
1900s. For our purposes the mixed volume serves only as a computational
tool.
In the literature various definitions of the mixed volume abound.
The following definition as used by Schneider~\cite{schneider},
Bernshtein~\cite{Bernshtein} and Huber and
Sturmfels~\cite{affine-bernshtein} is convenient for root counting.
\begin{definition}[Mixed volume {\cite[Theorem~5.1.6]{schneider}}]
\label{def:mixed-volume}
Let $P_1, \dots, P_n \subset \mathbb{R}^n$ be $n$ polytopes.  Their
\lqindex{mixed volume} $MV(P_1, \dots, P_n)$ is the
coefficient of the monomial $\lambda_1\dots \lambda_n$ appearing in
the expression for the
$n$-dimensional Euclidean volume $\Vol_n(\lambda_1 P_1 + \dots + \lambda_n P_n)$ of the Minkowski sum of the $P_i$ scaled by factors
$\lambda_i$.
\end{definition}

The process of calculating the mixed volume of two rectangles is
depicted in \Fref{fig:mixed-volume}.

\begin{figure}[H]
  \includegraphics[width=\linewidth]{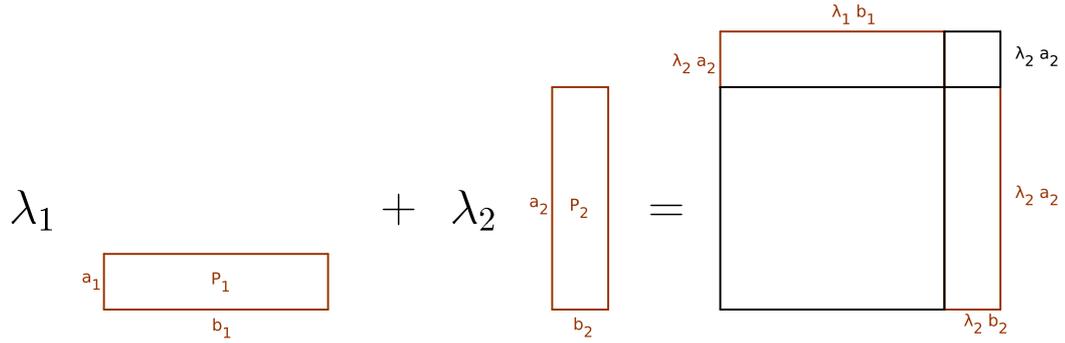}
\captionsetup{width=0.9\textwidth}
  \caption{The mixed volume $MV(P_1, P_2)$ of the polytopes $P_1$ and $P_2$ is the coefficient of $\lambda_1\lambda_2$ in the expression
$\lambda_1^2\Vol(P_1) + \lambda_2^2\Vol_2(P_2) + \lambda_1\lambda_2(a_1b_2 + a_2b_1)$ for the volume of the Minkowski sum $P_1 + P_2$. }
  \label{fig:mixed-volume}
\end{figure}

Before we move on to varieties, the last polytopal concept occuring in the statement of \Bernshtein's Theorem is the concept of a Newton polytope.
Let $f = \sum_{\gamma} c_\gamma x^\gamma \in \mathds{k}[x_1, \dots, x_n]$ be a polynomial. The \emph{Newton polytope}\index{polytope!Newton}
$\newton$ of $f$
is the convex hull of the exponents of the monomials of $f$,
$
  \newton = \conv \{ \gamma \in \mathbb{N}^n \mid c_\gamma \neq 0 \}$.

\begin{example}  The Newton polytopes of $\lambda_{00} + \lambda_{10} x
  + \lambda_{12} xy^2$ and $\mu_{10}x + \mu_{30} x^3 + \mu_{01}y +
  \mu_{03} y^3 + \mu_{11}xy$ are depicted in
  \Fref{fig:newton-polytopes}.  The points $(i, j)$ in the Newton polytopes that are an
  exponent of a monomial $x^iy^j$ are labeled with the corresponding term.
\begin{figure}[H]
\includegraphics[width=0.45\linewidth]{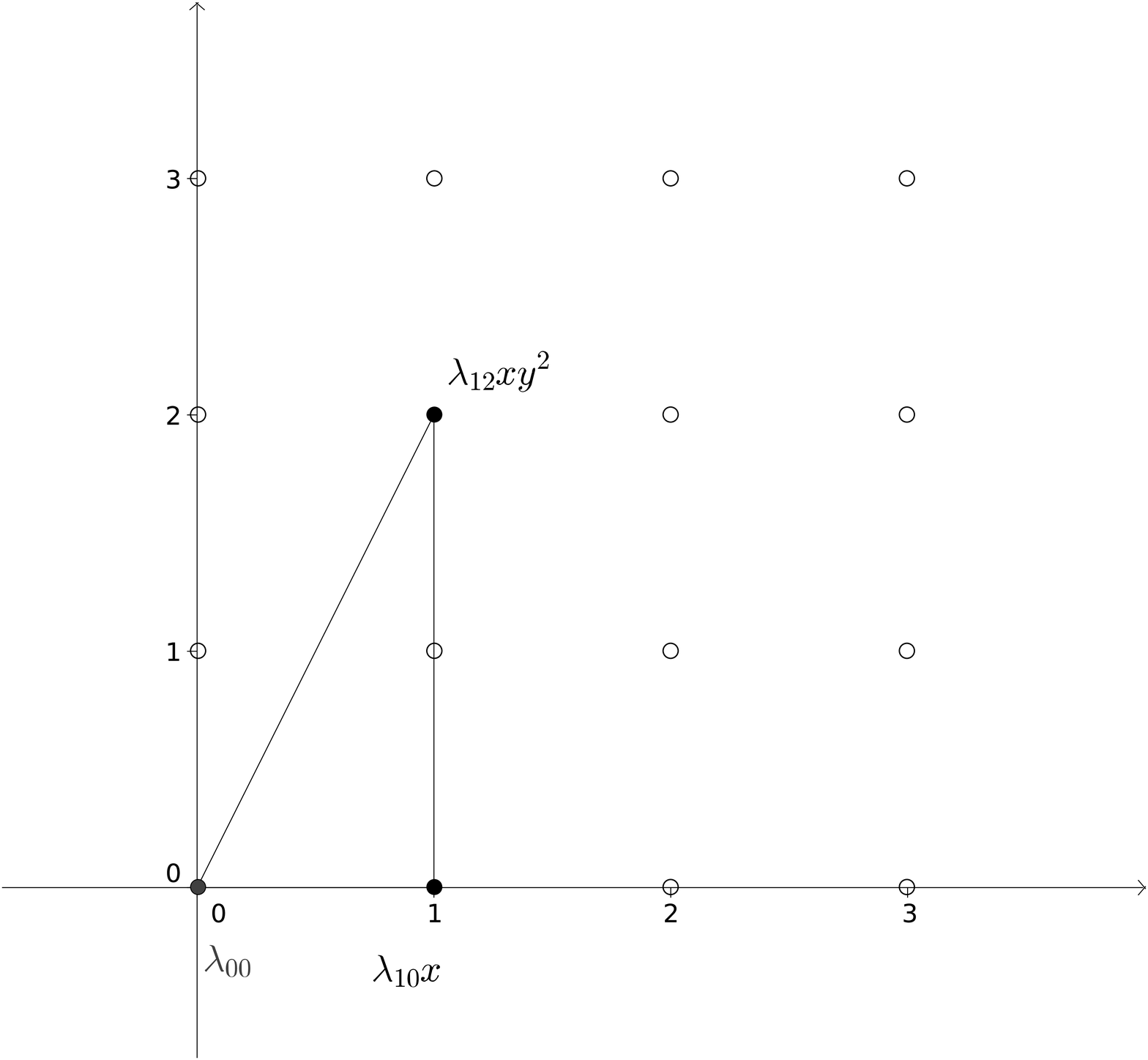}
\includegraphics[width=0.45\linewidth]{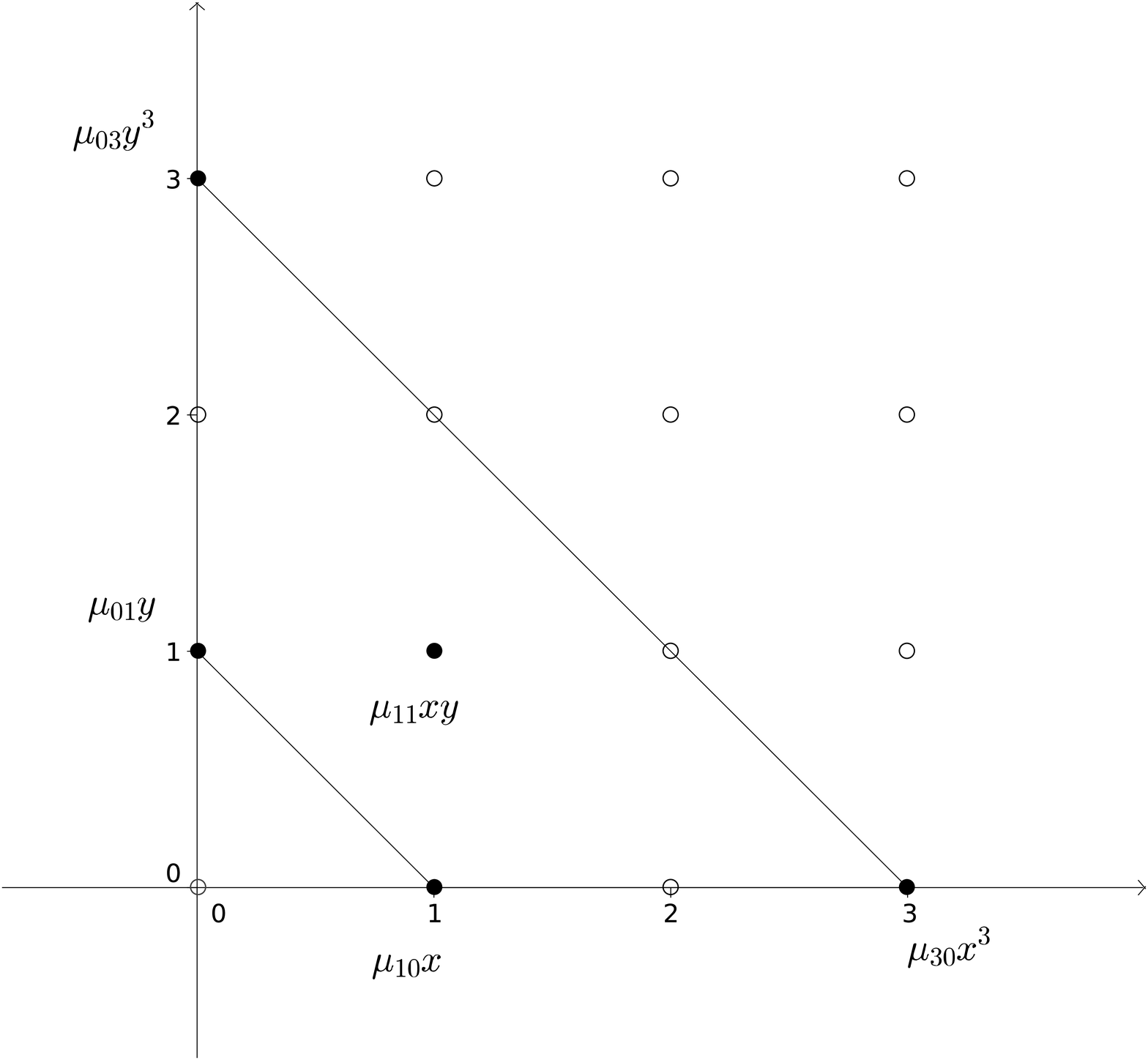}
\captionsetup{width=0.9\textwidth}
\caption{Newton polytopes of the polynomials $\lambda_{00} + \lambda_{10} x
  + \lambda_{12} xy^2$ and $\mu_{10}x + \mu_{30} x^3 + \mu_{01}y +
  \mu_{03} y^3 + \mu_{11}xy$.}
\label{fig:newton-polytopes}
\end{figure}
\end{example}

\subsection{Varieties}
\label{sec:background-varieties}
An algebraic curve and the set of squares inscribed on such a curve
are both examples of varieties.  Varieties are geometric objects we
can describe well by ideals of polynomials vanishing on the variety.
This connection enables the use of algebraic tools from the
Algebra background section to answer questions of
geometry. The Ascending Chain Condition allows us to show that
varieties consist of a finite number of irreducible components;
the difference of varieties defined by ideals $I$ and $J$ corresponds
to the variety defined by the saturation $I : J^\infty$.

Algebraic geometry is pursued over any field, be it finite or
infinite, a subfield of $\mathbb{C}$ or something more exotic. The
concrete fields used in the applications in this thesis are the rationals
$\mathbb{Q}$, the reals $\mathbb{R}$ and the complex numbers
$\mathbb{C}$. All of them are infinite fields, which makes some
reasoning easier.  The complex numbers additionally have the property
that they are \lqindex{algebraically closed}, any nonconstant polynomial with
complex coefficients has a complex root. Many proofs that work for the
complex numbers, such as the Strong Nullstellensatz, only depend on
the fact that the field of complex numbers is algebraically
closed. We shall state such results for an arbitrary algebraically
closed field.
\\

Let $f_1, \dots, f_r \in \mathds{k}[x_1, \dots, x_n]$ be
a set of polynomials.  The set of points \linebreak $(x_1, \dots, x_n)$ $\in \mathds{k}^n$ simultaneously satisfying
the system of equations
\[
   f_1(x_1, \dots, x_n) = 0, \dots, f_r(x_1, \dots, x_n) = 0,
\]
is called the \lqindex{variety} defined by $\{f_1, \dots, f_r\}$,
denoted $\V(f_1, \dots, f_r)$.  Linear and affine subspaces are
familiar examples, both defined by collections of linear polynomials.
Conics, finite sets of points, and graphs $y = f(x_1, \dots, x_n)$ of
polynomials are other examples the reader may have seen before.  Some
varieties and non-varieties are depicted in \Fref{fig:non-varieties}.
An algebraic plane curve is a variety defined by the vanishing of a
single polynomial in two variables.  The line through the origin with
slope one is an algebraic curve defined by the vanishing of the
polynomial $x - y$.  The unit circle is defined by the vanishing of
the polynomial $x^2 + y^2 - 1$.

The smallest variety $V$ that contains a set $S$ is called the
\lqindex{Zariski closure} $\overline{S}$ of $S$.  The Zariski
closure of a point is just the point, as it is already a variety.
The Zariski closure of the integers is all of $\mathbb{R}$, as any
polynomial that vanishes on all integers will vanish on all real numbers.
\begin{figure}
% HYPERBOLA
\begin{subfigure}[b]{0.45\linewidth}
   \includegraphics[width=\textwidth]{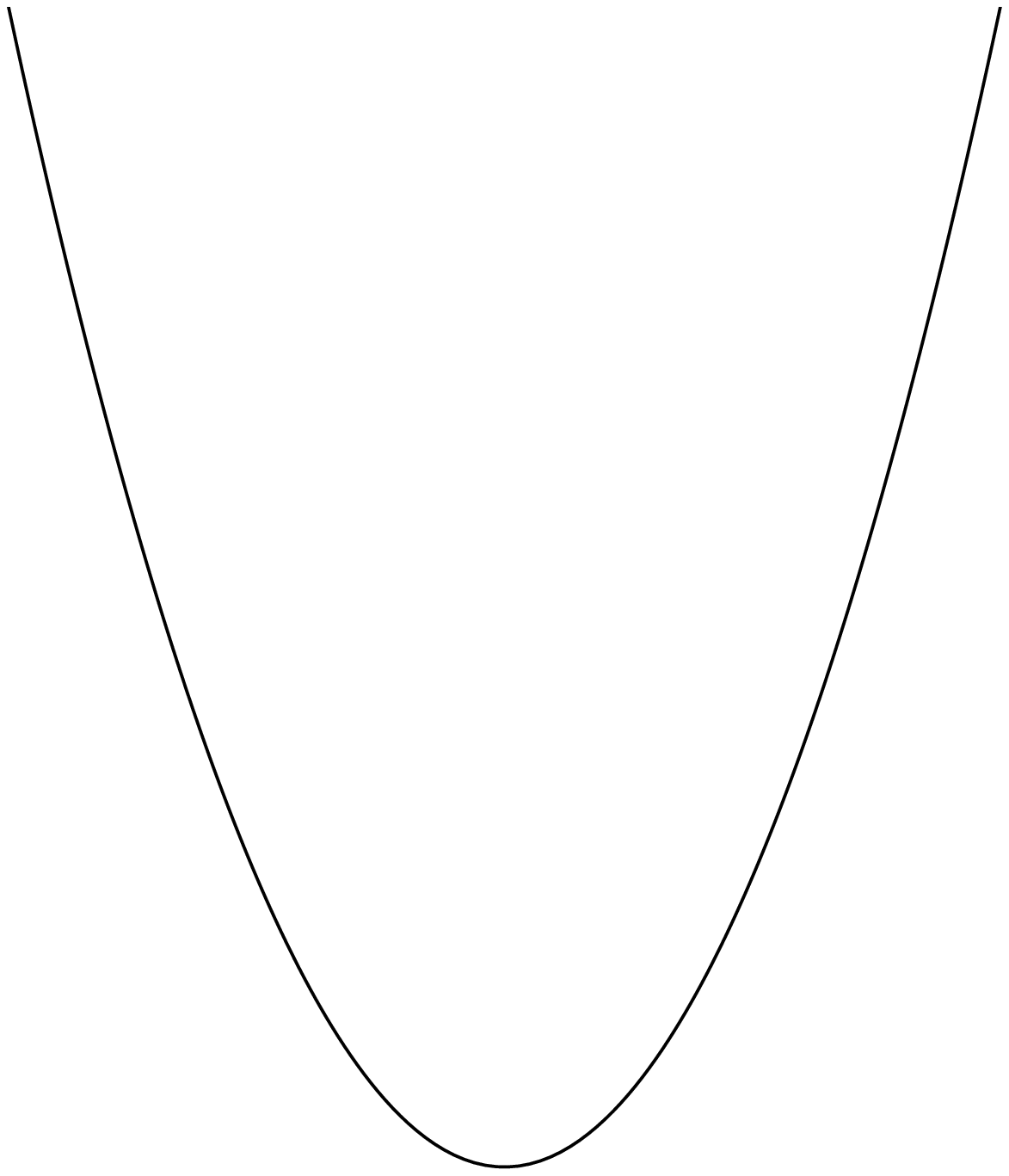}
   \caption{$\V(\frac{y}{4} - x^2)$}
\end{subfigure}
% HALFLINE
\begin{subfigure}[b]{0.45\linewidth}
\includegraphics[width=\textwidth]{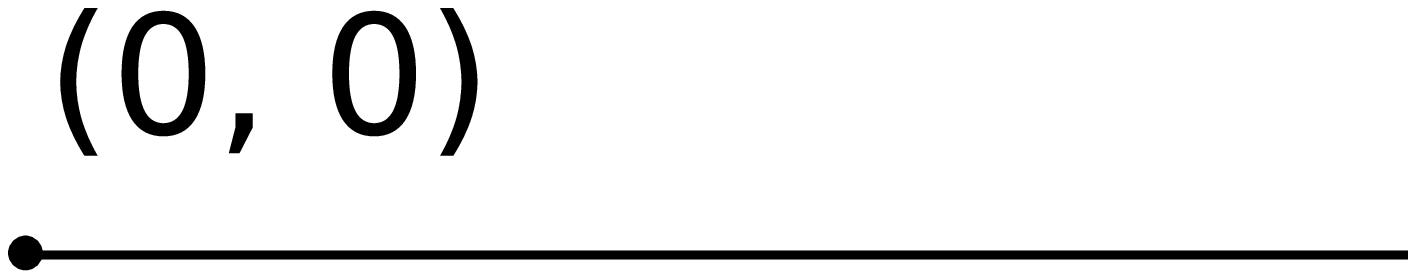}
  \caption{The positive half-line}
\end{subfigure}
% VARIETY
\begin{subfigure}[b]{0.45\linewidth}
\includegraphics[width=\textwidth]{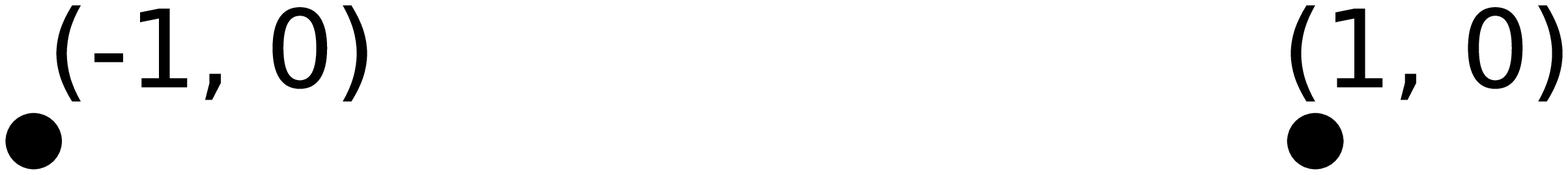}
 \caption{$\V(y, x^2 - 1)$}
\end{subfigure}
% SQUARE
\begin{subfigure}[b]{0.45\linewidth}
\centering
\includegraphics[width=\textwidth]{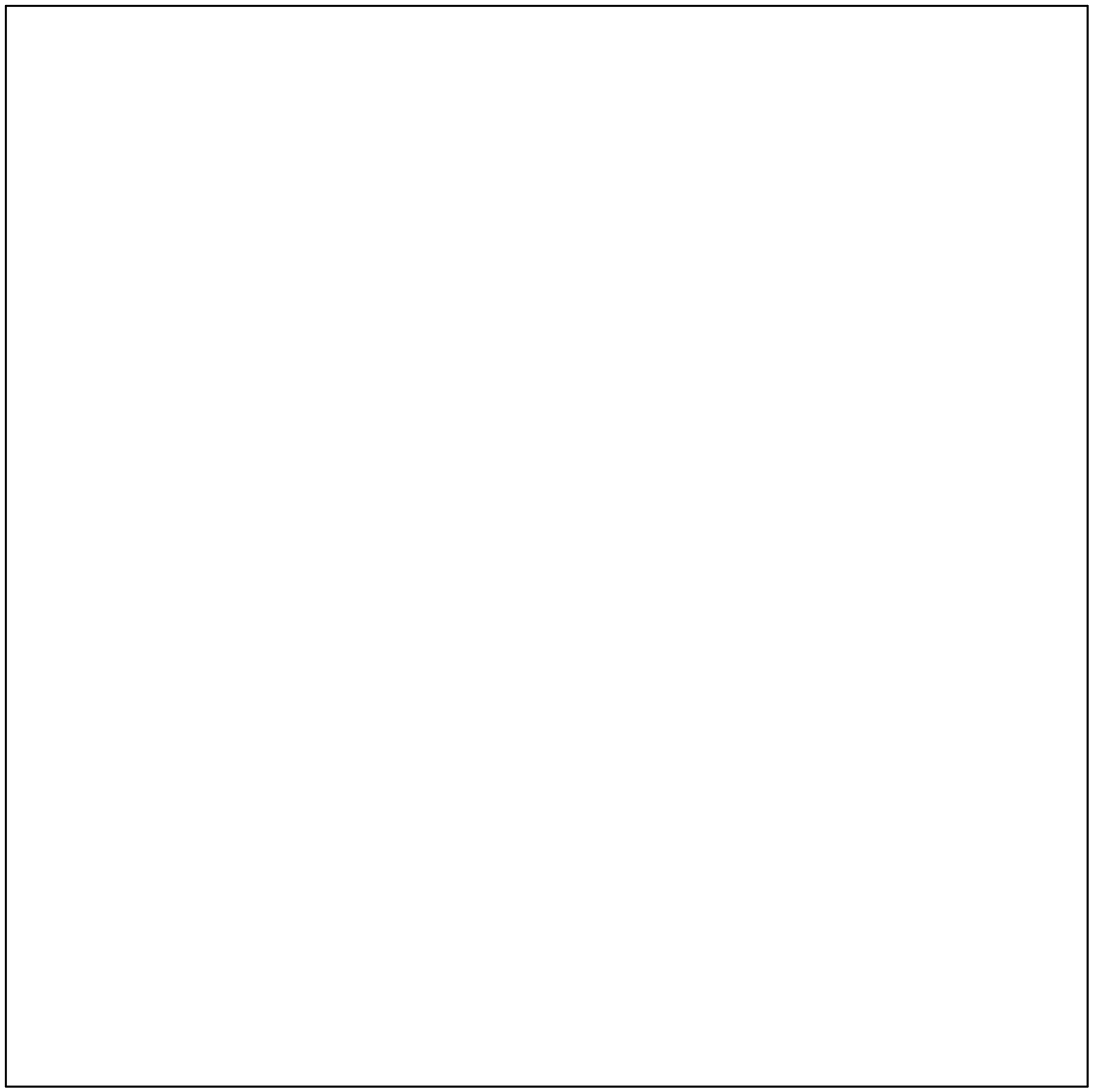}
 \caption{A square}
\end{subfigure}
% VARIETY
\begin{subfigure}[b]{0.45\linewidth}
  \includegraphics[width=\textwidth]{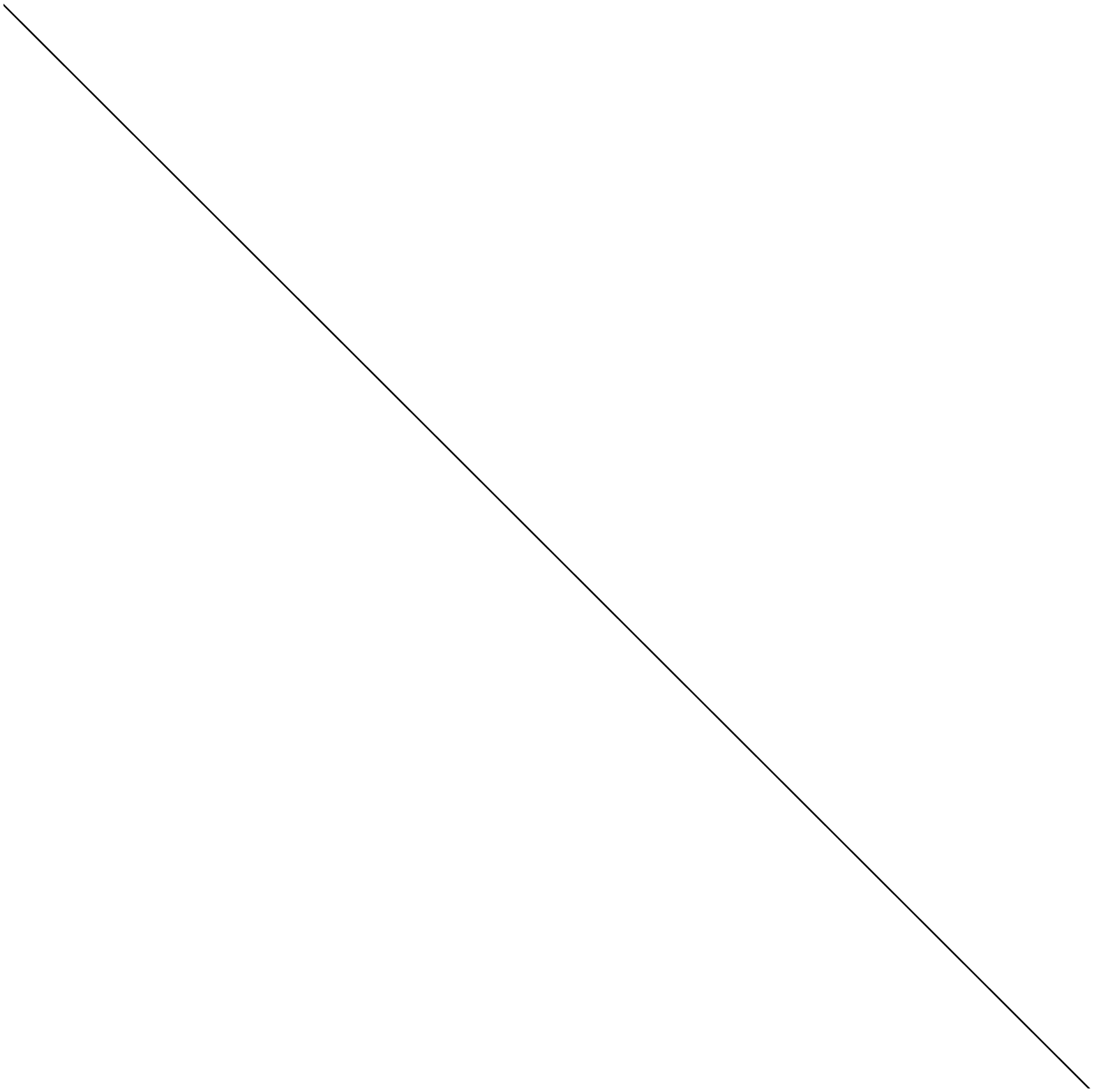}
  \caption{$\V(x + y)$}
\end{subfigure}
% 1/N
\begin{subfigure}[b]{0.45\linewidth}
\includegraphics[width=\textwidth]{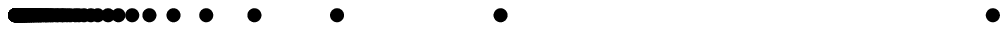}
\caption{The sequence $\left(\frac{1}{n}\right)_{n=1}^\infty$.}
\end{subfigure}
   \caption{Three varieties on the left and three non-varieties on the
     right.}
   \label{fig:non-varieties}
\end{figure}

The polynomials $f_1, \dots, f_r$ have the property that they vanish
on the variety $\V(f_1, \dots, f_r)$ by construction.  The collection
$\I(V)$ of all polynomials vanishing on a variety $V$ is called the
\lqindex{ideal of $V$}. One checks that $\I(V)$ indeed has the
structure of an ideal as defined in \Fref{sec:background-algebra}.
Any $\mathds{k}[x_1, \dots, x_n]$-linear combination of $f_1, \dots,
f_r$ vanishes on $\V(f_1, \dots, f_r)$ so we see that $\langle f_1,
\dots, f_r \rangle \subset \I(\V(f_1, \dots, f_r))$. That the
containment can be strict is illustrated by the ideal $\langle x^2
\rangle \subset \mathds{k}[x]$; the only point where $x^2$ is zero is
the origin, so $\V(x^2) = \{ 0 \}$.  The two monomials of
$\mathds{k}[x]$ not contained in $\langle x^2 \rangle$ are $x$ and
$1$. The constant monomial $1$ does not vanish anywhere, but
 $x$ also vanishes at the origin, so $\I(\{0 \}) = \langle x \rangle$.
There is another relation between the
previous two ideals: $\langle x \rangle$ is the radical of $\langle
x^2 \rangle$.   The \lqindex{radical} $\sqrt{I}$ of an ideal $I$ is the
ideal $\{ f \mid f^m \in I, m \in \N \}$ of all polynomials that occur
in $I$ to some non-negative power.  It is always true that $\sqrt{I}
\subset \I(\V(I))$, but when $\mathds{k}$ is not algebraically closed
equality is not guaranteed. If
$\mathds{k}$ is algebraically closed, it \emph{is} true that the radical of an
ideal $I$ contains all polynomials that vanish on $\V(I)$.
\begin{theorem}[Strong Nullstellensatz {\cite[Theorem~4.2.6]{IVA}}]
\label{thm:Nullstellensatz}
Let $\mathds{k}$ be an algebraically closed field. If $I$ is an ideal
in $\mathds{k}[x_1, \dots, x_n]$ then
\[
  \I(\V(I)) = \sqrt{I}.
\]
\end{theorem}
As a result there is a one-to-one correspondence between radical
ideals and varieties, the maps $\V\colon \text{radical ideals} \to
\text{varieties}$ and $\I\colon \text{varieties} \to \text{radical
  ideals}$ are inclusion-reversing inverses to each other.

The Nullstellensatz is one reason to pass to $\C$ rather than working
over $\R$; when we start out with an ideal $I$ it may be hard to
determine the ideal $\I(\V(I))$ of polynomials vanishing on the
variety $\V(I)$ defined by $I$.  Knowing that all such polynomials lie
in the radical $\sqrt{I}$ can make proofs easier, as happens in the
proof of \Fref{lem:saturation} that $\V(I : J^\infty) =
\overline{\V(I) \setminus \V(J)}$.  Another benefit is that there are
algorithms available to compute the radical of an ideal.
\\

\begin{figure}
 \centering
  \includegraphics[width=0.7\linewidth]{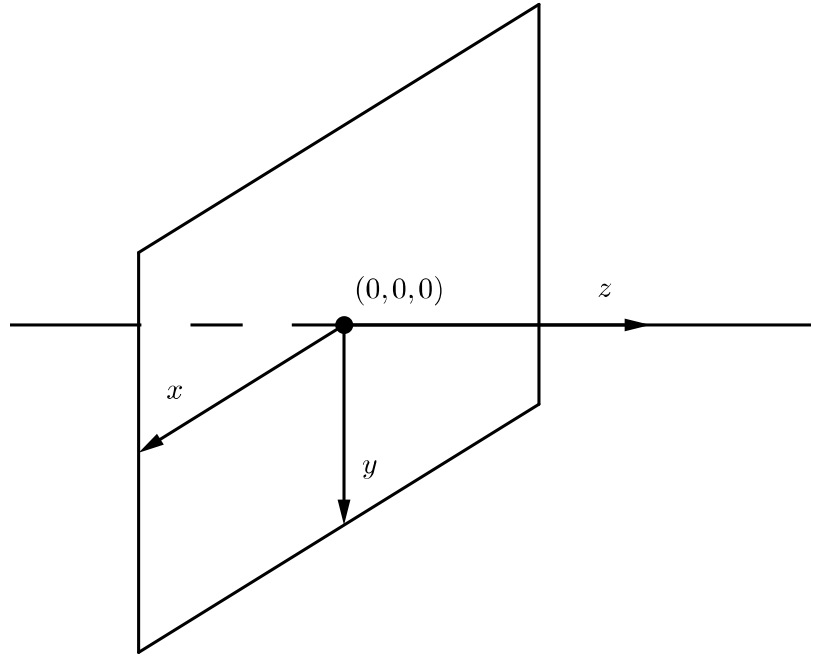}
  \caption{The variety $\V(xz, yz)$ consists of two irreducible
    components.}
  \label{fig:plane-and-axis}
\end{figure}
Some varieties are simpler than others.  Let $f$ and $g$ define two
distinct varieties $\V(f)$ and $\V(g)$. As the product $fg$ vanishes
there where at least one of the polynomials $f$ or $g$ vanish,
the variety $\V(fg)$ is the union of the two subvarieties $\V(f)$ and $\V(g)$.

Whenever a variety $V$ admits
a decomposition $V = W \cup Z$ into two proper subvarieties, $V$ is
said to be reducible.  Otherwise $V$ is \lqindex{irreducible}.  The
reducible variety $\V(xz, yz) \subset \mathds{k}^3$, depicted in
\Fref{fig:plane-and-axis}, is the union of two irreducible components:
the $z$-axis and the $xy$-planes.

As each point is itself a variety, any non-finite variety has an
infinite amount of subvarieties.  However, we can decompose
a variety into a
finite number of irreducible components.  The following proof is a
mixture of several results from Cox~\cite[Section~4.6]{IVA}. It can be
cast in the theory of primary decompositions, see
Eisenbud~\cite[Theorem~3.1a]{view}) for a more comprehensive
treatment.
\begin{lemma}
\label{lem:finite-components}
Any variety $V \subset \mathds{k}^n$ can be written as a finite union
$V = V_1 \cup \dots
\cup V_r$  of irreducible components such that
 $V_i \not\subset V_j$ for any pair $i$ and $j$.
\end{lemma}
\begin{proof}
Suppose $V$ can not be written as a finite union of irreducible
varieties. In particular $V$ is reducible, so there exist distinct
proper subvarieties $Z_1$ and $W_1$ such that $V = Z_1 \cup W_1$.
We can assume that $Z_1$ can not be written as a finite union of
irreducible varieties either, so then $Z_1 = Z_2 \cup W_2$ is
reducible.
Repeating this process we get a
chain $V \supsetneq Z_1 \supsetneq Z_2 \supsetneq \dots$ of strictly decreasing
varieties.   By passing to the ideals of these varieties we get an
increasing chain of ideals $\I(V) \subset \I(Z_1) \subset \I(Z_2)
\subset \dots$
, as all polynomials that vanish on $Z_i$
certainly vanish on $Z_{i+1}$. As $\mathds{k}[x_1, \dots, x_n]$ is Noetherian, these
ideals stabilize, and since $\V(\I(Z_i)) = Z_i$ we observe that the
chain $V \supset Z_1 \supset Z_2 \supset \dots$ stabilizes as well.
This contradicts the assumption that $V$ can not be written as a
finite union of irreducible varieties.

   We conlude that $V$ is a
finite union $V = V_1 \cup
\dots \cup V_r$ of irreducible subvarieties. If $V_i \subset V_j$ we can
drop $V_i$ from the union, proving the statement of the lemma.
\end{proof}

The difference of two varieties in general is no longer a variety.
Consider the case of a line $L$ in the plane and a point $p$ contained in $L$.
Suppose a polynomial $f$ vanishes on $L \setminus \{p\}$, the
restriction of $f$ to $L$ defines a univariate polynomial with an
infinite amount of zeros.  By the fundamental theorem of algebra a
nonzero polynomial of degree $m$ has at most $m$ roots, so the
restriction of $f$ to $L$ must be the zero polynomial.   But then it
also vanishes on $p$, so the smallest variety containing $L \setminus
\{p\}$ is $L$.

There is a relation between the smallest variety that contains the
difference of two varieties defined by ideals $I$ and $J$, and the
variety of the colon ideal $I : J$.
Over any field it is true that $\V(I : J) \supset \overline{\V(I) \setminus
  V(J)}$. Equality holds if in
addition the field is algebraically closed and $I$ is radical
\cite[Theorem~4.4.7]{IVA}.  If $\mathds{k}$ is algebraically closed
but we can not guarantee that $I$ is
radical, the following lemma shows we can instead pass to the saturation $I : J^\infty$.
\begin{lemma}
\label{lem:saturation}
Let $\mathds{k}$ be an algebraically closed field and let $I, J
\subset \mathds{k}[x_1, \dots, x_n]$ be ideals.
Then
\[
   \V(I : J^\infty) = \overline{\V(I) \setminus \V(J) }.
\]
\end{lemma}
\begin{proof}
Let $f \in I : J^\infty$, that is, for every $j \in J$ the product
$fj^k$ is an element of $I$, for some $k \in \N$.   Since for every $x
\in \V(I) \setminus \V(J)$ there is a $j \in J$ that is nonzero at
$x$,  the condition $fj^k \in I$ implies that $f(x) = 0$, as $\V(I)$
is per definition the set of points where \emph{all} elements of $I$ vanish.
Thus every element of $I : J^\infty$ vanishes on $\V(I) \setminus
\V(J)$.  Since $\overline{\V(I) \setminus \V(J)}$ is the smallest
variety containing $\V(I) \setminus \V(J)$, we have shown the inclusion
$\V(I : J^\infty) \supset \overline{\V(I) \setminus \V(J) }$.

For the reverse inclusion, let $f \in \I(\V(I) \setminus \V(J))$. For
any $j \in J$ the product $fj$ vanishes on the entirety of $\V(I)$ as
$j$ vanishes on $\V(J)$ and $f$ vanishes on the complement of $\V(J)$ in
$\V(I)$.    Since
we assumed that $\mathds{k}$ is algebraically closed, it follows that
$fj \in \sqrt{I}$ and thus
$(fj)^k \in I$ for some integer $k$.   If $f^kj^k \in I$ for all $j$ we can
conclude that $f^k \in I : J^\infty$.  We will use the fact that $J$
is finitely generated to argue that there is indeed an integer $k$
such that $f^kj^k \in I$ for all $j \in J$.
\\

Let $j_1, \dots, j_s$ be a finite set of generators for $J$.
 By the reasoning in the previous paragraph,
$(fj_i)^{k_i} \in I$ for some $k_i \in \N$.  Let $k$ be the
minimal integer such that $(fj_i)^k \in I$ for all $i \in
\{1, \dots, s\}$.     Let $j = \sum_{i=1}^s h_i j_i$ be an arbitrary
element of $J$, then
\[
     (fj)^{ks} = \sum_{|\alpha| = ks} g_{\alpha} f^{ks}
     j_1^{\alpha_1} \dots j_s^{\alpha_s},
\]
where the $g_\alpha$ are products of the $h_i$ and multinomial coefficients.
For each term $g_{\alpha} f^{ks}j_1^{\alpha_1} \dots j_s^{\alpha_s}$
at least one of the $\alpha_i \geq k$, otherwise $|\alpha| < ks$.
As $f^{ks} j_1^{\alpha_1} \dots j_s^{\alpha_s}$ is a multiple of
$f^kj_i^{\alpha_i}$, which is an element of $I$ by construction, the
product $(fj)^{ks}$ is a sum of elements of $I$ and thus an element of
$I$ itself.

Thus $f^{ks} \in I : J^\infty$ as $j$ was arbitrary.  We have shown that
every polynomial $f$ that vanishes on $\V(I) \setminus \V(J)$ is present to some
power in $I : J^\infty$, thus the radical $\sqrt{I : J^\infty}$
contains $\I(\V(I) \setminus \V(J))$ and we get the reverse inclusion
$\V(I : J^\infty) \subset \overline{\V(I) \setminus \V(J)}$.
\end{proof}

A formal definition of dimension of a variety requires some work,
see Chapter 9 ``The Dimension of a Variety'' of Cox~\cite{IVA}.
For this thesis our intuition that points, curves and surfaces are respectively of
dimensions zero, one and two will suffice to reason about
dimensionality.  Experimental computations of dimensions will rely on the {\bf dim} command
provided by Macaulay2.

\section{Problem formulation}
\label{sec:formulation}

Toeplitz's conjecture asks whether every Jordan curve inscribes
a square.  This existence question has eluded a complete answer for
over a hundred years; the class of continuous curves contains rather
pathological specimens.

In the algebraic square peg problem we consider algebraic plane curves
rather than Jordan curves; what can we say about the set of squares
inscribed on an algebraic plane curve?  A straight line does not
inscribe any squares, whereas a circle inscribes an uncountable amount
of squares.  In this thesis our aim is to count the number of
inscribed squares that do not come in infinite families, a circle inscribes
zero ``finite'' squares.

With a suitable concept of a
square, the set of inscribed squares has the structure of a variety.
We will see in \Fref{sec:upper-bound} that we can use \Bernshtein's
Theorem to bound the size of the finite part of this variety.  Before
we state how many squares one can maximally inscribe, let us consider
the variety of inscribed squares in some more detail.  The first issue
we should address is settling on a notion of a square that is
compatible with our algebraic worldview. \Fref{fig:square-param} is
the picture to keep in mind.
\\

Let $f \in \mathbb{R}[x, y]$ define an algebraic plane curve
$\V_\mathbb{R}(f) = \{ (x, y) \in \mathbb{R}^2 \mid f(x, y) = 0 \}$.
If we parametrize a square by a center $(a, b)$ and an offset $(c, d)$
to a distinguished corner, then the variety $\V_\mathbb{R}(f(a + c, b
+ d), f(a - c, b - d), f(a + d, b - c), f(a - d, b + d)) \subset
\mathbb{R}^4$ captures all the squares inscribed on $\V(f)$.  We
consider this variety as the real part of a complex variety defined by the same algebraic relations.
These relations motivate our definition of a complex square.
\begin{definition}[Parametrization of a complex square]
\label{def:complex-square}
A $4$-tuple $(a, b, c, d) \in \mathbb{C}^4$ parametrizes a
\lqindex{complex square} with center $(a, b)$ and corners
$(a + c, b + d), (a + d, b - c), (a - c, b - d), (a - d, b
+ c)$, depicted in \Fref{fig:square-param}.  As there are four choices of $(c, d)$ corresponding to
distinguishing a particular corner, there is a four-to-one
correspondence between $4$-tuples and complex squares with distinct corners.

\end{definition}
\begin{figure}[H]
   \begin{center}
\captionsetup{width=0.7\textwidth}
     \includegraphics[width=0.5\linewidth]{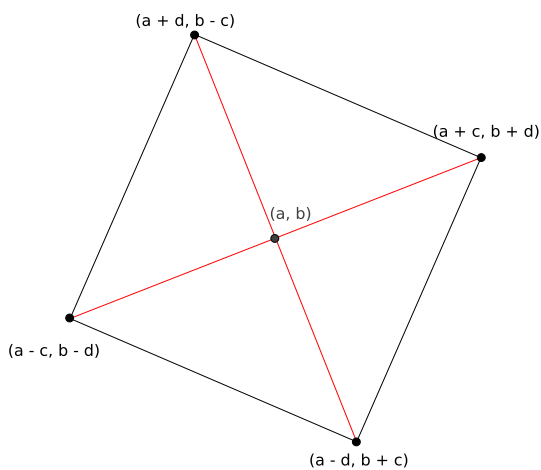}
\caption{Center $(a, b)$ and
  offset $(c, d)$ to a distinguished corner $(a + c, b + d)$
  parametrize a complex square.}
\label{fig:square-param}
\end{center}
\end{figure}
\vspace{-2em}
When constrained to $\mathbb{R}^2 \subset \mathbb{C}^2$ this
definition reduces
to the familiar definition of a square: the diagonals are two
perpendicular line segments of equal length intersecting each other in
their midpoints.
The four corners of a square are distinct as long as $(c, d) \neq (0,
0)$. If $(c, d) = (0, 0)$ the resulting square is \lqindex{degenerate},
it has collapsed to a single point. We combine the definition of a
complex square with a polynomial definining a plane curve to investigate
the set of squares inscribed on that curve.

Let $f \in \mathbb{C}[x, y]$ define an algebraic plane curve $\V(f)
\subset \mathbb{C}^2$ . The \lqindex{corner ideal} $I_f$ of $f$
is the ideal generated by the four polynomials that result from
evaluating $f$ at the four corners of a complex square,
\[
  I_f = \cornerideal \subset \mathbb{C}[a, b, c, d].
\]
The variety $\V(I_f)$ encodes all the squares inscribed on $\V(f)$,
both degenerate and non-degenerate squares.  All of the degenerate
squares are contained in the part of $\V(I_f)$ where the $c$ and $d$
coordinates are both zero.  There is one degenerate square $(a, b, 0, 0)
\in \V(I_f)$ for every point $(a, b) \in \V(f)$.
Thus we identify the degenerate squares $\V(I_f) \cap \{c = d = 0 \}$ with the
original plane curve $\V(f)$.  In the complement $\V(I_f) \setminus \V(f)$ all squares are
non-degenerate.
\\

There might be positive-dimensional components of $\V(I_f)$
other than the one containing $\V(f)$;  consider a plane curve
consisting of two parallel lines. The non-degenerate squares inscribed
on such a curve have two vertices on each component of the curve and are centered on a third line parallel to these two
components.  The sidelengths of the
squares equal the distance between the two parallel lines.

In this thesis we are mainly interested in counting the number of
inscribed squares that lie in the zero-dimensional parts of $\V(I_f)$.
Such squares are \emph{isolated} as they lie in a
neighbourhood that contains no other squares inscribed on $\V(f)$.
Our main result is the
following theorem, proven in the next section.

\begin{reptheorem}{thm:mine}
Let $f \in \mathbb{C}[x, y]$ of degree $m$ define an algebraic plane curve $\V(f)
\subset \mathbb{C}^2$. The number of isolated squares inscribed on
$\V(f)$ is at most $(m^4 - 5m^2 - 4m)/4$.
\end{reptheorem}

\section{An upper bound on the number of isolated squares}
\label{sec:upper-bound}

The variety $\V(I_f)$ of squares inscribed on an algebraic plane curve
$\V(f)$ consists of a finite number of irreducible components and hence
contains a finite number of isolated points by \Fref{lem:finite-components}.
   How do we count
or estimate the number of these isolated points? We will state and use a theorem by
\Bernshtein{} to provide an upper bound on the isolated squares inscribed on an algebraic plane curve.
\\

A classical result from algebraic geometry, called \Bezout{}'s Theorem,
supplies a bound on the cardinality of a variety in terms of the
degrees of the defining polynomials:   If $\V(f_1, \dots, f_s)$ is
finite, then its cardinality is at most the product $\prod \deg f_i$ of
the degrees of the defining polynomials.
The four generators of
$I_f = \cornerideal$ all have the same degree as $f$, say $m$. Ignoring for
a moment the technicality that $\V(I_f)$ is not finite, from \Bezout{} we
would expect that $\V(I_f)$ contains at most $m^4$ points.

\Bezout{}'s Theorem is best
stated in the context of projective space, and considering intersection
multiplicities, see Cox~\cite[Section~8.7]{IVA}.  Apart from
being a very useful theoretical tool, \Bezout's bound acts as a
baseline against which we can judge other root counting methods.

A more refined estimate than \Bezout{}'s bound makes use of more structure of the
polynomials defining a variety than just their degrees.
\Bernshtein{} in his paper ``The number of roots of a system
of equations''~\cite{Bernshtein}, and Kushnirenko and
Khovanskii in related papers, developed theorems to count the number of isolated roots of a polynomial system by
exploiting the sparsity structure of
the monomials appearing in the defining polynomials.  In deference to all
three mathematicians, the resulting bound is often called the
BKK-bound.
\begin{theorem}[\Bernshtein
  {\cite{Bernshtein,UAG,affine-bernshtein,Rojas97toricintersection}}]
\label{thm:Bernshtein}\index{Bernshtein's Theorem}
  Let $f_1, \dots, f_n \in \mathbb{C}[x_1, \dots, x_n]$.  Then the
  number of isolated zeros in $\mathbf{V}(f_1, \dots, f_n) \cap
  (\mathbb{C} \setminus \{0\})^n$ is bounded from above by the mixed volume
  $MV(\newton[f_1], \dots, \newton[f_n])$ of the Newton polytopes of
  the generators $f_i$.
\end{theorem}

A priori \Bernshtein's Theorem has two drawbacks: it provides no
information about positive-dimensional components of $\V(I_f)$, and it
may miss isolated solutions that lie in a coordinate hyperplane, a linear subspace where one or more coordinates are zero. We
relegate the study of the positive-dimensional components to future
work.

We will argue that the interference of the coordinate hyperplanes turns out to not
be a restriction for counting the zero-dimensional part of $\V(I_f)$;
let $f$ be a plane curve and suppose one of the isolated points $p$ of
$\V(I_f)$ lies in a coordinate hyperplane.  Two phenomena can cause
$p$ to lie in a coordinate hyperplane:   the square inscribed by
$\V(f)$ corresponding to $p$ either has
\begin{enumerate}
 \item a center located on the
union of the $x$- and $y$-axes $\V(xy)$, or
\item corners lying on
the translate $\V((x - a)(y - b))$ of the coordinate-axes to its
center.
\end{enumerate}
Note that both phenomena can occur at the same time,
\Fref{fig:square-in-hyperplane} depicts the square $(0, 0, 0, 1)$
inscribed by $\V(xy)$.
\begin{figure}[H]
\centering
 \includegraphics[width=0.4\linewidth]{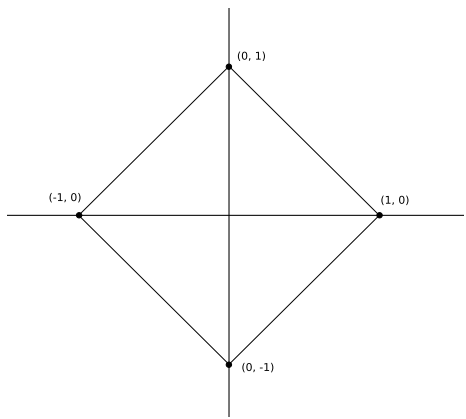}
  \caption{The square $(0, 0, 0, 1)$ lies in three coordinate
    hyperplanes.}
  \label{fig:square-in-hyperplane}
\end{figure}
Both these situations are an artifact of choosing coordinates for the
geometric object that is the curve.   By translating the curve we can
ensure the center of the square corresponding to $p$ no longer lies on
$\V(xy)$.  A rotation suffices to ensure the corners and the center do
not lie on the same translate of $\V(xy)$.

As $\V(I_f)$ has a finite number of irreducible components, there exists
a curve $f'$ obtainable from $f$ by translations and rotations so that
none of the zero-dimensional components of $\V(I_{f'})$ lie in a
coordinate hyperplane.    For the purpose of counting the number of
isolated squares inscribed on a curve we can safely assume
\Bernshtein's Theorem acounts for all of them.
\\

We want to bound the number of isolated squares in $\V(I_f)$ using
\Bernshtein's Theorem; What are the concrete objects appearing in the
expression for the mixed volume
$MV(\newton[f_1], \newton[f_2], \newton[f_3], \newton[f_4])$ for the
algebraic square peg problem?  It is straightforward to calculate the
mixed volume for the polynomials of the form $f(a + c, b + d)$ that
generate $I_f = \cornerideal$, but we show in \Fref{sec:us} that these
generators do not provide a useful BKK bound in general.

We pursue a five step program
to obtain the bound $(m^4 - 5m^2 + 4m)/4$ on the number of isolated
squares inscribed on an algebraic plane curve of degree $m$.
The first step is a better choice of generators $g_i$ of $I_f$ in
\Fref{sec:ideal-rewrite}. In \Fref{sec:monomial-presence} we will see
that this choice will allow for more control on the monomials present
in the generators.  That control translates into smaller Newton
polytopes in the third step discussed in
\Fref{sec:newton-polytope-shapes}.  The Minkowski sum of these smaller
Newton polytopes is described in \Fref{sec:minkowski-sum-shape}. In
the fifth and final step of our program we calculate the volume of the
Minkowski sum $\sum \lambda_i \newton[g_i]$ and extract the mixed volume of the $\newton[g_i]$.

The fact that an algebraic plane curve of degree $m$ inscribes at most
$(m^4 - 5m^2 + 4m)/4$ isolated squares is then an immediate
consequence of invoking \Bernshtein's Theorem, \Fref{thm:Bernshtein},
with the data $MV(\newton[g_1], \newton[g_2], \newton[g_3],
\newton[g_4])$ as calculated by the five step program.

\subsection{The effect of naive generators}
\label{sec:us}

Let $f = \sum c_{i,j} x^i y^j$ of degree $m$ define a plane curve.  We
saw that an application of \Bezout's Theorem to $I_f = \cornerideal$
only tells us that the finite part of $\V(I_f)$ is at most of size
$m^4$.  An application of \Bernshtein's Theorem will bound the number
of isolated squares inscribed on $\V(f)$, up to the squares that lie
in a coordinate hyperplane.  Can we do better than \Bezout's bound by
applying \Bernshtein's Theorem?  Unfortunately, not immediately.

Suppose that the monomials $1$, $x^m$ and $y^m$ appear in $f$
with nonzero coefficients, that is, the Newton polytope of $f$ is as
large as it can be for a curve of degree $m$.  To calculate the BKK bound
we first determine what the Newton polytopes of $f(a
+ c,  b + d)$, $f(a - c, b - d)$, $f(a + d, b - c)$, and $f(a - d, b +
c)$ are by looking at the monomials occuring in them.
\\

Substituting the corner $(a - c, b - d)$ into $f$ and expanding
$f(a -c, b - d)$, the monomial $x^m$ gets mapped to
$\sum_{j=0}^m {m \choose j} a^j (-1)^{m - j} c^{m - j}$, which
establishes that $a^m$ and $c^m$ appear with nonzero coefficients in
$f(a - c, b - d)$.  Similar reasoning applied to $y^m$ guarantees the
presence of the monomials $b^m$ and $d^m$.  As presence of the
monomial $1$ is unaffected by the substitution,
we see that the Newton polytope
$\newton[f(a - c, b - d)]$ contains at least $\conv \{a^m, b^m, c^m,
d^m, 1\} = m \conv\{0, e_1, e_2, e_3, e_4 \} = m\Delta$.  All
monomials of degree at most $m$ are contained in $m\Delta$, so we
conclude that $\newton[f(a - c, b - d)] = m\Delta$. The same argument
goes through for the other Newton polytopes.
Calculating the volume of the Minkowski sum $\sum_1^4 \lambda_i
m\Delta$ we see that
\[
  \Vol_4\left( \sum_1^4 \lambda_i m \Delta\right) = \left(\sum_1^4 \lambda_i\right)^n\!\!\!\!
  \Vol_4(m\Delta),
\]
so the mixed volume of the Newton polytopes is $4!$ times the volume of $m\Delta$.
That is, $4! m^4/4! = m^4$.

The resulting estimate is the same as the one supplied by \Bezout.
To
overcome this problem it is necessary that we pick a set of generators
for $I_f$ whose Newton polytopes are smaller than $m\Delta$. This is
the first step of our five step program, which we undertake in \Fref{sec:ideal-rewrite}.

\subsection{A better choice of generators}
\label{sec:ideal-rewrite}

The issue with the naive generators of $I_f = \cornerideal$ not providing
a BKK bound different from \Bezout's bound is that they contain a lot
of redundant information.  By reducing the redundancy in the
generators of $I_f$ we get a set of generators for which we will be
able to show in the next two sections that their Newton polytopes are
smaller than those of the original generators.

Define polynomials $g_1, g_2, g_3, g_4$ by
\begin{equation}
\label{eq:gi}
\begin{array}{l}
    g_1 = f(a + c, b + d) + f(a - c, b - d) - f(a - d, b + c) - f(a + d, b - c), \\
    g_2 = f(a + c, b + d) - f(a - c, b - d), \\
    g_3 = \phantom{f(a + c, b + d) + f(a - c, b - d) - ~ } f(a - d, b + c) - f(a + d, b - c), \\
    g_4 = \phantom{f(a + c, b + d) + f(a - c, b - d)  - f(a - d, b +
      c) - ~ } f(a + d, b - c).
\end{array}
\end{equation}
As the $g_i$ are linear combinations of the generators of $I_f$,
it is clear that they generate a subideal of $I_f$.  It is easily checked
that the original generators are contained in this subideal as well, so
$\langle g_1, g_2, g_3, g_4 \rangle = \cornerideal$. It may not be
immediately clear that we have gained anything by this different
choice of generators.  Over the course of
\Fref{sec:monomial-presence}, \Fref{sec:newton-polytope-shapes},
\Fref{sec:minkowski-sum-shape} and \Fref{sec:minkowski-sum-volumes} we
will show that $MV(\newton[g_1], \newton[g_2], \newton[g_3],
\newton[g_4]) = m^4 - 5m^2 + 4m$, a definite improvement over the
previous estimate $m^4$.

\subsection{Monomials present in $g_i$}
\label{sec:monomial-presence}

We have shown that the Newton polytopes $\newton[f(a + c, b + d)]$
of the generators of $I_f$ all equal the simplex $m\Delta$ by showing
that they contain the vertices $(0, 0, 0, 0)$ and $me_i$ for $i = 1,
2, 3, 4$.  Since $g_4 = f(a + d, b - c)$ we know that $\newton[g_4] =
m\Delta$.

The construction of the generators $g_1$, $g_2$,
and $g_3$ causes the constant term to disappear, but it is less clear
which monomials of the $g_i$ then will be vertices of the Newton
polytopes.   Which monomials are even present in the generators $g_i$?

Since our five step program has the aim of proving the bound $(m^4 -
5m^2 + 4m)/4$ for all curves of degree $m$, we can assume that the
coefficients of
$f = \sum_{i + j \leq m} C_{i,j}x^i y^j$ are not related in such a way
that they cause cancellation in the $g_i$.
After some
algebraic manipulation we will see that the presence of
$a^{\gamma_1}b^{\gamma_2}c^{\gamma_3}d^{\gamma_4}$ in $g_i$ then only
depends on $i$ and the parity of $\gamma_3 + \gamma_4$, barring the
exceptional case for $g_1$ whenever $\gamma_3 = \gamma_4$ is an even number. The presence
of the monomial $a^{\gamma_1}b^{\gamma_2}c^{\gamma_3}d^{\gamma_4}$ in $g_i$ can be
read off from \Fref{eq:presence} and is summarized in \Fref{tab:monomial-presence}.
An example of the monomials present in a fourth degree curve is displayed in \Fref{sec:fourth-example}.
\begin{table}[H]
\centering
\begin{tabular} {l|ccc}
                 &  $\gamma_3 + \gamma_4$ odd                           &   \multicolumn{2}{c}{$\gamma_3 + \gamma_4$ even}  \\
                 &  &   $\gamma_3 = \gamma_4$, even               &
                 otherwise  \\
\hline
$g_1$            &  absent                     &   absent                                    & present \\
$g_2$ and $g_3$  &  present                    &   absent                                    & absent \\
$g_4$            &  present                    &   present                                   & present \\
\end{tabular}
%\captionsetup{width=0.6\textwidth}
\caption{Presence of monomials $a^{\gamma_1} b^{\gamma_2} c^{\gamma_3}
d^{\gamma_4}$ in $g_i$ depends on the parity of $\gamma_3 + \gamma_4$.}
\label{tab:monomial-presence}
\end{table}

Substituting the expressions for the corners into the variables $x$ and
$y$ transforms monomials $x^iy^j$ of degree $k$ to monomials
$a^{\gamma_1}b^{\gamma_2}c^{\gamma_3}d^{\gamma_4}$ of the same degree
$k$, as seen from the binomial expansion
\[
   (a + c)^i (b + d)^j  = \sum_{p=0}^i {i \choose p} a^p c^{i - p}
   \sum_{q=0}^j {j \choose q} b^q d^{j - q}.
\]
To establish the presence or absence of monomials in $g_i$ of degree
$k$ it thus suffices to consider the $k$-th
homogeneous part of $f$.
We consider $(g_i)_k = h_{i1}f(a + c, b + d)_k + h_{i2}f(a - c, b -
d)_k + h_{i3}f(a - d, b + c)_k
+ h_{i4}f(a + d, b - c)_k$, where $h_{ij} \in \{-1, 0, 1\}$ according to the
choices in \Fref{eq:gi}.  Expanding the definitions
results in the equations
\[
\begin{array}{lr}
   f(a \pm c, b  \pm d)_k = \sum_{j=0}^k C_{k-j, j} (a \pm c)^{k-j} (b
   \pm d)^j,
   \\
   f(a \pm d, b \mp c)_k = \sum_{j=0}^k C_{k-j, j} (a \pm d)^{k-j} (b \mp c)^j.
\end{array}
\]
In addition to expanding the binomial terms $(a \pm d)^{k - j}$ and
$(b \mp c)^j$ in $f(a \pm d, b \mp c)_k$ as before, we keep track of the
coefficients $C_{k -j, j}$ and minus signs.  Gathering monomials we get
\begin{align*}
 f(a \pm d, b \mp c)_k &= \sum_{j=0}^k C_{k - j,j} \sum_{i=0}^{k - j} {k -j \choose i}a^i d^{k -j - i}
  (\pm)^{k - j -i} \sum_{l = 0}^j {j \choose l } b^l c^{j - l}
  (\mp)^{j -l } \\
 &=
\sum_{j=0}^k \sum_{i=0}^{k - j} \sum_{l = 0}^j C_{k - j,j}{k -j
  \choose i} {j \choose l } (\pm)^{k - j -i}
  (\mp)^{j -l } a^i b^l c^{j - l} d^{k -j - i}
 .
\end{align*}
Summing up $h_{i3}f(a - d,  b + c) + h_{i4}f(a + d, b -c )$ we can read off the coefficient of
the monomial with exponent $\gamma = (i, l, j - l, k - j -i)$ as
\[
C_{\gamma_1 + \gamma_4, \gamma_2 + \gamma_3} {\gamma_1
  + \gamma_4 \choose \gamma_1} {\gamma_2 + \gamma_3 \choose
  \gamma_2}(h_{i3}(-1)^{\gamma_4} + h_{i4}(-1)^{\gamma_3}).
\]
The derivation for
$h_{i1}f(a + c, b + d) + h_{i2}f(a - c, b - d)$ is analogous.
The constant term $C_{0,0}$ disappears from $g_i$ as
long as the sum $h_{i1} + h_{i2} + h_{i3} + h_{i4}$ vanishes.  With our choice
of generators this is the case.
For $k > 0$ the degree $k$ monomial $a^{\gamma_1}b^{\gamma_2}c^{\gamma_3}d^{\gamma_4}$ occurs
in $g_i$ in the term
\begin{equation}
\label{eq:presence}
\begin{aligned}
 &\left[ {\gamma_1 + \gamma_3 \choose \gamma_1}{\gamma_2 + \gamma_4 \choose \gamma_2}C_{\gamma_1 + \gamma_3, \gamma_2 + \gamma_4} \left( h_{i1} + h_{i2}(-1)^{\gamma_3 + \gamma_4} \right)\right.
+ \\
 &~ ~ \left.{\gamma_1 + \gamma_4 \choose \gamma_1}{\gamma_2 + \gamma_3 \choose \gamma_2}C_{\gamma_1 + \gamma_4, \gamma_2 + \gamma_3} \left( h_{i3}(-1)^{\gamma_4} + h_{i4}(-1)^{\gamma_3} \right) \right] a^{\gamma_1}b^{\gamma_2}c^{\gamma_3}d^{\gamma_4}.
\end{aligned}
\end{equation}
Here we see that for particular values of the coefficients $C_\alpha$
some extra cancellation may occur that does not happen in the general
case.  However, for a generic choice of coefficients, if $\gamma_3
\neq \gamma_4$ the two summands between brackets in \Fref{eq:presence}
are independent.  Both $g_2$ and $g_3$ have two of the $h_{ij}$ set to
zero, so then the bracketed term is zero if, respectively,
\[
  1 + (-1)(-1)^{\gamma_3 + \gamma_4} = 0, \quad \text{or} \quad
  (-1)^{\gamma_4} + (-1)(-1)^{\gamma_3} = 0.
\]
Multiplying the second equation with $(-1)^{\gamma_3}$ we obtain the
equation $(-1)^{\gamma_3 + \gamma_4} - 1 = 0$.   Thus for both $g_2$
and $g_3$ if $\gamma_3 + \gamma_4$ is even the monomial
$a^{\gamma_1}b^{\gamma_2}c^{\gamma_3}d^{\gamma_4}$ is absent,
otherwise it is present.

A similar argument for $g_1$ shows that
$a^{\gamma_1}b^{\gamma_2}c^{\gamma_3}d^{\gamma_4}$ is absent from
$g_1$ if $\gamma_3 + \gamma_4$ is odd, since $h_{11} = h_{12}$ and
$h_{13} = h_{14}$.   When $\gamma_3 + \gamma_4$ is even there are two
further cases to distinguish;
when $\gamma_3 = \gamma_4$ is an even number, \Fref{eq:presence}
collapses to
\[
   {\gamma_1 + \gamma_3 \choose \gamma_1}{\gamma_2 + \gamma_4 \choose
     \gamma_2} C_{\gamma_1 + \gamma_3, \gamma_2 + \gamma_4} \left( 1 +
     1 - 1 - 1\right) a^{\gamma_1}b^{\gamma_2}c^{\gamma_3}d^{\gamma_4}
   = 0.
\]
Otherwise, either $\gamma_3 = \gamma_4$ is odd and \Fref{eq:presence}
evaluates to
\[
4  {\gamma_1 + \gamma_3 \choose \gamma_1}{\gamma_2 + \gamma_4 \choose
     \gamma_2} C_{\gamma_1 + \gamma_3, \gamma_2 + \gamma_4}
   a^{\gamma_1}b^{\gamma_2}c^{\gamma_3}d^{\gamma_4},
\] or $\gamma_3 \neq \gamma_4$ and the two equations $1 + (-1)^{\gamma_3 + \gamma_4} = 1$ and
$(-1)(-1)^{\gamma_4} + (-1)(-1)^{\gamma_3} $ need to be simultaneously
zero.
\\

In conclusion: monomials of odd $c, d$-degree are present in $g_2$ and
$g_3$ but absent in $g_1$.  Monomials of even $c, d$-degree are absent
in $g_2$ and $g_3$ but present in $g_1$ when the degrees of $c$ and
$d$ are not both even.   These relations are tabulated in \Fref{tab:monomial-presence}.

\subsubsection{Example for a fourth degree curve}
\label{sec:fourth-example}

The presence of monomials in the $g_i$ so far is a little abstract.
Let us look at a somewhat more concrete example by considering a
generic fourth degree curve $f =
{C}_{4,0} x^{4}+{C}_{3,1} x^{3} y+{C}_{{2,2}} x^{2}
      y^{2}+{C}_{{1,3}} x y^{3}+{C}_{{0,4}} y^{4}+{C}_{{3,0}}
      x^{3}+{C}_{{2,1}} x^{2} y+{C}_{{1,2}} x y^{2}+{C}_{{0,3}}
      y^{3}+{C}_{{2,0}} x^{2}+{C}_{{1,1}} x y+{C}_{{0,2}}
      y^{2}+{C}_{{1,0}} x+{C}_{{0,1}} y+{C}_{{0,0}}$.
\begin{comment}
load "preamble.m2"
D = 4
degreeSetup(D)
abstractCurve_D
IJ_D_0
\end{comment}
According to
 Table \ref{tab:monomial-presence}, the monomials in $g_1$ should be all
 even $c, d$-degree  monomials of total degree at most four, excluding
 the monomials $1$ and $c^2d^2$, which is indeed the case:
\begin{align*}
  g_1 &= (-2 {C}_{{2,2}}+12 {C}_{{4,0}}) a^{2} c^{2}+(-6 {C}_{{1,3}}+6
  {C}_{{3,1}}) a b c^{2}+(-12 {C}_{{0,4}}+2 {C}_{{2,2}}) b^{2}c^{2} \\
  &+(-2 {C}_{{0,4}}+2 {C}_{{4,0}}) c^{4}+12 {C}_{{3,1}} a^{2} c d+16
  {C}_{{2,2}} a b c d+12 {C}_{{1,3}} b^{2} c d \\
  & +(2 {C}_{{1,3}}+2{C}_{{3,1}}) c^{3} d +(2 {C}_{{2,2}}-12 {C}_{{4,0}}) a^{2}
  d^{2}+(6 {C}_{{1,3}}-6 {C}_{{3,1}}) a b d^{2} \\ &+(12 {C}_{{0,4}}-2
  {C}_{{2,2}}) b^{2} d^{2} +(2 {C}_{{1,3}}+2 {C}_{{3,1}}) c d^{3} + (2
  {C}_{{0,4}}-2 {C}_{{4,0}}) d^{4} \\
  &+(-2 {C}_{{1,2}}+6 {C}_{{3,0}})
  a c^{2}+(2 {C}_{{2,1}}-6 {C}_{{0,3}}) b c^{2}+8 {C}_{{2,1}} a c d+8
  {C}_{{1,2}} b c d \\ &+(2 {C}_{{1,2}}-6 {C}_{{3,0}}) a d^{2}+(-2
  {C}_{{2,1}}+6 {C}_{{0,3}}) b d^{2}+(2 {C}_{{2,0}}-2 {C}_{{0,2}})
  c^{2} \\ &+4 {C}_{{1,1}} c d+(-2 {C}_{{2,0}}+2 {C}_{{0,2}}) d^{2}.
\end{align*}
Of the list of monomials $\{a^2c^2, abc^2, b^2c^2, c^4, a^2cd, abcd,
  b^2cd, c^3d, a^2d^2,  abd^2, b^2d^2$, $cd^3$, $d^4$, $ac^2$, $bc^2$, $acd$, $bcd$, $ad^2$, $bd^2$, $c^2$,
      $cd$, $d^2\}$ occuring in $g_1$, those with only the variables $c$ and $d$ are depicted in \Fref{fig:squares}.
\begin{figure}[H]
  \begin{subfigure}[b]{0.45\linewidth}
\includegraphics[width=\textwidth]{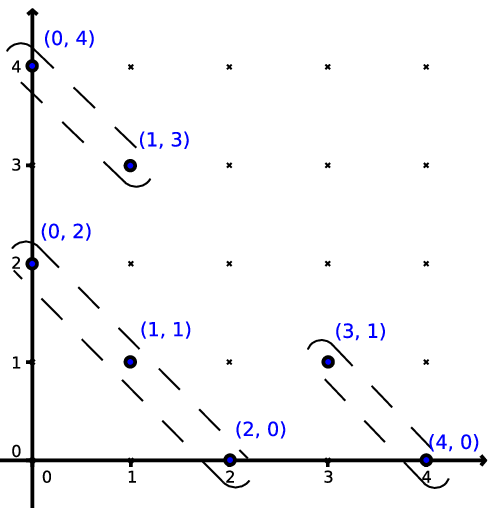}
     \caption{Monomials $c^{\gamma_3}d^{\gamma_4}$ present in $g_1$ are represented by blue circles.}
  \end{subfigure}
 \quad
  \begin{subfigure}[b]{0.45\linewidth}
 \includegraphics[width=\textwidth]{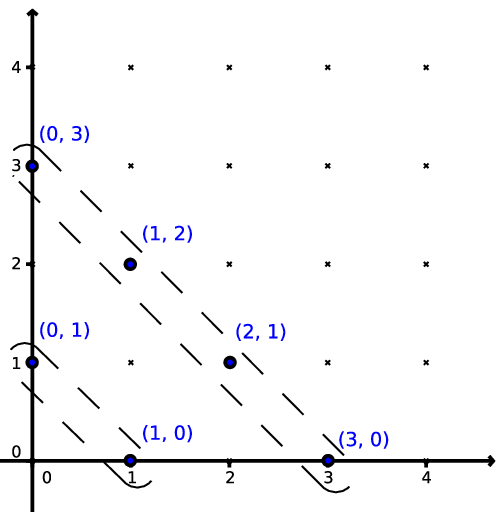}
     \caption{Monomials $c^{\gamma_3}d^{\gamma_4}$ present in $g_2$ are represented by blue circles.}
  \end{subfigure}
  \caption{The parity of $\gamma_3 + \gamma_4$ determines whether
    monomials $c^{\gamma_3}d^{\gamma_4}$ are present in the generators $g_1$ and $g_2$.}
  \label{fig:squares}
\end{figure}

\subsection{Newton polytope shapes}
\label{sec:newton-polytope-shapes}

In the previous two sections we have shown which monomials are present in the $g_i$.
In the third step of our five step program to prove that
the mixed volume $MV(\newton[g_1], \newton[g_2], \newton[g_3], \newton[g_4]) = m^4
-5m^2 + 4m$ we describe the Newton polytopes $\newton[g_i]$.
We already know that $\newton[g_4] = m\Delta$ and $\newton[g_i]
\subset m\Delta$ since the $g_i$ are of degree $m$. We also saw from \Fref{tab:monomial-presence} that
$\newton[g_2] = \newton[g_3]$.

In this section we prove that the Newton polytopes $\newton[g_1]$ and
$\newton[g_2]$ alternate between the two types of simple polytopes
$P_1$ and $P_2$ from \Fref{def:polytopes-pi}, according to the parity
of $m$. This dependence is summarized in
Table \ref{tab:shape-alternation}. Their Schlegel diagrams are depicted in
\Fref{fig:schlegel1} and \Fref{fig:schlegel2}; the vertex descriptions
of $P_1$ and $P_2$ as well as expressions of the vertices as
intersections of facets are given in \Fref{lem:polytope-shape} and
\Fref{lem:second-polytope-shape}.
\\

  The Newton polytopes $\newton[g_i]$
are the convex hulls of the monomials appearing in the $g_i$; the
pertinent information about $g_1$, $g_2$ and $g_3$ is shown in
\Fref{tab:monomial-presence}. Let us rewrite this information in a
form convenient for thinking about polytopes as intersections of halfspaces,
\begin{align*}
  \{ \text{exponents of } g_1 \} &= m\Delta \cap \bigcup_{n=0}^\infty \ \{x_3 + x_4
    = 2n + 2\} \setminus \{ x_3 = x_4 \text{ even} \}, \\
  \{ \text{exponents of } g_2 \} &= m\Delta \cap \bigcup_{n=0}^\infty \ \{x_3 + x_4 = 2n + 1\}.
\end{align*}
The extreme monomials determine the convex hull, so we can express
$\newton[g_1]$ and $\newton[g_2]$ as the following intersections of halfspaces:
\[
\begin{array}{l}
    \newton[g_1] = m\Delta \cap H_{x_3 + x_4 \geq 2} \cap H_{x_3 + x_4 \leq 2n_1 + 2}, \\
    \newton[g_2] = m\Delta \cap H_{x_3 + x_4 \geq 1} \cap H_{x_3 + x_4 \leq 2n_2 + 1},
\end{array}
\]
where $n_1$ and $n_2$ are the largest integers $n_1$ and $n_2$ such
that $2n_1 + 2$ and $2n_2 + 1$ are both smaller than or equal to $m$
. If $m$ is even, then the halfspace $H_{x_3+x_4 \leq 2n_1 + 2}$ is redundant as the hyperplane $H_{x_3 + x_4 = 2n_1 + 2}$ intersects $m\Delta$ in the
facet defined by the hyperplane $H_{\sum x_i = m}$.  When $m$ is odd,
$H_{x_3+x_4 \leq 2n_2 + 1}$ is redundant. These polytopes are central to the rest of this
section, so let us fix some notation.
\begin{definition}
\label{def:polytopes-pi}
\index{polytope!$P_1(m, l)$}
\index{polytope!$P_2(m, l, k)$}
The three types of polytopes $P_0$, $P_1$ and $P_2$
  are obtained from $m\Delta$ by successively adding a facet-defining hyperplane
  parallel to $H_{(0, 0, 1, 1)}$ so that
\[
\begin{array}{l}
P_0 = P_0(m) = m\Delta,  \\
P_1 = P_1(m, l)  = P_0 \cap H_{x_3 + x_4 \geq l}, \\
P_2 = P_2(m, l, k) = P_1(m, l)
  \cap H_{x_3 + x_4 \leq k}.
\end{array}
\]
\end{definition}
The polytopes $P_1$ and $P_2$ are both four-dimensional when
$m \geq 4$ but not for $m \in \{2, 3\}$. Schlegel diagrams for $m=4$ are depicted in \Fref{fig:schlegel1} and
\Fref{fig:schlegel2}.
With the notation from \Fref{def:polytopes-pi} we can summarize the
Newton polytopes of $g_1$ and $g_2$ for even and odd $m$ as
\begin{table}[H]
\[
\begin{array}{lll}
& m = 2n + 2  & \phantom{ugly} m = 2n + 1 \vspace{0.5em}\\
\newton[g_1] & P_1(m, 2)    & \phantom{ugly}  P_2(m, 2, m - 1) \\
\newton[g_2] & P_2(m, 1, m - 1) & \phantom{ugly} P_1(m, 1).  \\
\end{array}
\]
\caption{$\newton[g_1]$ and $\newton[g_2]$ alternate between the
  polytopes $P_1$ and $P_2$}
\label{tab:shape-alternation}
\end{table}

\begin{figure}
\centering
  \includegraphics[width=0.6\linewidth]{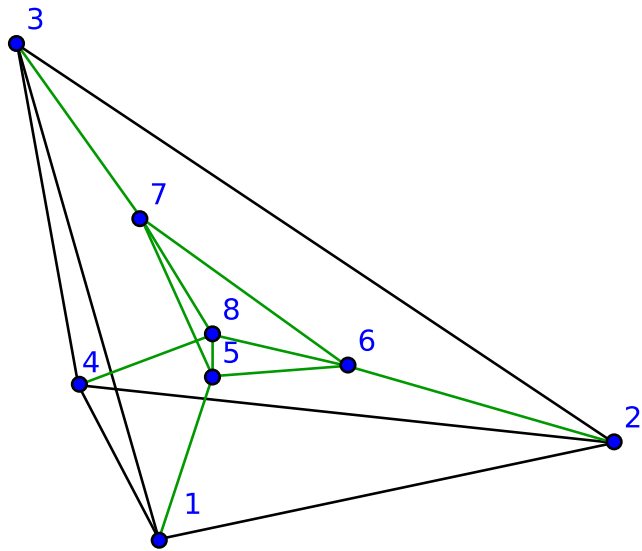}
\caption{Schlegel diagram of $P_1$ projected onto its facet where $x_4 = 0$.}
\label{fig:schlegel1}
\end{figure}
\begin{figure}
\centering
\includegraphics[width=0.5\linewidth]{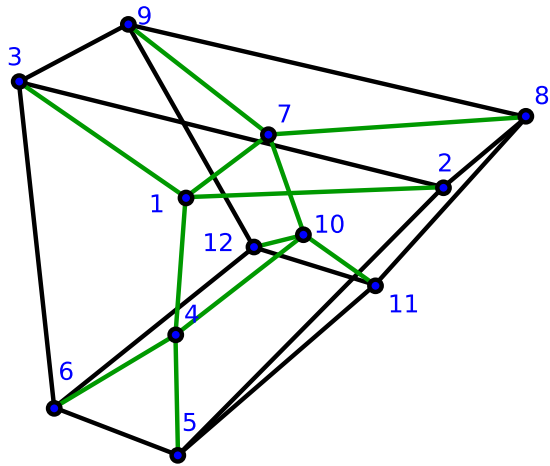}
\caption{Schlegel diagram of $P_2$ projected onto its facet where $\sum x_i= m$.}
\label{fig:schlegel2}
\end{figure}
\index{polytope!Schlegel diagram}
The combinatorial structure of the polytopes $P_1$ and $P_2$, that is, which
vertices are included in which faces, can be read off from the
Schlegel diagrams. For those unconvinced that the Schlegel diagrams
are correct,
the next two lemmas establish vertex descriptions and the
facet-vertex incidences of $P_1(m, l)$ and $P_2(m, l, k)$, without the
visual aid.
\begin{lemma}
\label{lem:polytope-shape}
\index{polytope!$P_1(m, l)$}
Let $m \geq 4$ and $0 < l < m$. Then
$P_1 = P_1(m, l)$, as defined in \Fref{def:polytopes-pi}, is a simple polytope
with eight labeled vertices given by the columns of the matrix
\[
\bordermatrix{
 &  1      & 2        & 3       & 4    & 5    & 6        & 7       & 8 \cr
  & 0      & m - l  & 0       & 0    & 0    & m - l  & 0       & 0 \cr
  & 0      & 0        & m - l & 0    & 0    & 0        & m - l & 0 \cr
  & l       & l         & l        & m   & 0    & 0        & 0       & 0 \cr
  & 0      & 0        & 0       & 0    & l     & l         & l        & m \cr
 }.
 \]
The vertices are expressed as intersections of hyperplanes in the
following way,
\begin{align}
\begin{array}{l}
  H_{-e_1,0} \cap H_{-e_2,0} \cap H_{-e_i,0} \cap H_{\sum e_k,m} = \{
  me_j \},  % & i, j \in \{3, 4\}, i \neq j\\
 \\
  H_{-e_1,0} \cap H_{-e_2,0} \cap H_{-e_i,0} \cap H_{-e_3 - e_4,l} =
  \{l e_j \}, %& i, j \in \{3,  4\}, i \neq j\\
\\
  H_{-e_{1 + j_1},0} \cap H_{-e_{3 + j_2},0} \cap H_{\sum e_i,m} \cap H_{-e_3 - e_4,l} = \{
  (m -l) e_{2 - j_1} +l e_{4 - j_2} \}, % &  j_1, j_2 \in \{0, 1\}.
\end{array}
\label{eq:p1-combinatorics}
\end{align}
where $i, j \in \{3, 4\}$, $i \neq j$ and $j_1, j_2 \in \{0, 1\}$.
\end{lemma}
\begin{proof}
  The polytope $P_1(m, l)$ has six facet-defining hyperplanes.  There
  are ${6 \choose 4}$ ways to form intersections of four of these
  hyperplanes.  Due to the constraint $x_3 + x_4 \geq l$ the
  intersection $H_{-e_3,0} \cap H_{-e_4,0}$ does not contain any part of
  $P_1$.  The intersection $H_{-e_1,0} \cap H_{-e_2,0} \cap H_{-e_3 - e_4,l} \cap H_{\sum e_i,m}$ is
  empty due to conflicting constraints. Thus any intersection of five
  hyperplanes is either empty or lies outside $P_1$, as a five-fold intersection of the hyperplanes defining $P_1$ involves at least one of these two intersections.  Hence any vertex
  of $P_1$ is contained in at most four facets.

  This leaves $2{4 \choose 3} = 8$ combinations of intersecting four
  hyperplanes to check, each involving exactly one of $H_{-e_3,0}$ or $H_{-e_4,0}$.  These eight intersections are listed above and result in eight
  distinct vertices, each of which is contained in precisely four
  facets.
\end{proof}
We obtain $P_2$ from $P_1$ by intersecting it with the halfspace
$H_{x_3 + x_4 \leq k}$. The facet of $P_2$ defined by this halfspace
 is parallel to the hyperplane
$H_{x_3 + x_4 \geq l}$ that cuts out $P_1$ from $P_0$, and thus the
derivation of $P_2$ follows the same kind of  reasoning as \Fref{lem:polytope-shape}.
\begin{lemma}
 \label{lem:second-polytope-shape}
\index{polytope!$P_2(m, l, k)$}
  Let $0 < l < k < m$ and $m \geq 4$. Then $P_2 = P_2(m, l, k)$ is a
  simple polytope with twelve labeled vertices given by the colums of the matrix
\[
\scriptscriptstyle
\hspace{-0.5em}
\bordermatrix{
 &  1   & 2       & 3         & 4  & 5         & 6        & 7  & 8 & 9              & 10           &  11      &  12 \cr
  & 0   & m - l & 0         & 0  & m -  k &0         & 0 & m - l  & 0               & 0             &  m - k & 0 \cr
  & 0   & 0       & m - l   & 0  & 0         & m - k & 0 & 0        & m - l & 0             &  0        & m - k\cr
  & l    & l        & l          & k  & k         & k        & 0 & 0        & 0               & 0                    &  0        &  0\cr
  & 0   & 0       & 0         & 0  & 0         &  0       & l  & l     & l    & k &   k       & k \cr
 }\displaystyle.
 \]
The vertices are expressed as intersections of hyperplanes in the
following way,
\[
\begin{array}{l}
  H_{-e_1,0} \cap H_{-e_2,0} \cap H_{e_i,0} \cap H_{e_3 + e_4,k} = \{k
  e_j \},  % & i, j \in \{3,    4\}, i \neq j\\
\\
  H_{-e_1,0} \cap H_{-e_2,0} \cap H_{e_i,0} \cap H_{-e_3 - e_4,l} = \{
  l e_j \},  % & i, j \in \{3,    4\}, i \neq j\\
\\
  H_{-e_{1 +j_1},0} \cap H_{-e_{3 + j_2},0} \cap H_{e_3 + e_4,k} \cap
  H_{\sum e_i,m} = \{  (m - k) e_{2 - j_1} + k e_{4 - j_2} \},  %&  i,
                                %j \in \{0, 1\} \\
\\
  H_{-e_{1 + j_1},0} \cap H_{-e_{3 + j_2},0} \cap H_{-e_3 - e_4,l} \cap
  H_{\sum e_i,m} = \{  (m - l) e_{2 - j_1} + l e_{4 - j_2} \}, % &  i,j
                                % \in \{0, 1\}.
\end{array}
\]
where $i, j \in \{3, 4\}$, $i \neq j$ and $j_1, j_2 \in \{0, 1\}$.
\end{lemma}
\begin{proof}
  As in the previous lemma, the
  intersection $H_{-e_3,0} \cap H_{-e_4,0}$ contains no part of $P_2$.
  Likewise, the intersection $H_{-e_1,0} \cap H_{-e_2,0} \cap H_{\sum e_i,m}$ contains
  no vertices due to the conflicting constraint $x_3 + x_4
  \leq k$. Again the implication is that no intersection of five hyperplanes
  contains a vertex of $P_2$.

  Of the four-fold intersections those involving neither of $H_{-e_3,0}$
  nor $H_{-e_4,0}$ are either contained in $H_{e_3 + e_4,k} \cap H_{-e_3 - e_4,l}$ or in $H_{-e_1,0} \cap
  H_{-e_2,0} \cap H_{\sum e_i,m}$, and thus contribute nothing.
  The remaining $4{3 \choose 2} = 12$ options involving exactly one of
  $\{H_{-e_3,0}, H_{-e_4,0}\}$ and exactly one of $\{H_{e_3 + e_4,k}, H_{-e_3 - e_4,l}\}$ all
  contribute a vertex of $P_2$.

\end{proof}

\subsection{Minkowski sum shapes}
\label{sec:minkowski-sum-shape}

We are over halfway in our five step program to proving that there are
at most $(m^4 - 5m^2 + 4m)/4$ squares inscribed on an algebraic plane
curve of degree $m$.   In the previous section we showed that the
Newton polytopes $\newton[g_1]$ are of the types
$P_0$, $P_1$ and $P_2$ defined in \Fref{def:polytopes-pi}. In the
fourth step of our program we show that the Minkowski sum $\lambda_1
\newton[g_1] + \lambda_2 \newton[g_2] + \lambda_3 \newton[g_3] +
\lambda_4 \newton[g_4]$ is itself a $P_2$ type polytope. This result, \Fref{lem:minkowski-normal-fan},
is due to the combination of two
facts: the common refinement of the normal fans of $P_0$, $P_1$
and $P_2$ is the normal fan of $P_2$, and \Fref{lem:cone-refinement},
which states that the normal fan of a Minkowski sum is the common
refinement of the normal fans of the summands.  Knowing the form of
the Minkowski sum enables us to calculate its volume to finally
determine the mixed volume of the $\newton[g_i]$.
\\

As the polytopes $\newton[g_i]$ are of different shape depending on
the parity of $m$, as summarized in \Fref{tab:shape-alternation}, we rewrite the Minkowski sum $\sum \lambda_i \newton[g_i]$ as
$\mu_1 P_1 + \mu_2 P_2 + \lambda_4\newton[g_4]$.  Since $\newton[g_2]
= \newton[g_3]$ one of $\mu_1$ or $\mu_2$ equals $\lambda_2 +
\lambda_3$, while the other coefficient $\mu_i$ is set to $\lambda_1$.
\Fref{tab:mu} summarizes the values of $\mu_1$ and $\mu_2$.
\begin{table}[H]
\[
\begin{array}{ccc}
              & m \text{ even}                 & m \text{ odd} \\
   \mu_1 & \lambda_1                      & \lambda_2 + \lambda_3 \\
   \mu_2 & \lambda_2 + \lambda_3 &  \lambda_1
\end{array}
\]
\captionsetup{width=0.7\textwidth}
\caption{The values of the coefficients $\mu_1$ and $\mu_2$ in \\ the
  expression of~$\sum_{i=1}^4 \lambda_i \newton[g_i] = \mu_1 P_1 +
  \mu_2 P_2 + \lambda_4 \newton[g_4]$.}
\label{tab:mu}
\end{table}
The following lemma from Ziegler's Lectures on Polytopes tells us that
we should look at the normal fans of the Newton polytopes to determine
the normal fan of the Minkowski sum.
\begin{lemma}[{\cite[Proposition
7.12, p198]{ziegler}}]
\label{lem:cone-refinement}
\index{polytope!normal fan}
The normal fan of a Minkowski sum is the common refinement of normal
fans of the summands.
\end{lemma}
\begin{proof}
Let $P = P_1 + \dots + P_n$ and let $\Gamma$ be a face of $P$. Fix a functional
$\alpha$ in the normal cone of $\Gamma$, that is, $\Gamma$ is precisely
 the subset of $P$
that is maximal under $\alpha$.  Let $\Gamma \ni
v = v_1 + \dots + v_n$.  Suppose that some $v_j$ does not maximize
$\alpha$ in $P_i$. Then there exists a $w_j \in P_j$ such that
\[
   \alpha(v) = \sum \alpha(v_i) < \sum_{i \neq j} \alpha(v_i) +
   \alpha(w_j) = \alpha(v - v_j + w_j).
\]
The vector $v - v_j + w_j$ is an element of  $P$ by definition of the Minkowski sum,
but this contradicts $\Gamma$ being the maximizer of $\alpha$.
Thus the faces of the $P_i$ that are the summands in $\Gamma =
\Gamma_1 + \dots + \Gamma_n$ are themselves maximizers of $P_i$ with respect to
$\alpha$.  The normal cone of $\Gamma$ is then the intersection of the
normal cones of the $\Gamma_i$.
\end{proof}

The normal cone of any face of a polytope is spanned by the
facet normals of the facets said face is contained in. Thus, the normal fan of a polytope is completely determined  by the normal
cones of the vertices of a polytope. The descriptions of the vertices
as intersections of hyperplanes in
\Fref{lem:polytope-shape} and
\Fref{lem:second-polytope-shape}
directly tell us what the normal cones of the vertices of $P_1$ and $P_2$ are.  To show
that $\mu_1 P_1(m, l_1) + \mu_2 P_2(m, l_2, k) + m\Delta$ is of type
$P_2$ we first show that $P_0$, $P_1$ and $P_2$ have normal fans that
successively refine each other.
\begin{lemma}
\label{lem:minkowski-normal-fan}
\index{polytope!normal fan}
The Minkowski sum $\mu_1 P_1(m, l_1) + \mu_2 P_2(m, l_2, k) + m\Delta
= P_2(m', l', k')$ where
\[
\begin{array}{lcr}
  m' = (\mu_1 + \mu_2 + \lambda_4)m, \quad &
  l'   = \mu_1 l_1 + \mu_2 l_2, \quad &
  k'  =  (\mu_1 + \lambda_4)m +\mu_2k \\
\end{array}.
\]
\end{lemma}
\begin{proof}
  We obtain $P_{i + 1}$ from $P_{i}$ by introducing an additional
  facet-defining hyperplane $H^i$.  As $P_i$ and $P_{i+1}$ are both
  simple, any vertices contained in $H^i$ are contained in three other
  hyperplanes.  The normal cone of a vertex in $H^i$ lies within the
  normal cone of a vertex of $P_{i}$ cut off from $P_{i + 1}$ by
  $H^i$; each vertex cut off lies in an intersection $H_1^i \cap \dots
  \cap H_{r_i}^i$ of hyperplanes whose facet-normals generate a cone
  containing the facet-normal of $H^i$.

  We see from the vertex-facet incidences of \Fref{lem:polytope-shape} and
  \Fref{lem:second-polytope-shape} that the vertices of $P_0$ that are
  cut off from $P_1$ by $H_{- e_3 - e_4,l}$ lie in the intersection $H_{-e_3,0} \cap H_{-e_4,0}$
  and the facet-normal $-e_3 - e_4$ of $H_{-e_3 -e_4,l}$ is the sum of the facet-normals of $H_{-e_3,0}$ and $H_{-e_4,0}$.

  Likewise, the vertices of $P_2$ that are cut off from $P_1$ by
  $H_{e_3 + e_4,k}$ lie in the intersection $H_{-e_1,0} \cap H_{-e_2,0}
  \cap H_{\sum e_i,m}$ and again the facet-normal of $H_{e_3 + e_4,k}$ is the sum
  of the facet normals $e_1 + e_2 + e_3 + e_4$, $-e_1$ and $-e_2$.

  Thus the normal fan of $P_2$ is a refinement of the normal fan of
  $P_1$ which is a refinement of the normal fan of $P_0$; the common
  refinement of the normal fans of $P_0$, $P_1$ and $P_2$ then is
  the normal fan of $P_2$.  By \Fref{lem:cone-refinement} this is also
  the normal fan of the Minkowski sum $\sum_{1}^4 \lambda_i
  \newton[g_i]$.
\\

In particular the Minkowski sum is itself a $P_2(m', l', k')$
polytope for appropriate constants $m'$, $l'$ and $k'$.
We can read off the values of $m'$ and $k'$ from the vertices of
$P_2(m', l', k')$ contained in the intersection of hyperplanes with
normals $(0, 0, 1, 1)$ and $(1, 1, 1, 1)$, for example the vertex $(m'
- k', 0, k', 0)$.  This vertex is the sum of vertices $v_i$ of the
summands of $\mu_1 P_1 + \mu_2 P_2 + P_0$ that have
a normal cone containing its normal cone.

As the normal cone of $H_{-e_1,0} \cap H_{-e_2,0} \cap H_{-e_3,0} \cap H_{\sum e_i,m}$ contains the normal cone of $H_{-e_{1 + j},0} \cap H_{-e_3,0}
\cap H_{\sum e_i,m} \cap H_{e_3 + e_4,k}$, we get the vertex
$
   (\mu_1 + \lambda_4) me_{4} + \mu_2 \left( (m - k)e_{2 - j} +
     ke_{4} \right)
$.
Summing up the coefficients gives $m' = (\mu_1 + \mu_2 + \lambda_4)m$.
The coefficient of $e_{4}$ is $k' = (\mu_1 + \lambda_4)m +
\mu_2k$.

The value of $l'$ can be recovered  from a vertex contained in $H_{-e_3 - e_4,l}$.
As the normal cone of $H_{-e_1,0} \cap H_{-e_2,0} \cap H_{-e_3,0} \cap H_{-e_4,0}$
contains the normal cone of $H_{-e_1,0} \cap H_{-e_2,0} \cap H_{-e_{3},0}
\cap H_{-e_3 - e_4,l}$, we get the vertex $(\mu_1 l_1 + \mu_2 l_2) e_{4}$ of the
Minkowski sum, so $l' = \mu_1 l_1 + \mu_2 l_2$.
\end{proof}

\subsection{Minkowski sum volumes}
\label{sec:minkowski-sum-volumes}

We have one step left of our program towards proving \Fref{thm:mine}.
Recall that \Bernshtein's Theorem uses the
mixed volume $MV(\newton[g_1], \newton[g_2], \newton[g_3],
\newton[g_4])$ to bound the number of isolated solutions in
$\V(g_1, g_2, g_3, g_4) \cap (\mathbb{C} \setminus \{ 0 \})^4$.
The mixed volume, defined in \Fref{def:mixed-volume}, is the coefficient
of the monomial $\lambda_1\lambda_2\lambda_3\lambda_4$ as it appears
in the expression for the volume of the Minkowski sum
$\sum_{i=1}^4 \lambda_i\newton[g_i] $.   In
\Fref{lem:minkowski-normal-fan} we showed that this Minkowski sum can
be expressed as
the polytope $P_2( (\mu_1 + \mu_2 + \lambda_4)m, \mu_1 l_1 + \mu_2 l_2, (\mu_1 + \lambda_4)m + \mu_2k)$.  To complete the final step of our program, we should calculate the volume of a $P_2$ type polytope.

From the halfspace definition in \Fref{def:polytopes-pi} we see that
$P_2(m', l', k')$ is the closure of the set difference ${P_1(m', l') \setminus P_1(m',
  k')}$. Thus the volume of $P_2(m', l', k')$ can be calculated as the
difference in volumes of $P_1(m', l')$ and $P_1(m', k')$.  In turn
we can calculate the volume of $P_1$ as the sum of four simplices that
triangulate $P_1$. The volume of a simplex is straightforward to
calculate by taking the determinant of a matrix whose columnvectors
are the offsets from a distinguished vertex of the simplex to the
other vertices.
For the triangulation of $P_1$ it is convenient to express its facets
in a more combinatorial way.

\begin{corollary}
\label{cor:combinatorial-p1}
Labeling the vertices of $P_1$ by the numbers from one to eight, in the same way as in
  \Fref{lem:polytope-shape}, the combinatorial
facet description of $P_1$ is
\[
\begin{array}{ll}
F_1 = H_{-e_1,0} \cap P_1 = \{1, 3, 4, 5, 7, 8\}  & F_m = H_{\sum e_i,m} \cap P_1 = \{2, 3, 4, 6, 7, 8 \} \\
F_2 = H_{-e_2,0} \cap P_1 = \{1, 2, 4, 5, 6, 8\}  & F_l = H_{-e_3 - e_4,l} \cap P_1 = \{1, 2, 3, 5, 6, 7 \} \\
F_3 = H_{-e_3,0} \cap P_1 = \{5, 6, 7, 8 \}         & F_4 = H_{-e_4,0} \cap P_1 = \{1, 2, 3, 4\}
\end{array}
\]
\end{corollary}
\begin{proof}
The statements of \Fref{lem:polytope-shape} and
\Fref{lem:second-polytope-shape} express the vertices as intersections
of hyperplanes.  Inverting the relationship and expressing the facets as the set of vertices they
contain ends up with the statement above.
\end{proof}
We triangulate $P_1$ by writing it as the union of four simplices,
each of which is defined by a set of five affinely independent
vertices of $P_1$. As long as these simplices intersect in
lower-dimensional faces we obtain a triangulation of $P_1$.

\begin{corollary}
\label{cor:volume}
The volume of $P_1(m, l)$ is $(m - l)^3(m + 3l)4!$.
\end{corollary}
\begin{proof}
  We shall first triangulate $P_1$, calculating its volume is then a matter of summing the volumes of the triangulating simplices.

  Let $v$ be a vertex of $P_1$. An \lqindex{opposing facet} of $v$ is
  facet of $P_1$ that does not
  contain $v$.  Assume that we have a triangulation
  of every opposing facet of $v$.  The convex hull of $v$ and a simplex in a
  triangulation of an opposing facet is again a simplex. By
  \Fref{lem:triangulation} the simplices thus obtained triangulate $P_1$.
  The Cohen-Hickey algorithm~\cite[Section 3.1]{exact-volumes}
  triangulates a polytope by picking a vertex and recursively
  triangulating its opposing facets.

  From the  combinatorial description of $P_1$ given in
  \Fref{cor:combinatorial-p1} it is easy to read off what the
  facets opposing a vertex are.  In that notation the vertices of
  $P_1$ are labeled $1, \dots, 8$.
  We start the Cohen-Hickey algorithm by selecting as the first vertex
  $v_1 = 1$. Its opposing facets are $F_3$ and $F_m$, the former of
  which is already a simplex (it is three-dimensional on four vertices).

  The next step of the recursion triangulates $F_m$ by picking $v_2 =
  2$.  The facets of $F_m$ that oppose $v_2$ are intersections of
  $F_m$ with facets of $P_1$ that oppose $v_2$, that is, $F_m \cap F_3 = \{6,
  7, 8\}$, a simplex, and $F_m \cap F_1 = \{3, 4, 7, 8\}$.
  At the deepest level of the recursion we triangulate $F_m \cap F_1$ by
  picking $v_3 = 3$ and we find the one-dimensional simplices $F_m
  \cap F_1 \cap F_2 = \{4, 8\}$ and $F_m \cap F_1 \cap F_3 = \{7,
  8\}$.  The triangulation of $F_m \cap F_1$ is depicted in \Fref{fig:P1-triangulation}.

Our application of the Cohen-Hickey algorithm results in the following triangulation of $P_1$: $\{\{1, 5, 6, 7, 8\}$, $\{1, 2, 6,
7, 8\}$, $\{1, 2, 3, 4, 8\}$, $\{1, 2, 3, 7, 8\}\}$.  The volume of $P_1(m, l)$ is the sum of the volumes of the simplices in this triangulation,
\begin{align*}
 \Vol_4(P_1(m, l)) &=
  \begin{vmatrix*}[r]
    0                    & m - l & 0                       & 0 \cr
    0                    & 0                       & m - l & 0 \cr
    l    & 0                       & 0                       & 0 \cr
    -l  & 0                       & 0                       & m - l \cr
  \end{vmatrix*}4!  +
  \begin{vmatrix*}[r]
    m - l  & m - l & 0       & 0    \cr
    0        & 0       & m - l & 0     \cr
    0        & -l      & -l     & -l \cr
    0        & l        & l       & m \cr
  \end{vmatrix*}4!  \\&+
  \begin{vmatrix*}[r]
    m - l  & 0       & 0       & 0    \cr
    0        & m - l & m - l & 0     \cr
    0        & 0       & -l     & -l \cr
    0        & 0       & l       & m \cr
  \end{vmatrix*}4! +
  \begin{vmatrix*}[r]
    m - l  & 0       & 0       & 0    \cr
    0        & m - l & 0       &  0     \cr
    0        & 0       & m - l & -l \cr
    0        & 0       & 0       & m \cr
  \end{vmatrix*}4! \\
&= (m - l)^3 l4! + (m - l)^3 l4! +
(m - l)^3 l4! + (m - l)^3 m4! \\ &= (m - l)^3(m + 3l)4!.
\end{align*}
\end{proof}

\begin{figure}[H]
\centering
  \includegraphics[width=\linewidth]{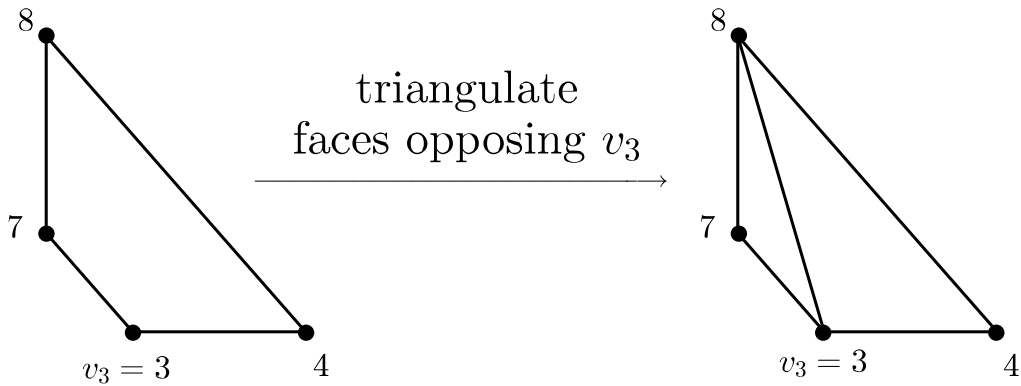}
  \captionsetup{width=0.9\textwidth}
  \caption{Triangulation of the face $F_m \cap F_3$ of $P_1$, as in the proof of \Fref{cor:volume}.}
 \label{fig:P1-triangulation}
\end{figure}

To calculate the volume of the Minkowski sum $\sum \lambda_i \newton[g_1] = \overline{P_1(m', l') \setminus P_1(m', k')}$ we apply \Fref{cor:volume} and subtract the volume of $P_1(m', k')$ from that of $P_1(m', l')$.  The expression for the volume we obtain is
$(m' -l')^3(m' + 3l')4! - (m' - k')^3(m' + 3k')4!$.

The mixed volume of $\newton[g_1]$, $\newton[g_2]$, $\newton[g_3]$, $\newton[g_4]$ can be extracted from the above volume as the coefficient of the monomial $\lambda_1\lambda_2\lambda_3\lambda_4$.
Extracting this coefficient by hand is somewhat tedious; Macaulay2 code that
performs the necessary algebraic manipulations is included in the
appendix, see \Fref{lst:minvol.m2}.   Recall from
\Fref{sec:newton-polytope-shapes} that for degrees two and three the
polytopes $P_1$ and $P_2$ are not both four-dimensional.   For these
two boundary cases the code in \Fref{lst:lowDegreeBKK.m2} uses
the PHCpack~\cite{m2-phcpack} interface from Macaulay2 to calculate the mixed volumes, which conform to the same formula as the $m \geq 4$ case.

At last we see that for all $m \in \N$ the mixed volume of the Newton polytopes $\newton[g_1]$, $\newton[g_2]$, $\newton[g_3]$, $\newton[g_4]$ is $m^4 -
5m^2 + 4m$.

\subsection{Applied BKK bound}

We set out to prove that the number of isolated squares inscribed on
an algebraic plane curve of degree $m$ is bounded by $(m^4 - 5m^2 +
4m)/4$.  In the last five sections we have shown that the variety of
complex squares inscribed on a plane curve $\V(f)$ is defined by four
polynomials $g_i$ with the property that the mixed volume of their
Newton polytopes is $(m^4 - 5m^2 + 4m)$.  An immediate consequence of
\Bernshtein's Theorem applied to these data is that the number of
isolated squares of $\V(g_1, g_2, g_3, g_4)$ that do not lie in a
coordinate hyperplane is bounded by $(m^4 - 5m^2 + 4m)$.  By passing
to a different choice of coordinates we can assume no isolated squares
lie in any coordinate hyperplane.  Finally, as there are four
parametrizations of every square inscribed on $\V(f)$ we divide the
mixed volume by four and have proven \Fref{thm:mine}.

\begin{theorem}
\label{thm:mine}
Let $f \in \mathbb{C}[x, y]$ of degree $m$ define an algebraic plane curve $\V(f)
\subset \mathbb{C}^2$. The number of isolated squares inscribed on
$\V(f)$ is at most $(m^4 - 5m^2 - 4m)/4$.
\end{theorem}

\section{Experimental evidence for the number of complex squares }
\label{sec:experimental}

How many squares can be inscribed on an algebraic plane curve?
\Fref{thm:mine} states that at most $(m^4 - 5m^2 + 4m)/4$ isolated
squares are inscribed on a plane curve of degree $m$.  Is this bound
sharp, and if so, how often?

\Fref{tab:experiments} tabulates, for degrees three to ten, the number
of squares (possibly with multiplicities) inscribed on the majority of
plane curves from a sample of randomly chosen curves.  The experiments
were carried out using the computer algebra system
Macaulay2~\cite{M2}, the code used is listed in
\Fref{lst:numevidIdeal.m2}. In all the cases the varieties turned out
to be zero-dimensional, in which case all the squares inscribed on a
curve are isolated.  Note that the number of squares found on the
curves of the sample, entered in the third column of
\Fref{tab:experiments}, agrees exactly with the maximum $(m^4 - 5m^2 +
4m)/4$ provided by \Fref{thm:mine}.  Not only is the bound sharp,
these experiments suggest that the bound is attained for \emph{all}
squares inscribed on a generic curve.
Proving this stronger result is out of scope for the current thesis.
\begin{table}
\begin{centering}
\begin{tabular}{|c|c|c|c|c|}
\hline
Degree $m$ &  \# solutions & squares & fraction & field \\
\hline
3  & 48   & 12   & 4991/5000 & $\mathbb{Q}$ \\
4  & 192  & 48   & 4998/5000 & $\mathbb{Q}$ \\
5  & 520  & 130  & 100/100   & $\mathbb{Q}$ \\
6  & 1140 & 285  & 50/50     & $\mathbb{Z}/32479$ \\
7  & 2184 & 546  & 1/1       & $\mathbb{Z}/32479$ \\
8  & 3808 & 952  & 1/1       & $\mathbb{Z}/32479$ \\
9  & 6192 & 1548 & 1/1       & $\mathbb{Z}/32479$ \\
10 & 9540 & 2385 & 1/1       & $\mathbb{Z}/32479$ \\
\hline
\end{tabular}
\captionsetup{width=0.8\textwidth}
\caption{Experimental results for number of complex squares calculated
  using \Fref{lst:numevidIdeal.m2}. The fraction column harbors the
  fraction of the sample of curves that attain the maximal number of squares.}
\label{tab:experiments}
\end{centering}
\end{table}

The curves featuring in \Fref{tab:experiments} were generated by
having Macaulay2 randomly pick the coefficients $c_\gamma$ of $f = \sum_{|\gamma|
  \leq m} c_\gamma x^{\gamma_1}y^{\gamma_2}$ for a fixed degree $m$. As
the degree goes up the memory usage grows. Even a degree six curve
already used more than fourteen gigabytes of memory when working with the
rationals as a base field. Computations for degree seven ran out of
memory after using more than fifty gigabytes.
For this reason finite fields were used in
the calculations with higher degrees.

\section{Illustrative examples of real squares}
\label{sec:illustrative}

The previous section argues that there is not much of interest going
on in the complex case, almost all complex algebraic plane curves inscribe the
maximum number of squares.
For real plane curves, however, we have no
evidence as to what the generic case is.

This section contains selected real plane curves of low degree that
inscribe varying numbers of squares.  The pictures have been plotted
in Maple, using the code from \Fref{lst:drawSquares.mw}, based on
numerical data for the locations of the squares computed by
PHCpack~\cite{m2-phcpack}.  The topology of the curves has been
determined by a manual process: the RAGlib~\cite{raglib} Maple package
provides at least one point on each connected component of a plane
curve, by inspecting the plot and intersecting the curves with
suitably chosen lines we can determine which visible components
connect outside of the plotted range.  The ``realroots.m2''
functionality written by Dan Grayson and Frank
Sottile~\cite{sottile-computations} was used for determining how many
real intersections these lines and the curves have.  The polynomials
that define the curves in the plots are listed in
\Fref{tab:long-polys}.

The maximal number of squares inscribed
on a third degree curve is twelve, according to
\Fref{thm:mine}; the examples in this section show that
a third degree real curve can inscribe any number of squares from zero to twelve, see
\Fref{tab:three-topologies}.
 Two topological types attaining the maximum number
are shown in \Fref{fig:twelve-clear} and
\Fref{fig:twelve-awesome}. Curves of these types look like perturbations of either a)
an oval times a line, or b) the product of three lines.  The
perturbation approach of constructing curves is called the ``marking
method'' by Gudkov~\cite[Section~2.10]{gudkov}.

The proofs of Emch, Jerrard and Stromquist
establish that, generically, on a smooth enough Jordan curve the number of
inscribed squares will be odd.  It is no surprise then that we see
the same behaviour for algebraic plane curves that topologically
speaking are
circles. \Fref{fig:inscribed-zeroOne} shows algebraic Jordan curves
inscribing one, three, five and seven squares.

% zeroOne = select(allCurves, t -> (0, 1) == (t_3, t_4))
% zeroOne / (k -> (length last k, k))
% zeroOnePrime = sort oo / last
% forMapleSequence(oo / first, oo / last)

\begin{figure}
\centering
  \begin{subfigure}[b]{0.44\linewidth}
    \includegraphics[width=\textwidth]{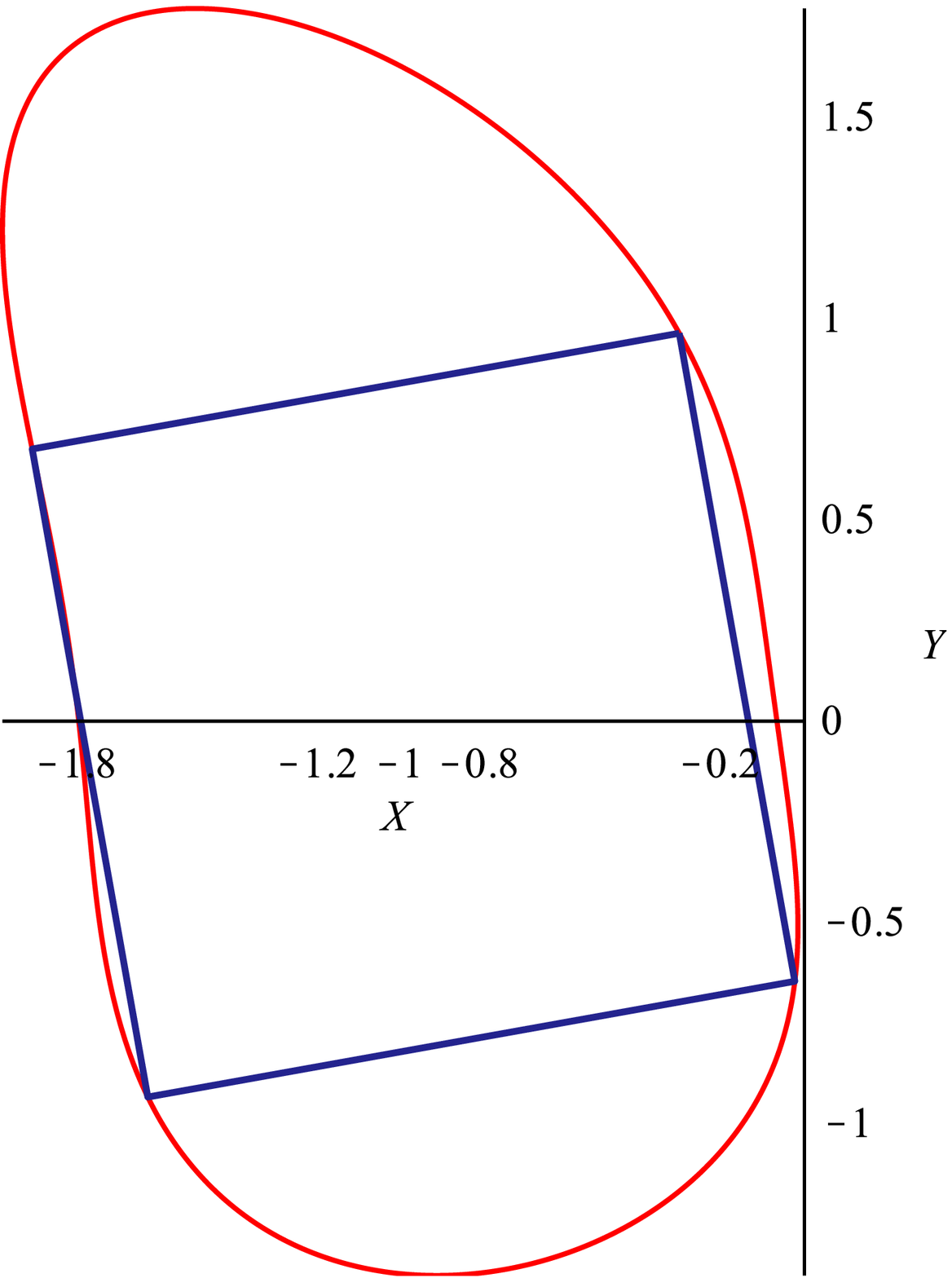}
    \caption{\inscribeOne{One}{30}}
  \end{subfigure}
%\hspace{-3em}
  \begin{subfigure}[b]{0.44\linewidth}
    \includegraphics[width=\textwidth]{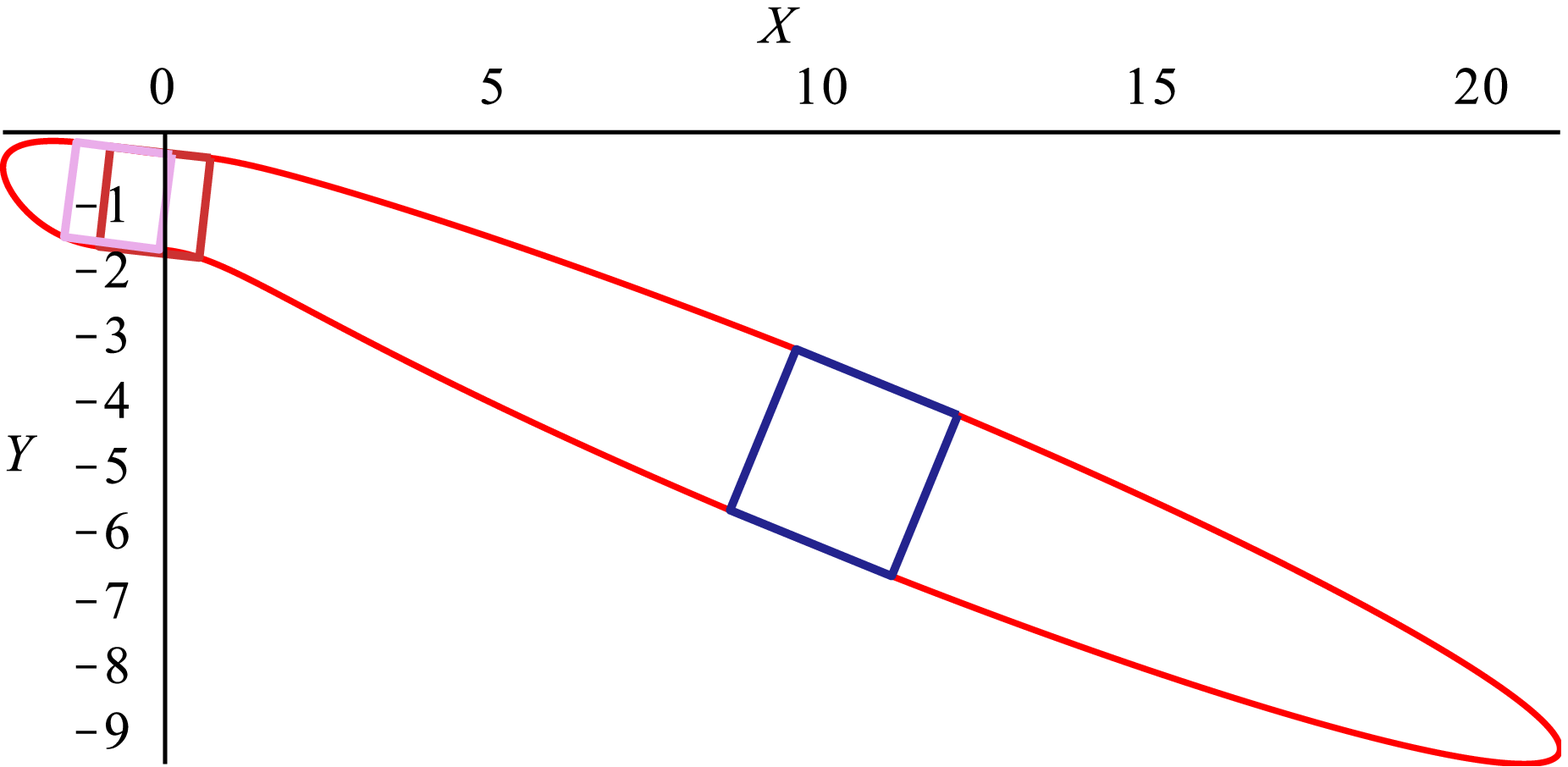}
    \caption{\inscribe{Three}{31}}
  \end{subfigure}
 \begin{subfigure}[b]{0.44\linewidth}
   \includegraphics[width=\textwidth]{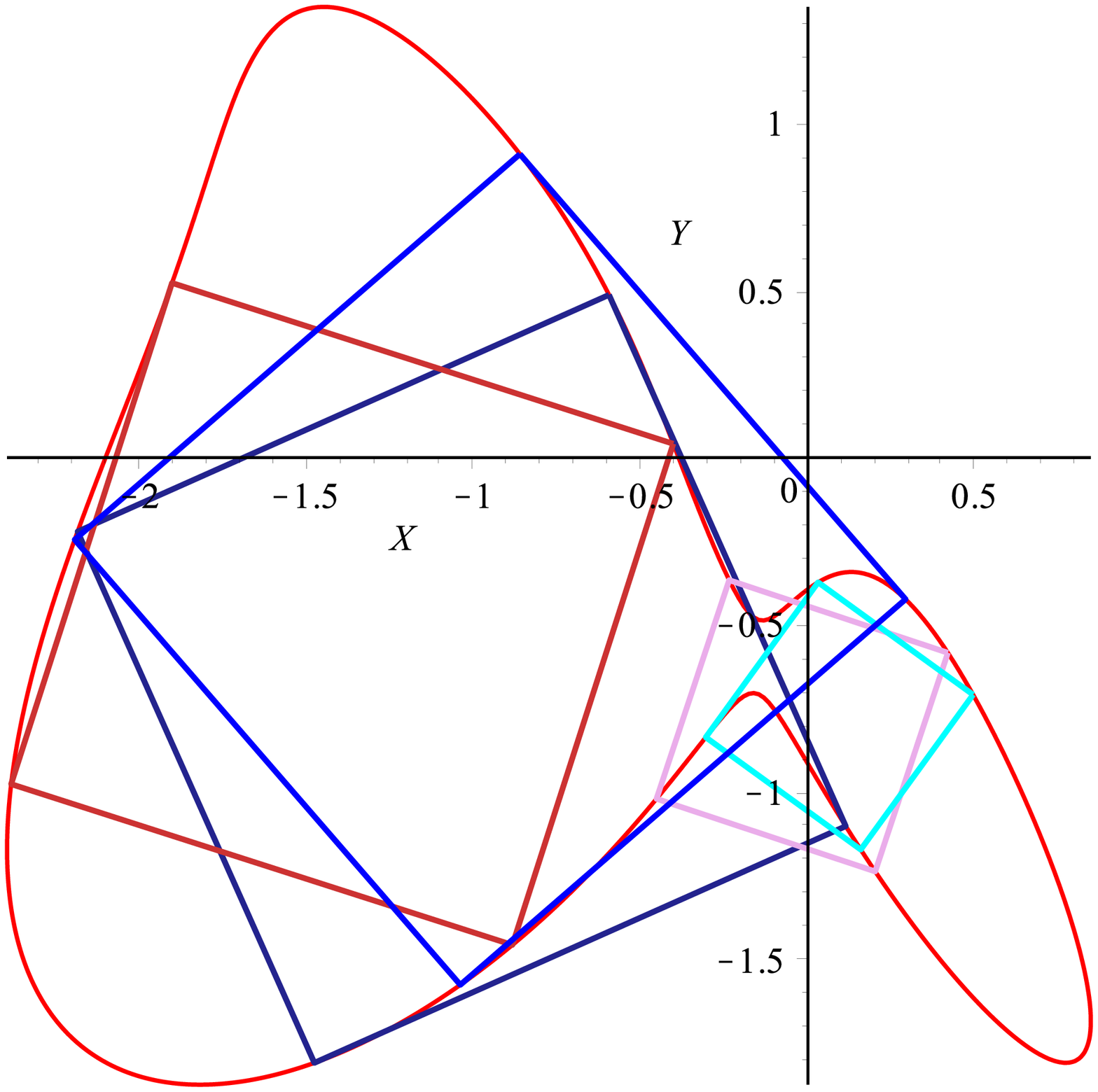}
    \caption{\inscribe{Five}{32}}
  \end{subfigure}
 \begin{subfigure}[b]{0.44\linewidth}
   \includegraphics[width=\textwidth]{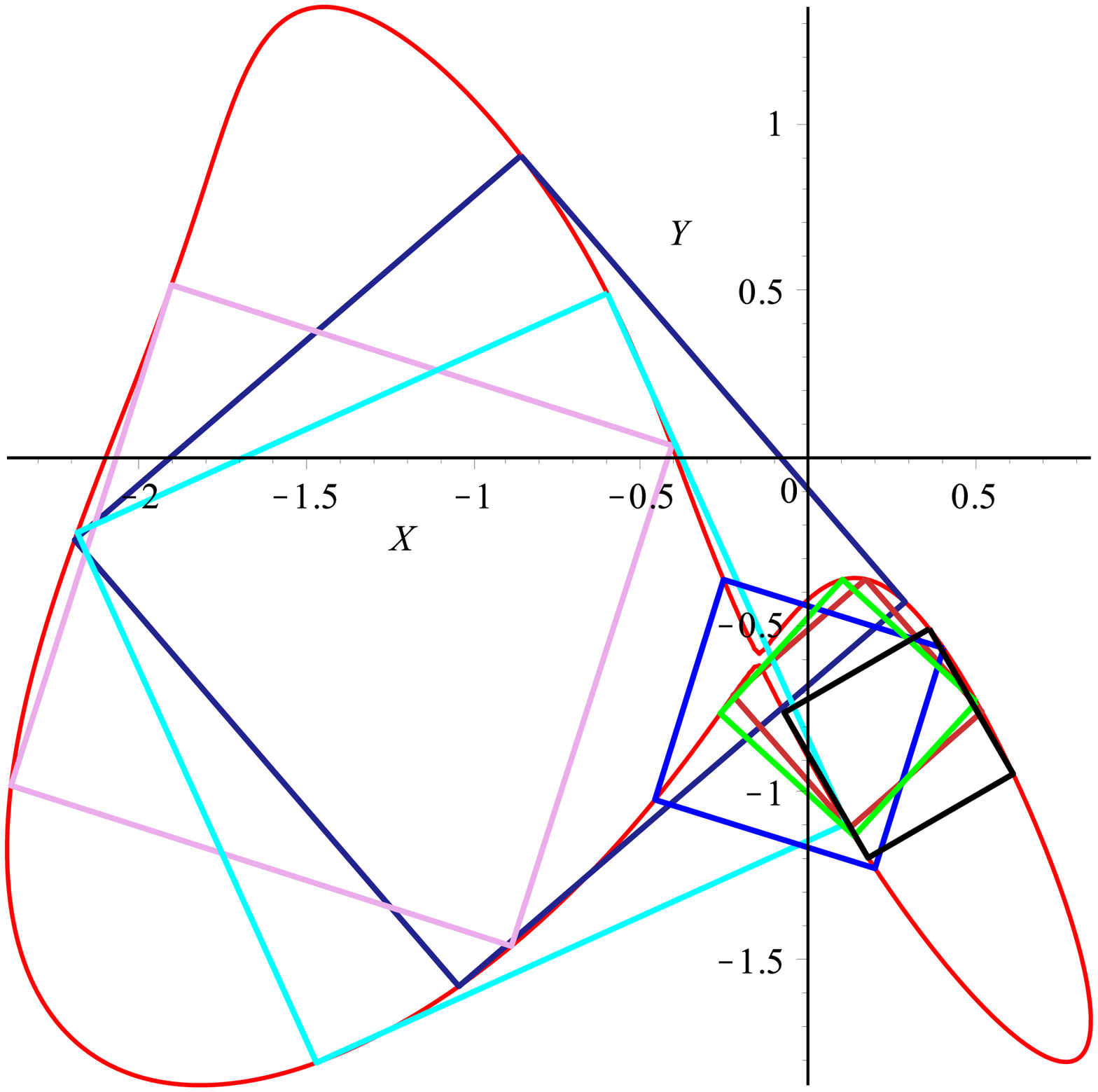}
    \caption{\inscribe{Seven}{33}}
  \end{subfigure}
  \caption{Algebraic Jordan curves inscribing an odd number of
    squares.}
  \label{fig:inscribed-zeroOne}
\end{figure}

Recall that a Jordan curve starts and ends at the same point without
intersecting itself, it is closed and simple.  A Jordan curve has only
one connected component and it is homeomorphic to a circle.  Unlike
Jordan curves, a simple algebraic plane curve can consist of multiple
components, and the components can be homeomorphic to a circle or to
the real line.  \Fref{tab:square-topologies} tabulates the number of
squares found on plane curves computed for this thesis with the code from
\Fref{lst:poging3.m2}; the rows of the table are indexed by the number
of components homeomorphic to the real line, and the columns are
indexed by the number of components homeomorphic to a circle (called
ovals).

The example curves homeomorphic to a real line, as well as some other
topological types of curves, exhibit a parity condition on the number
of inscribed squares just as in the Jordan case, see
\Fref{sec:possible-parity}.  The types for which this occurs have
their entries shaded gray in \Fref{tab:square-topologies}.  Whether
this parity condition is an actual property of these curves or an
artifact of our selection of examples remains to be seen.  Other
topological types have both an odd and an even number of squares,
these are listed in \Fref{sec:no-parity}.

% select(allCurves, t -> (0, 2) == (t_3, t_4))
%forMapleSequence(oo / first, oo / last)

\begin{table}[H]
\centering
\begin{tabular}{l|c|c|c|c|c}
\diaghead{lines ova $i$}{lines $i$}{ovals $j$} & 0 & 1 & 2 & 3 & 4\\
\hline
0 & {} & \cellcolor{gray}{1, 3, 5, 7} & \cellcolor{gray} {0, 2, 4, 6, 16} & & {8} \\
1 & \cellcolor{gray} {0, 2, 4, 6, 12} & {1, $2^*$, 3, 5, 7, 9, 11} & {} &      {} &  \\
2 & {1, 4, 8, 9, 11} & \cellcolor{gray}{3, 5, 7} & {} & {} & \\
3 & {1, 4, 7, 8, 10, 11, 12} & {8, 9, 11} &  {} & {} &  \\
% 4 & {44} &      {} & {} & {} & \\
\end{tabular}
    \captionsetup{width=0.97\textwidth}
    \caption{Number of squares inscribed on curves of degree up to five. The
      $(i, j)$-th cell corresponds to curves homeomorphic to $i$
      copies of the real line and $j$ copies of the circle.  The entry
      $2^*$ in the (1, 1) cell corresponds to \Fref{fig:debate}.}
\label{tab:square-topologies}
\end{table}
The $2$ that occurs in the entry for curves that consist of one line
and one oval corresponds to \Fref{fig:debate}. Inclusion of this reducible curve
is debatable.
If one allows reducible curves,
then taking unions of lower degree curves will construct examples
where the total number of inscribed squares is the sum of the squares
inscribed on each curve in the union, each part behaving independently.  At this point it is not clear
to us whether reducible curves should be excluded.
\begin{table}[H]
\footnotesize
%\centering
\hspace{0.1em}
\begin{subtable}[t]{.45\linewidth}
  \vspace{0pt}
      \centering
      \begin{tabular}{l|c|c}
        & 0 & 1 \\
        \hline
        0 & \\
        1 & \cellcolor{gray}{0, 2, 6, 12} & {1, 2, 3, 5, 7, 9, 11} \\
        2 & \\
        3 & {4, 7, 8, 10, 11, 12} \\
      \end{tabular}
      \caption{Squares inscribed on degree three curves}
      \label{tab:three-topologies}
   \end{subtable}%
\,
    \begin{subtable}[t]{.45\linewidth}
      \vspace{0pt}
    \centering
      \begin{tabular}{l|c|c|c|c|c}
        & 0 & 1 & 2 & 3 & 4\\
        \hline
        0 & {} & \cellcolor{gray}{1, 3, 5, 7} & \cellcolor{gray}{0, 2, 4, 6, 16} & {} & {8} \\
        1 &  &  &  & \\
        2 & {4, 8, 9, 11} & \cellcolor{gray} {3, 5,      7} & & \\
        3 &  &  &  &\\
   %     4 & {44}  & &  &\\
      \end{tabular}
      \caption{Squares inscribed on degree four curves}
    \end{subtable}
    \captionsetup{width=\linewidth}
    \caption{Number of squares inscribed on curves of degree three and four. The
      $(i, j)$-th cell corresponds to curves homeomorphic to $i$
      copies of the real line and $j$ copies of the circle.
    }
\end{table}

\subsection{Topological types of curves with a possible parity condition on the number of
  inscribed squares}
\label{sec:possible-parity}

\subsubsection{One topological line inscribing an even number of squares}

A straight line does not inscribe any squares. Among the curves
computed for this thesis, all of the curves that consist of one
topological component homeomorphic to the real line inscribe an even
number of squares.  Included are two examples of cubic curves
inscribing the maximal number of twelve squares:
\Fref{fig:twelve-clear} and \Fref{fig:inscribed-oneZero-12-max}.  The
other curves in \Fref{fig:inscribed-oneZero} inscribe zero, two, four
and six squares.

Curves that are homeomorphic to a real line but not neccessarily
algebraic are not restricted by this parity condition of inscribing an
even number of squares. Consider
the curve, displayed in \Fref{fig:geo-zigzag}, consisting of two parallel rays in opposite directions,
connected by a line segment at a fortyfive degree angle to both the
rays.  This curve inscribes one
square, it has the line segment $BC$ as a diagonal.

\begin{figure}[H]
\centering
\includegraphics[width=0.7\linewidth]{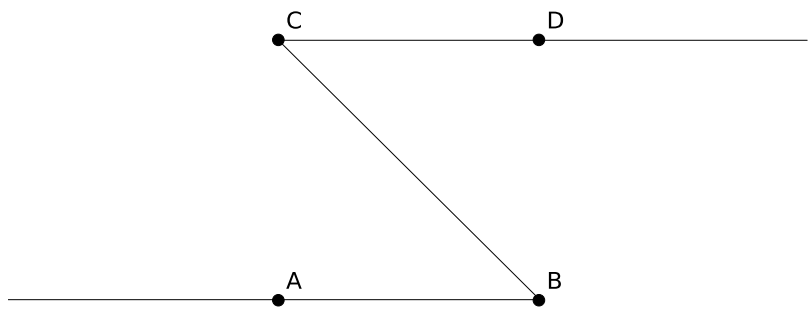}
\caption{A topological line inscribing one square.}
\label{fig:geo-zigzag}
\end{figure}

% select(allCurves, t -> (1, 0) == (t_3, t_4))
% forMapleSequence(oo / first, oo / last)
\begin{figure}
\centering
  \begin{subfigure}[b]{0.45\textwidth}
    \includegraphics[width=\textwidth]{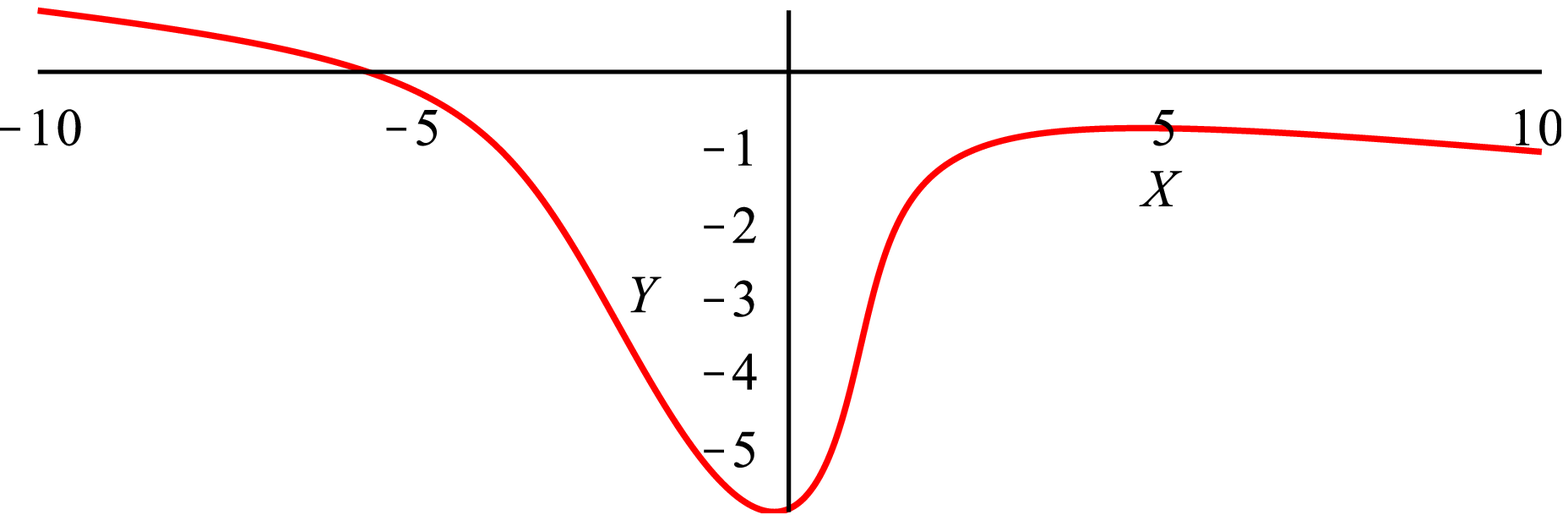}
    \caption{\inscribe{Zero}{1}}
  \end{subfigure}
  \begin{subfigure}[b]{0.45\textwidth}
    \includegraphics[width=\textwidth]{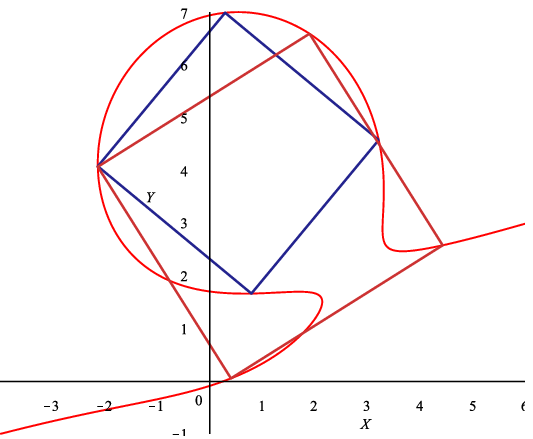}
   \caption{\inscribe{Two}{2}}
  \end{subfigure}
 \begin{subfigure}[b]{0.45\textwidth}
   \includegraphics[width=\textwidth]{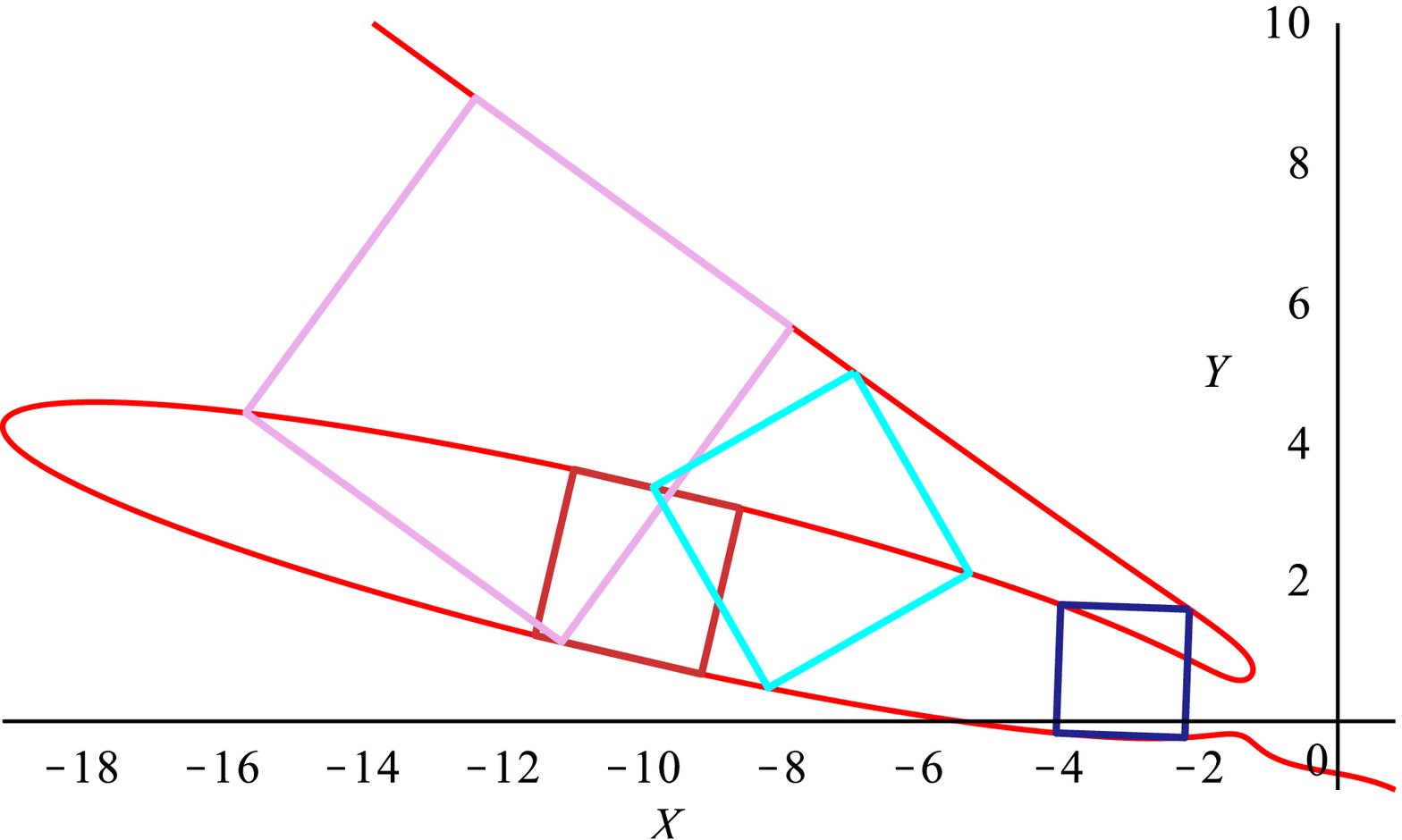}
   \caption{\inscribe{Four}{3}}
  \end{subfigure}
 \begin{subfigure}[b]{0.45\textwidth}
   \includegraphics[width=\textwidth]{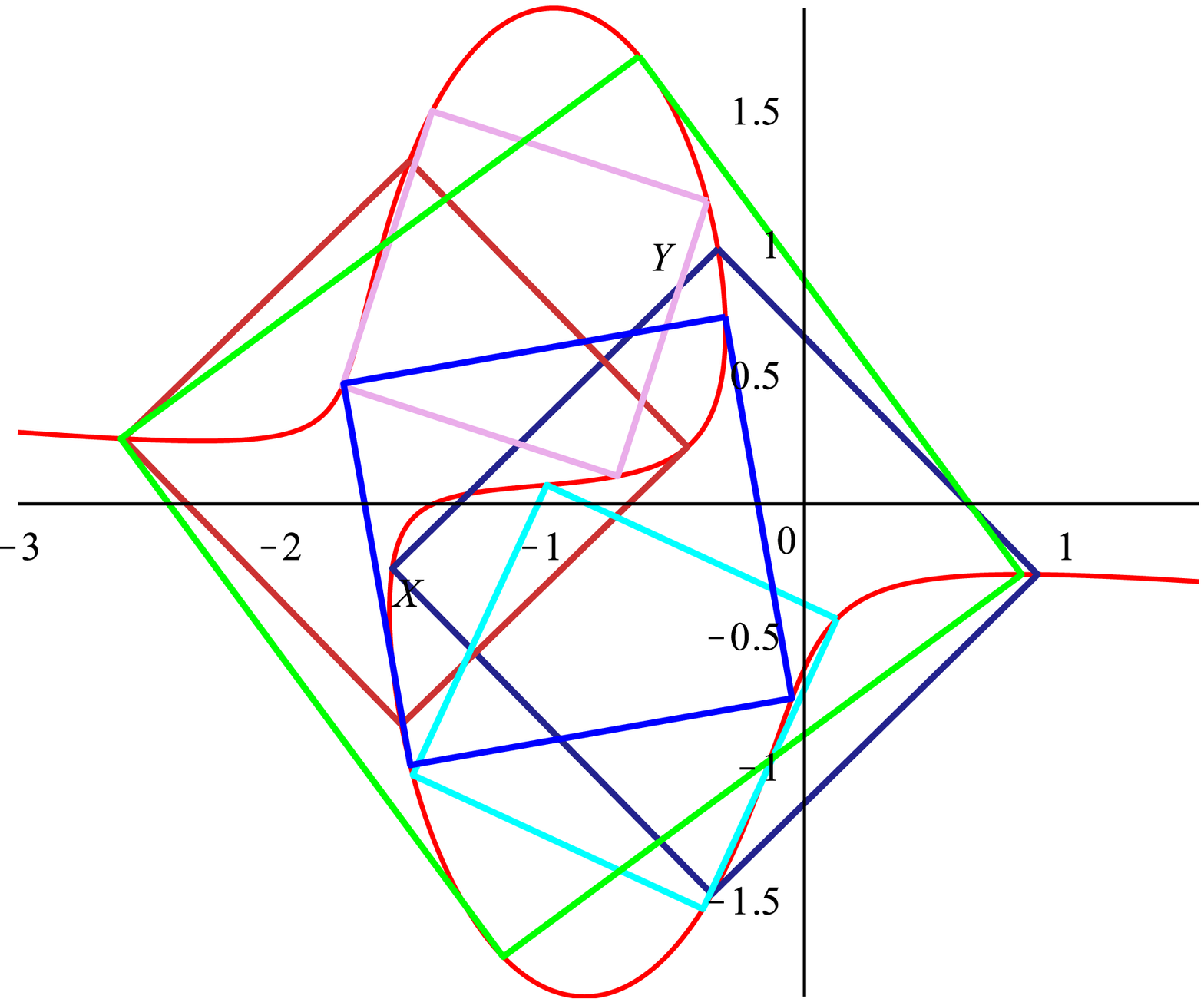}
   \caption{\inscribe{Six}{4}}
  \end{subfigure}
 \begin{subfigure}[b]{0.45\textwidth}
   \includegraphics[width=\textwidth]{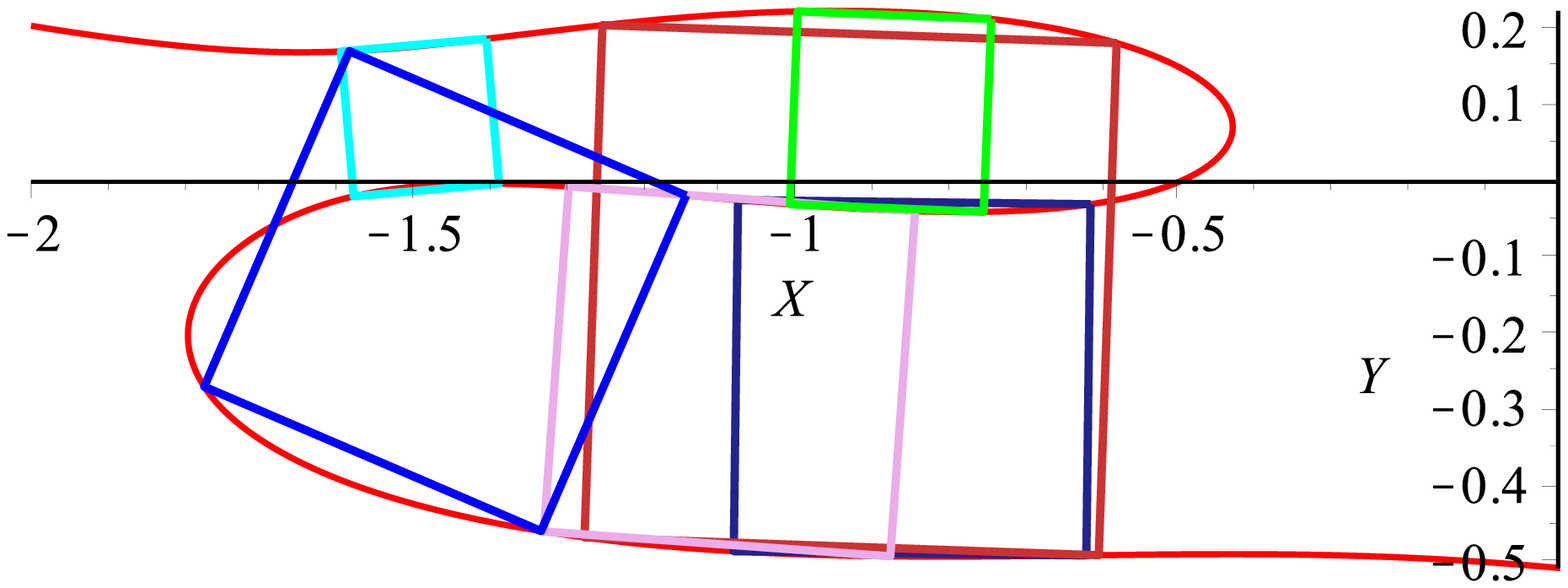}
   \caption{\inscribe{Six}{5}}
  \end{subfigure}
  \begin{subfigure}[b]{0.45\textwidth}
   \includegraphics[width=\textwidth]{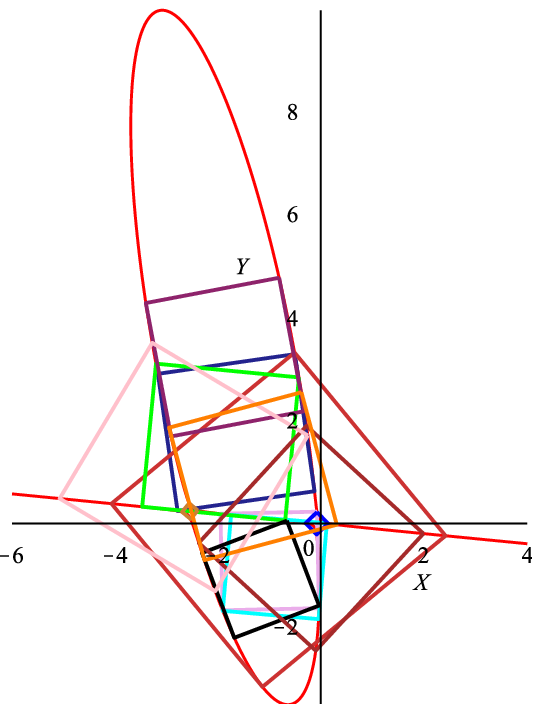}
   \caption{\inscribe{Twelve}{6}}
  \label{fig:inscribed-oneZero-12-max}
  \end{subfigure}
  \caption{An even number of squares inscribed on a line.}
  \label{fig:inscribed-oneZero}
\end{figure}
\begin{figure}
\centering
   \includegraphics[width=0.8\textwidth]{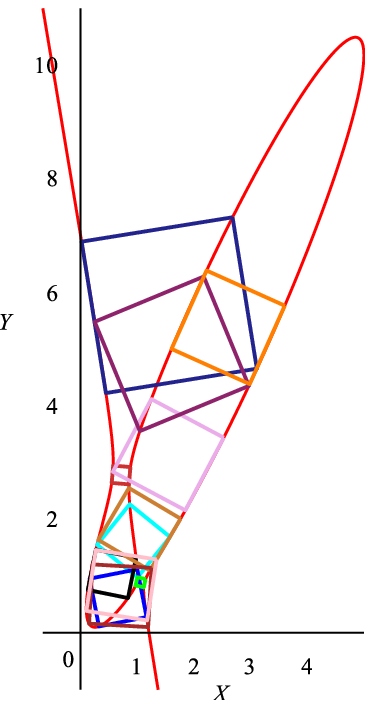}
   \caption{Twelve squares inscribed on
     \protect\hyperlink{poly:f7}{$f_7$} in \Fref{tab:long-polys}}
    \label{fig:twelve-clear}
\end{figure}

\subsubsection{Pairs of ovals inscribing an even number of squares}
The curves in \Fref{fig:inscribed-zeroTwo} consist of two ovals and
inscribe zero, two, four, six and sixteen isolated squares.  The
curves in \Fref{fig:inscribed-zeroTwo-zero} and
\Fref{fig:inscribed-zeroTwo-sixteen} are of the form $(X^2 + Y^2/4 -
1)(X^2/4 + Y^2 - 1) + k$. If $(X, Y)$ lies on such a curve,
then by symmetry it forms one corner of a square centered at the
origin.
The squares depicted in Figures
\ref{fig:inscribed-zeroTwo-zero} and \ref{fig:inscribed-zeroTwo-sixteen} are the squares that do not lie on the
positive-dimensional components of respectively $\V(I_{f_{41}})$ and $\V(I_{f_{45}})$.
% select(allCurves, t -> (0, 2) == (t_3, t_4))
%forMapleSequence(oo / first, oo / last)
\begin{figure}
\centering
  \begin{subfigure}[b]{0.45\textwidth}
    \includegraphics[width=\textwidth]{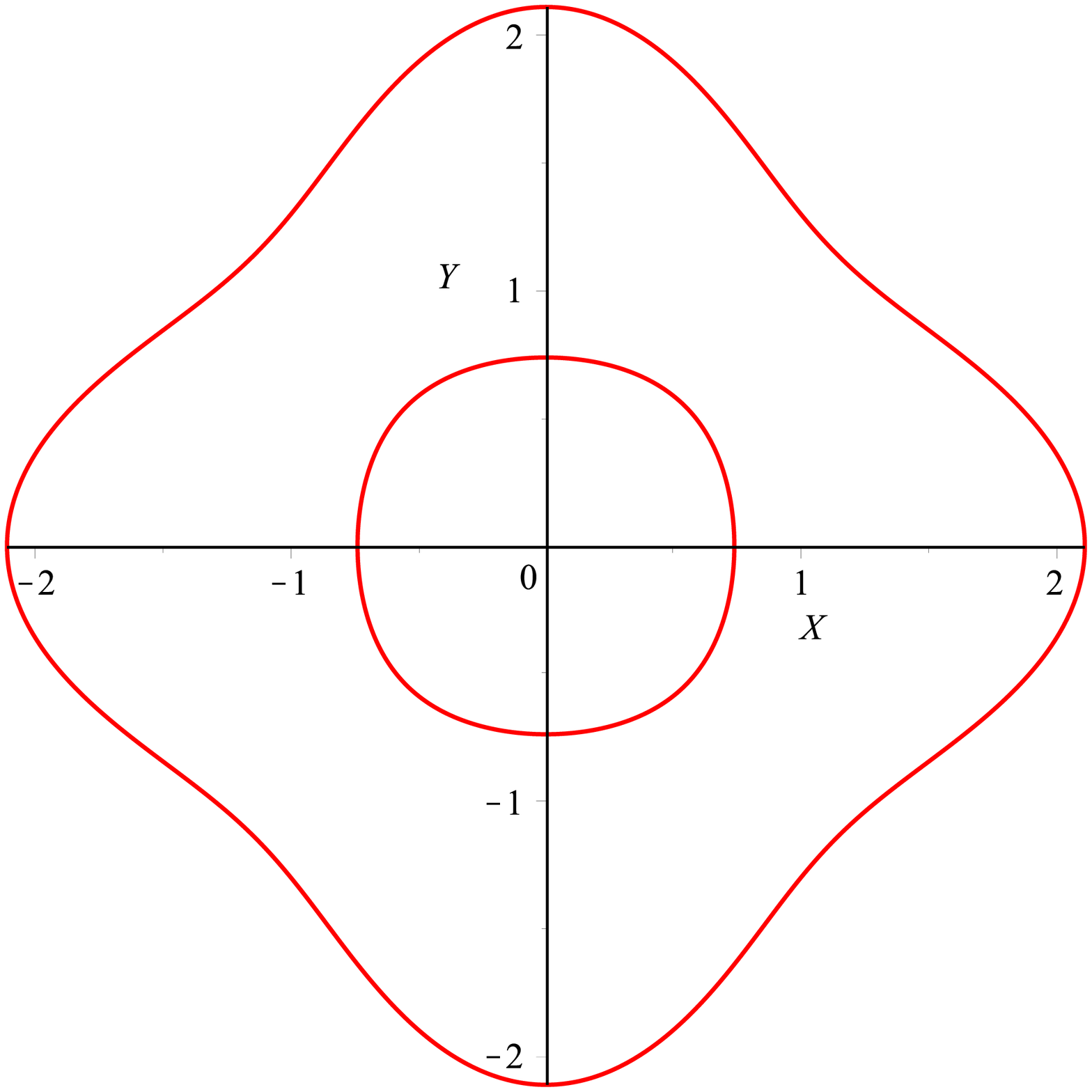}
    \caption{\inscribe{Zero}{41}}
    \label{fig:inscribed-zeroTwo-zero}
  \end{subfigure}
  \begin{subfigure}[b]{0.45\textwidth}
    \includegraphics[width=\textwidth]{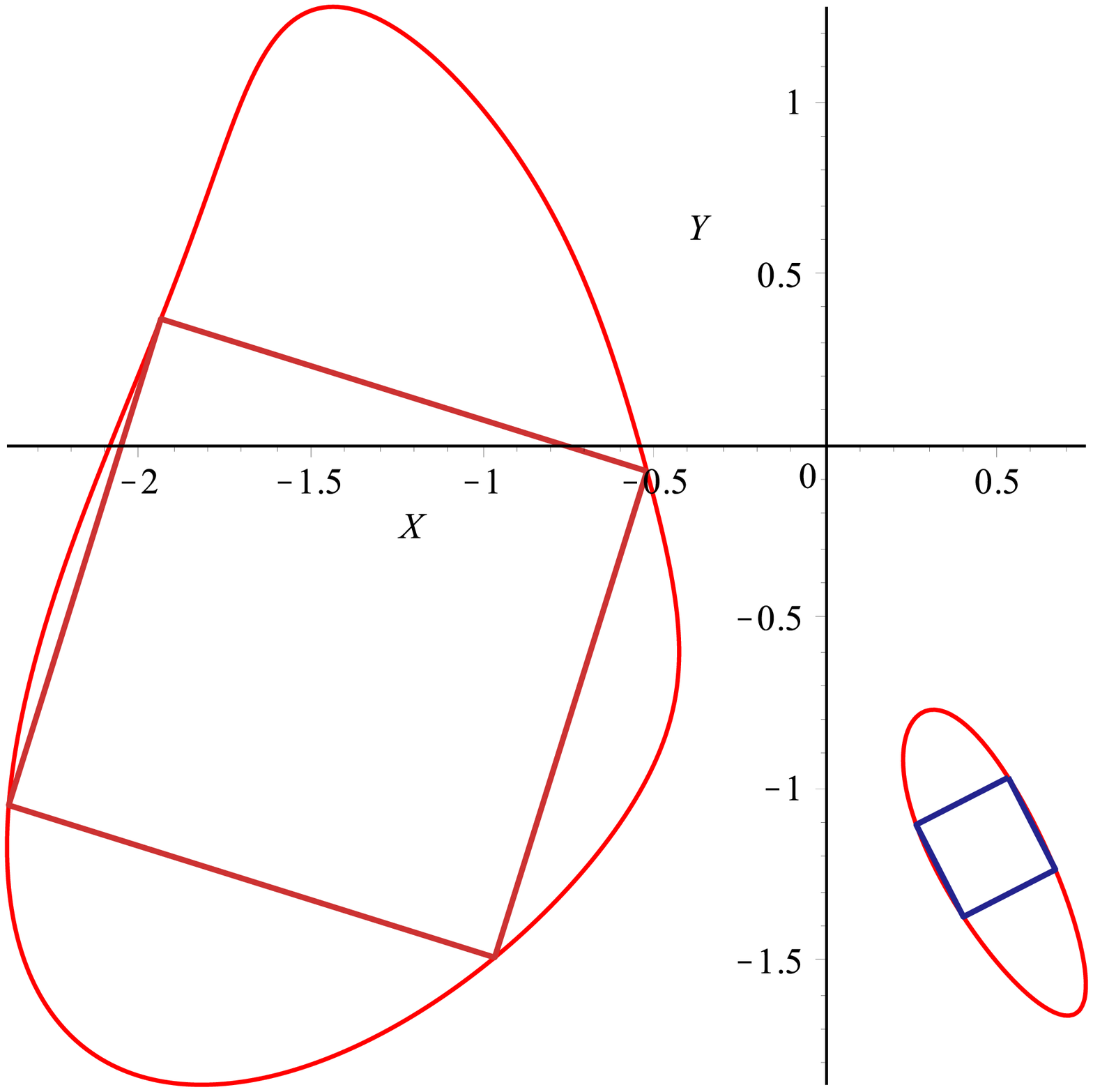}
    \caption{\inscribe{Two}{42}}
  \end{subfigure}
 \begin{subfigure}[b]{0.45\textwidth}
   \includegraphics[width=\textwidth]{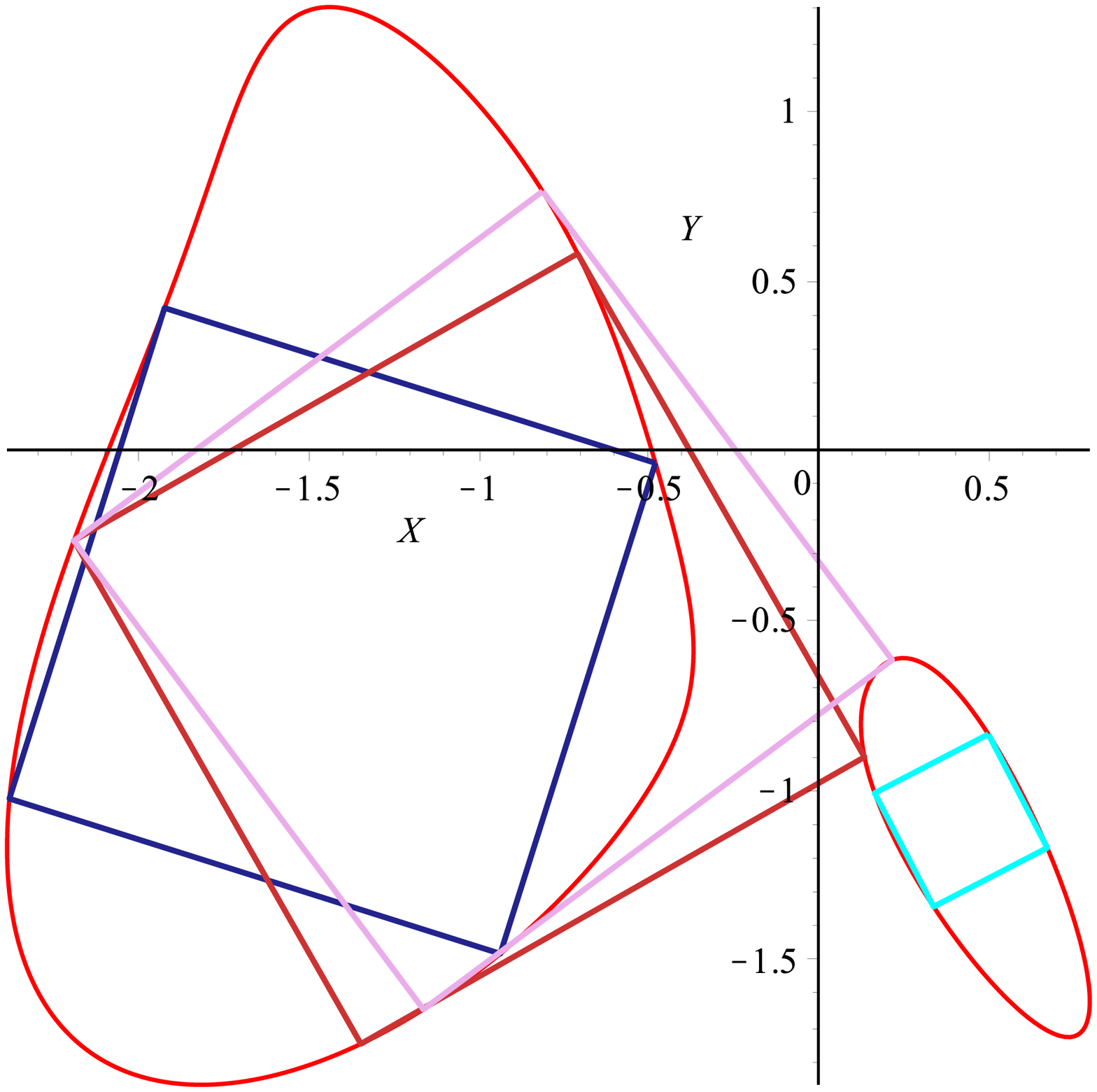}
   \caption{\inscribe{Four}{43}}
  \end{subfigure}
 \begin{subfigure}[b]{0.45\textwidth}
   \includegraphics[width=\textwidth]{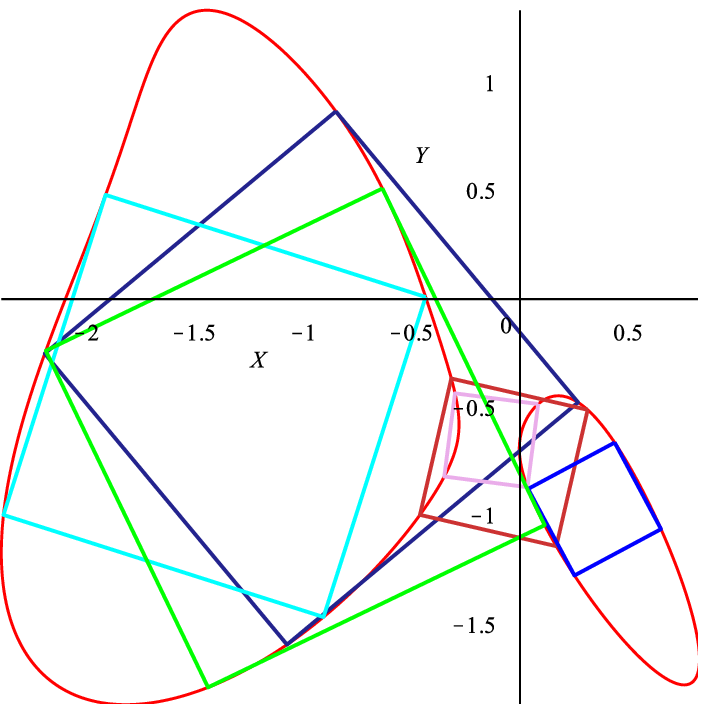}
    \caption{\inscribe{Six}{44}}
  \end{subfigure}
 \begin{subfigure}[b]{0.45\textwidth}
   \includegraphics[width=\textwidth]{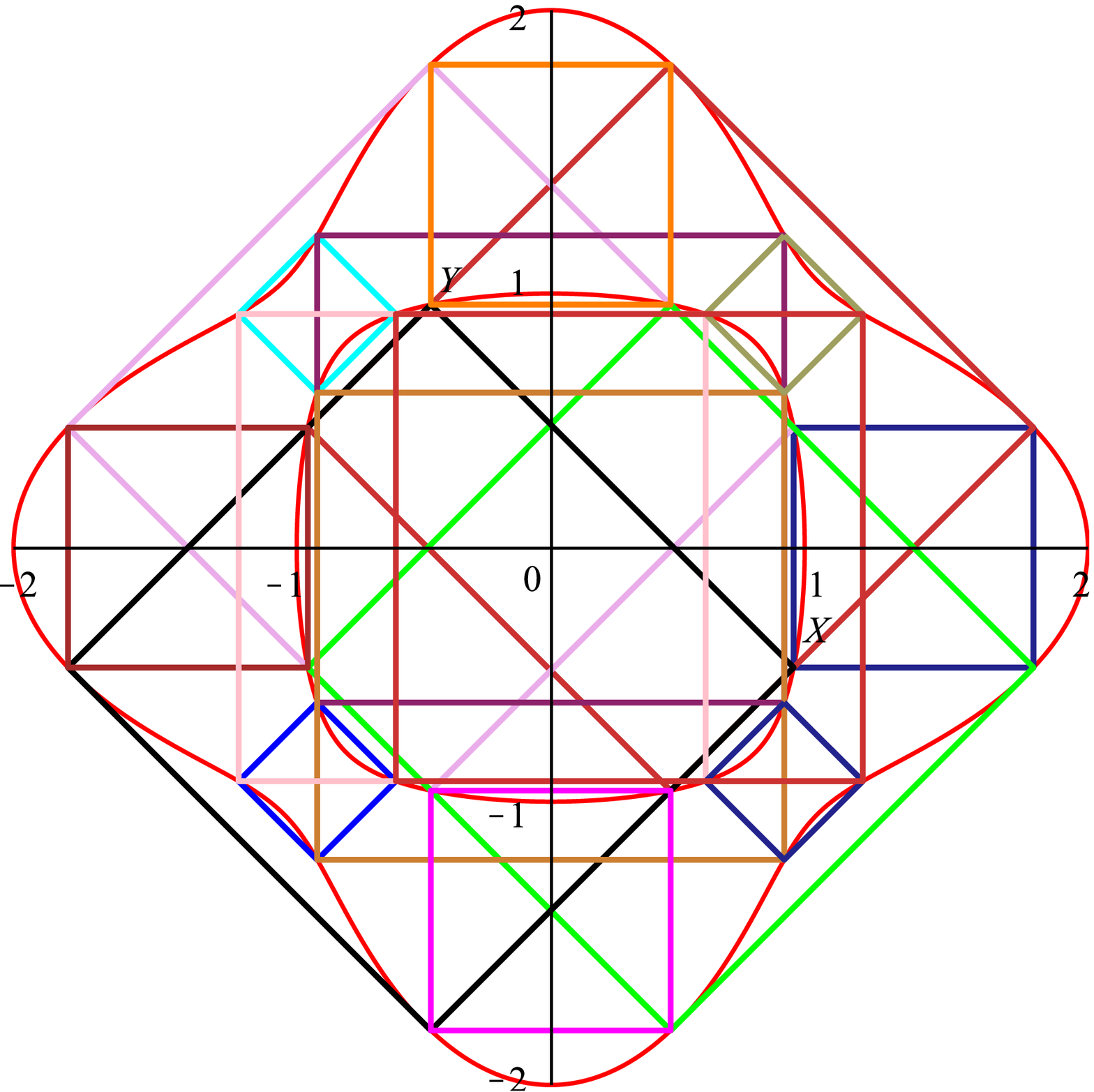}
   \caption{\inscribe{Sixteen}{45}}
    \label{fig:inscribed-zeroTwo-sixteen}
  \end{subfigure}
\begin{subfigure}[b]{0.45\textwidth}
   \includegraphics[width=\textwidth]{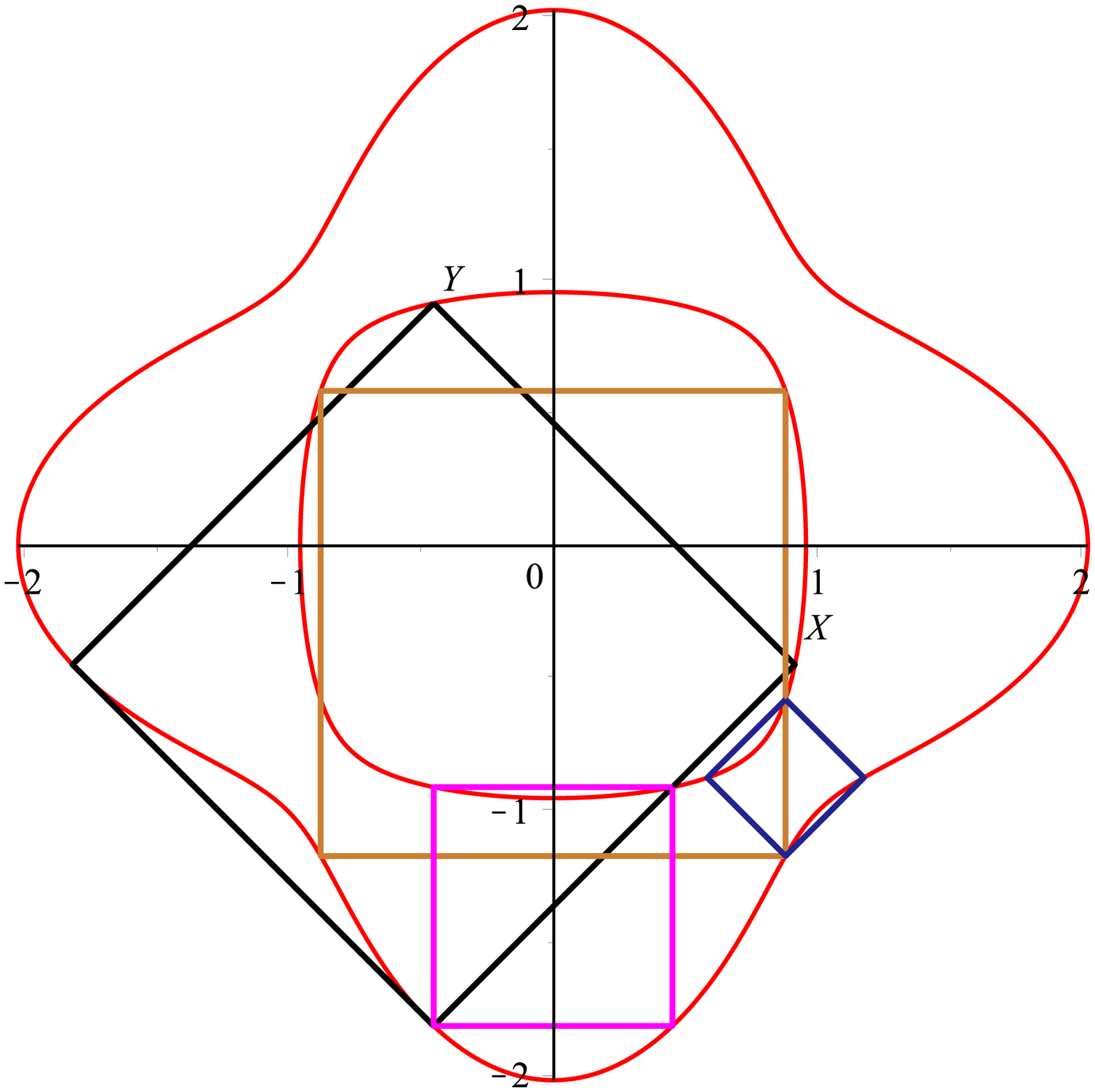}
   \caption{Up to rotational symmetry, four squares inscribed on \hyperlink{poly:f45}{$f_{45}$}}
  \end{subfigure}
  \caption{Two ovals inscribing an even number of squares.}
  \label{fig:inscribed-zeroTwo}
\end{figure}

\subsubsection{An oval and two lines inscribing an odd number of squares}

The curves in \Fref{fig:inscribed-twoOne} inscribe an odd number of squares:
three, five and seven.
% twoOne = select(allCurves, t -> (2, 1) == (t_3, t_4))
% twoOne / (k -> (length last k, k))
% twoOnePrime = sort oo / last
% forMapleSequence(oo / first, oo / last)
\begin{figure}[H]
  \begin{subfigure}[b]{0.5\linewidth}
    \includegraphics[width=\textwidth]{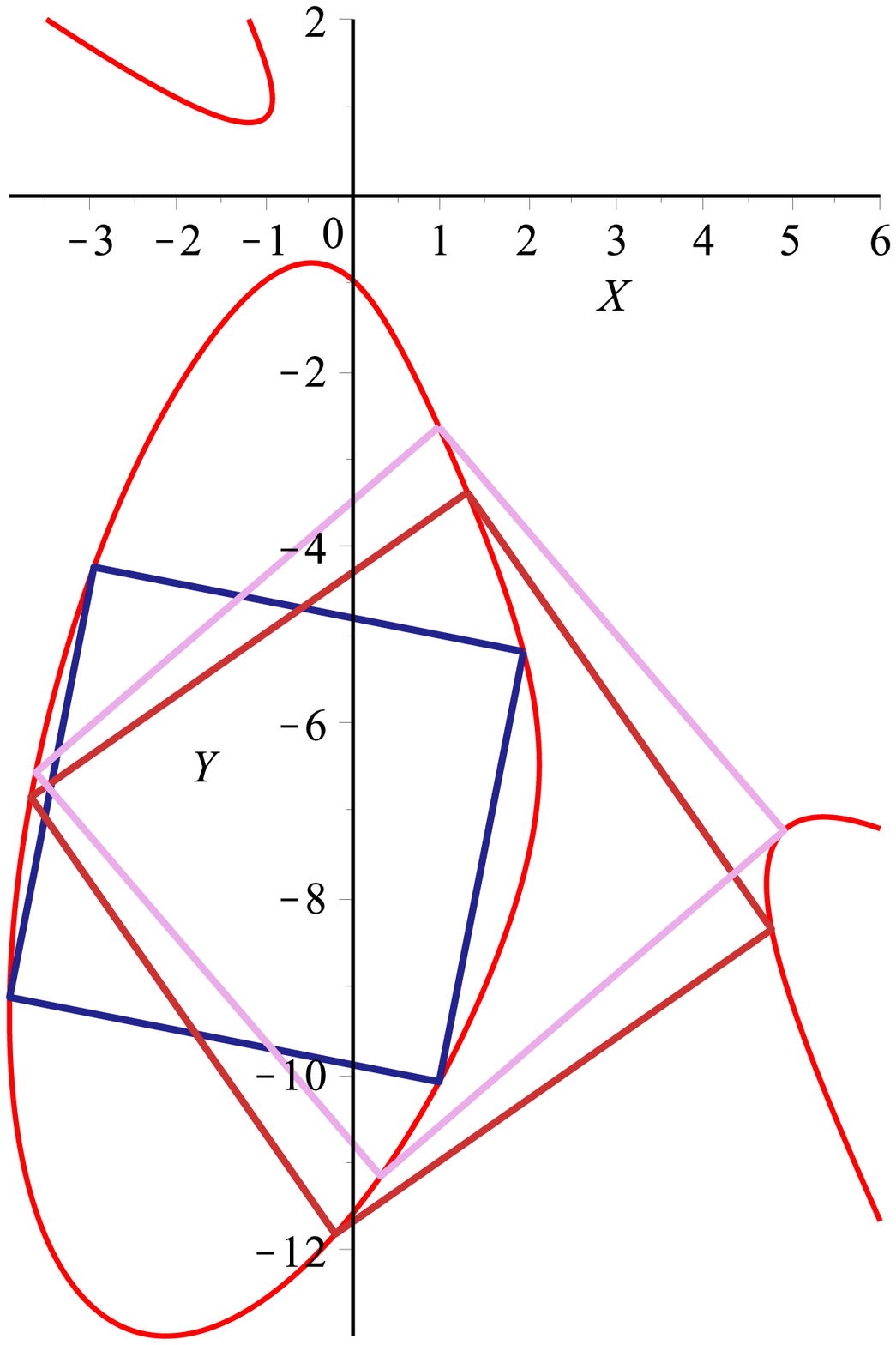}
    \caption{\inscribe{Three}{34}}
  \end{subfigure}
  \begin{subfigure}[b]{0.5\linewidth}
    \includegraphics[width=\textwidth]{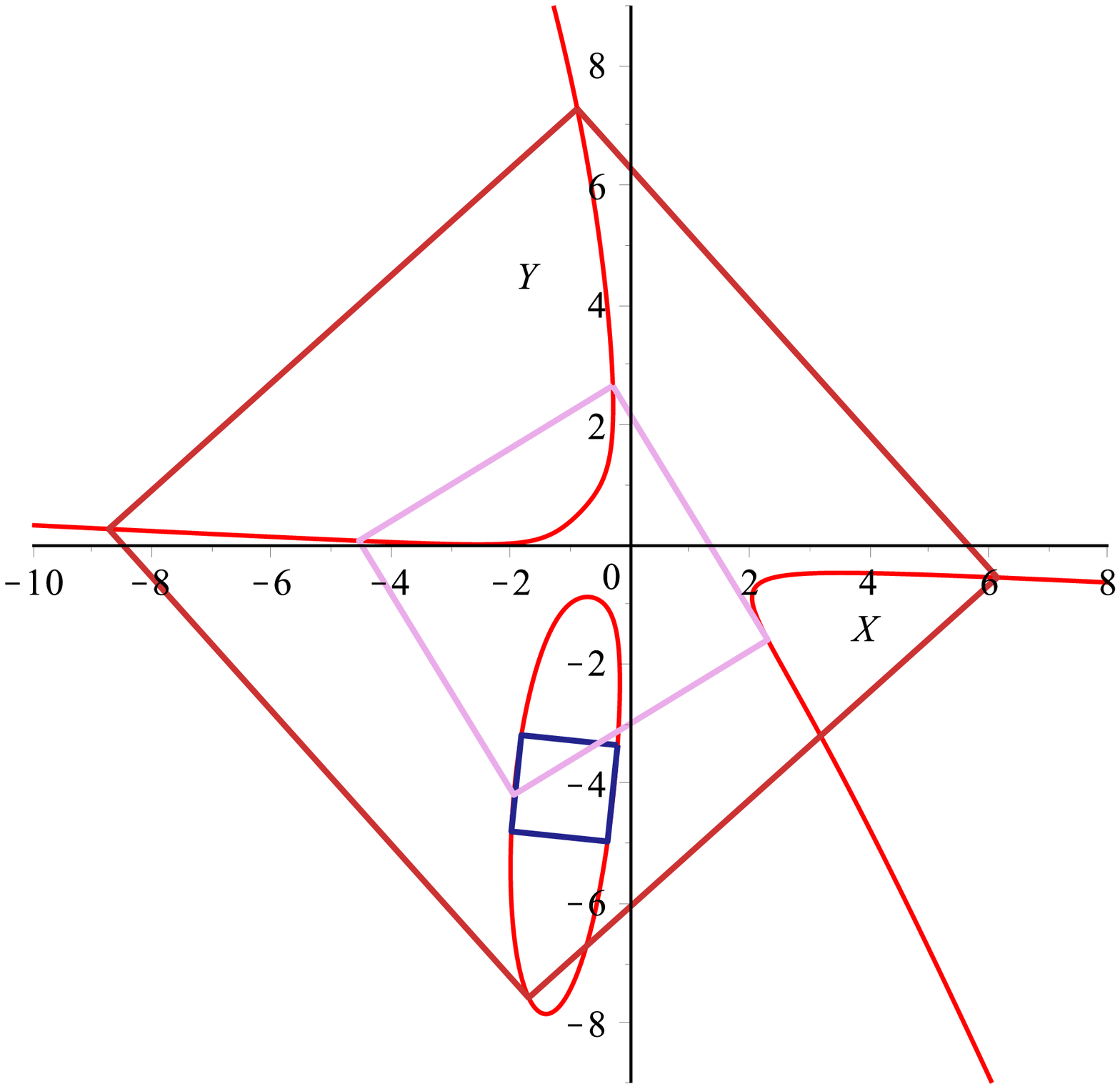}
    \caption{\inscribe{Three}{35}}
  \end{subfigure}
\end{figure}
\begin{figure}[H]
  \begin{subfigure}[b]{0.5\linewidth}
    \includegraphics[width=\textwidth]{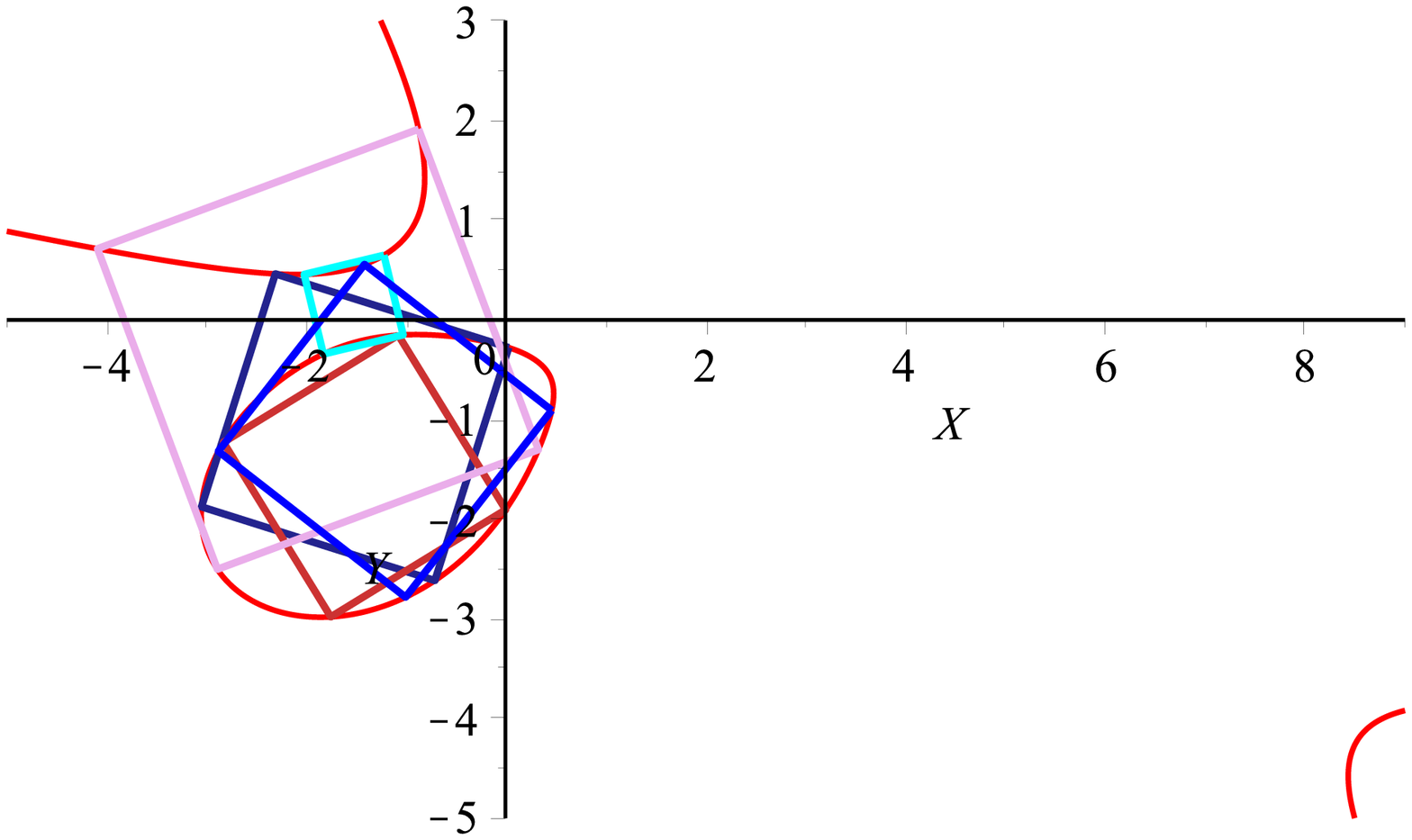}
    \caption{\inscribe{Five}{36}}
  \end{subfigure}
  \begin{subfigure}[b]{0.5\linewidth}
    \includegraphics[width=\textwidth]{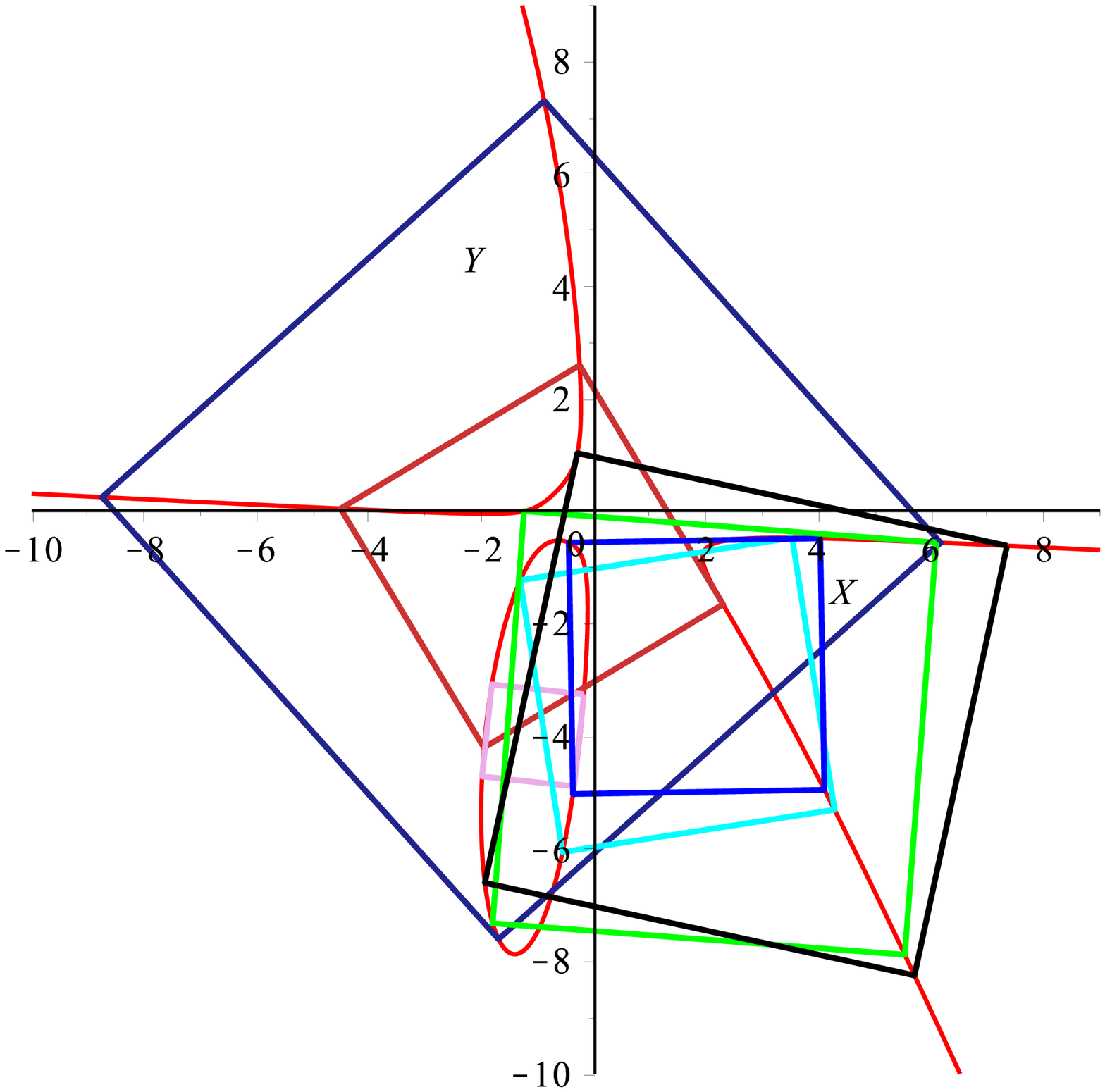}
    \caption{\inscribe{Seven}{37}}
  \end{subfigure}
\caption{An oval and two lines inscribing an odd number of squares.}
\label{fig:inscribed-twoOne}
\end{figure}

\subsection{Topological types of curves lacking a parity condition on
  the number of inscribed squares}
\label{sec:no-parity}

\subsubsection{Squares inscribed on one oval and one line}

The curves in \Fref{fig:inscribed-oneOne} inscribe one, two, three,
five, seven, nine and eleven squares. Note that the curve in Figure \ref{fig:debate} is reducible.

% oneOne = select(allCurves, t -> (1, 1) == (t_3, t_4))
% oneOne / (k -> (length last k, k))
% oneOnePrime = sort oo / last
% forMapleArray(oo / first, oo / last)
\begin{figure}[H]
\centering
  \begin{subfigure}[b]{0.3\linewidth}
    \includegraphics[width=\textwidth]{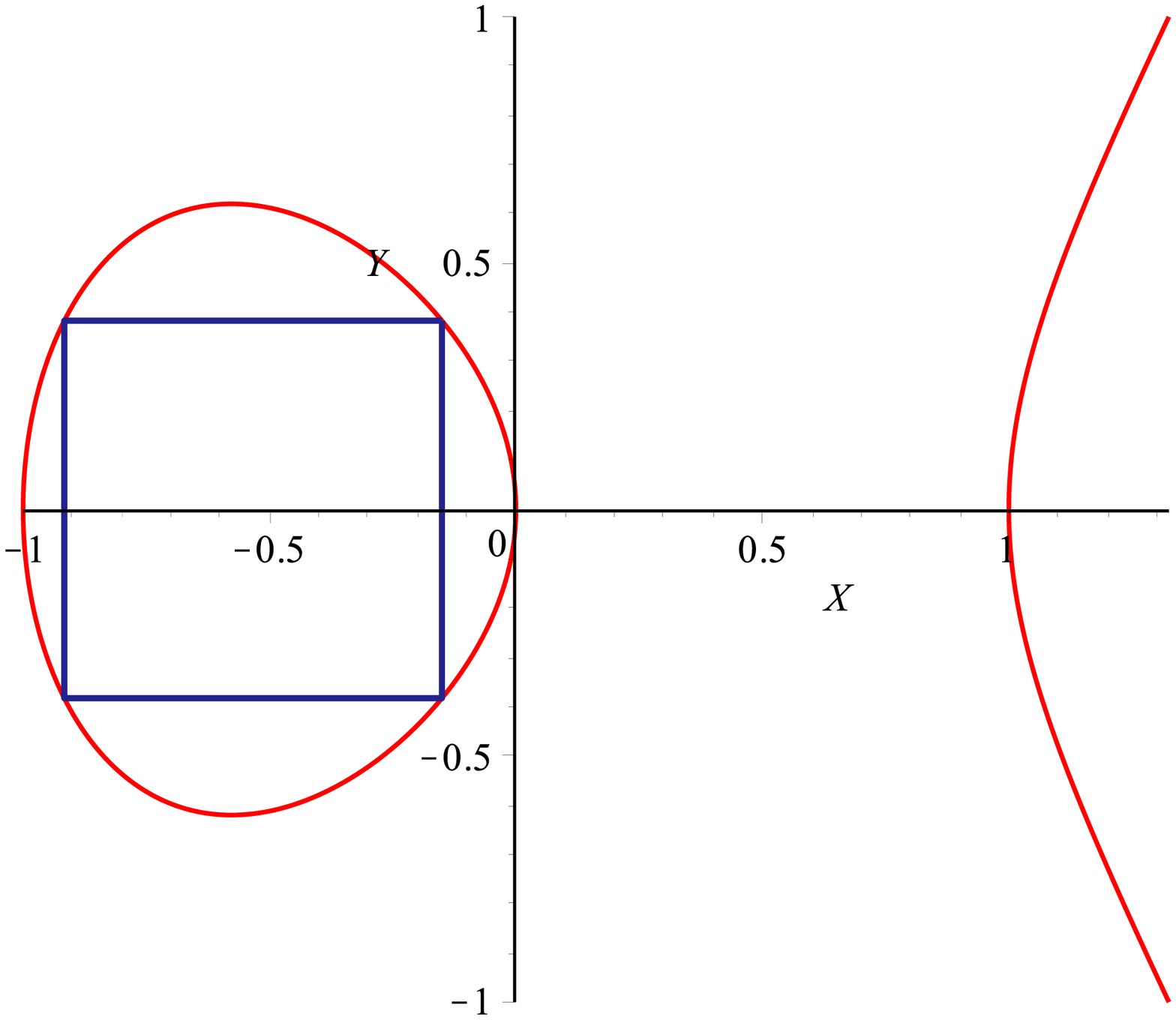}
    \caption{\inscribeOne{One}{18}}
  \end{subfigure}
  \begin{subfigure}[b]{0.3\linewidth}
    \includegraphics[width=\textwidth]{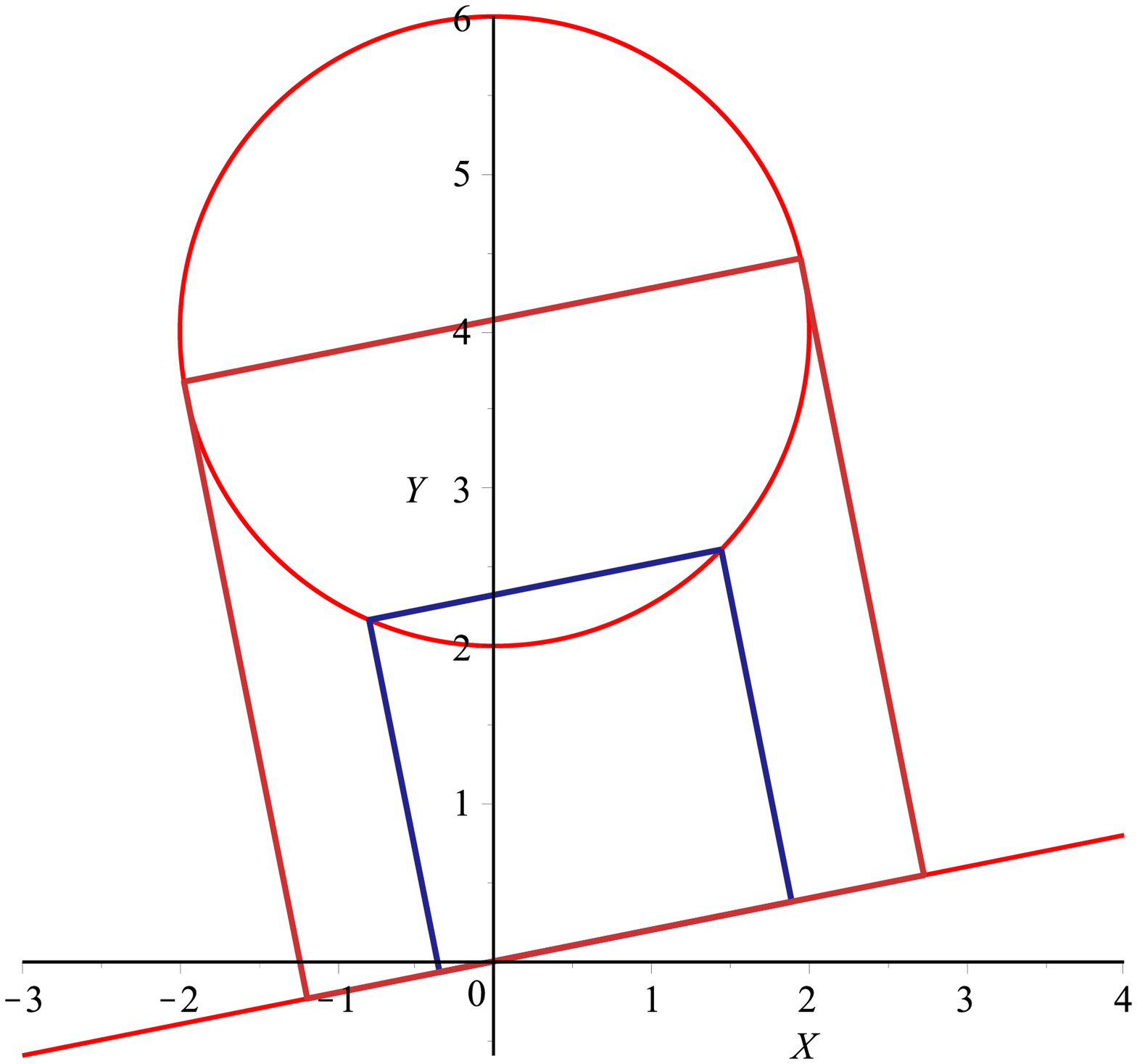}
    \caption{\inscribe{Two}{19}}
    \label{fig:debate}
   \end{subfigure}
  \begin{subfigure}[b]{0.3\linewidth}
    \includegraphics[width=\textwidth]{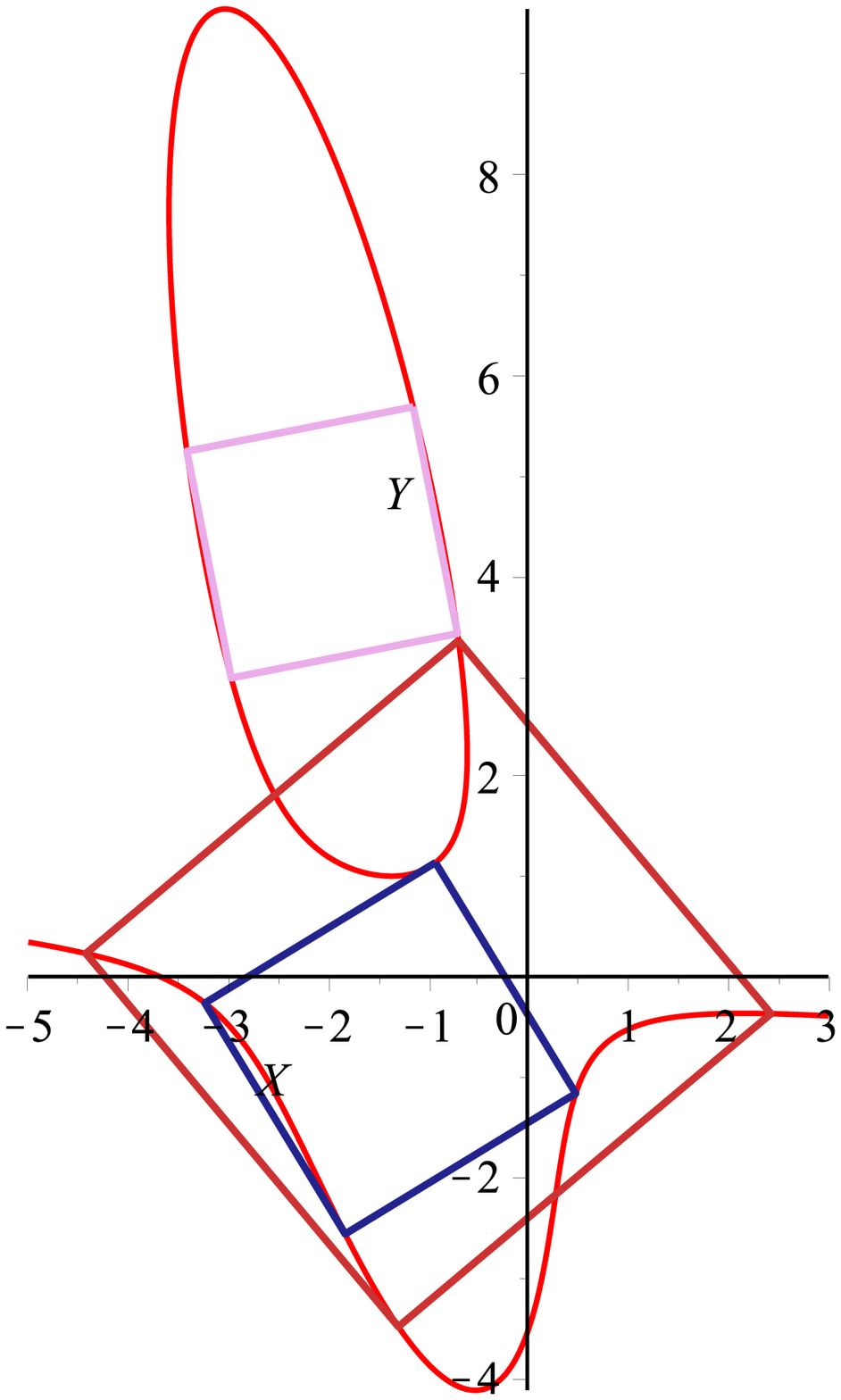}
    \caption{\inscribe{Three}{20}}
  \end{subfigure}
  \begin{subfigure}[b]{0.3\linewidth}
    \includegraphics[width=\textwidth]{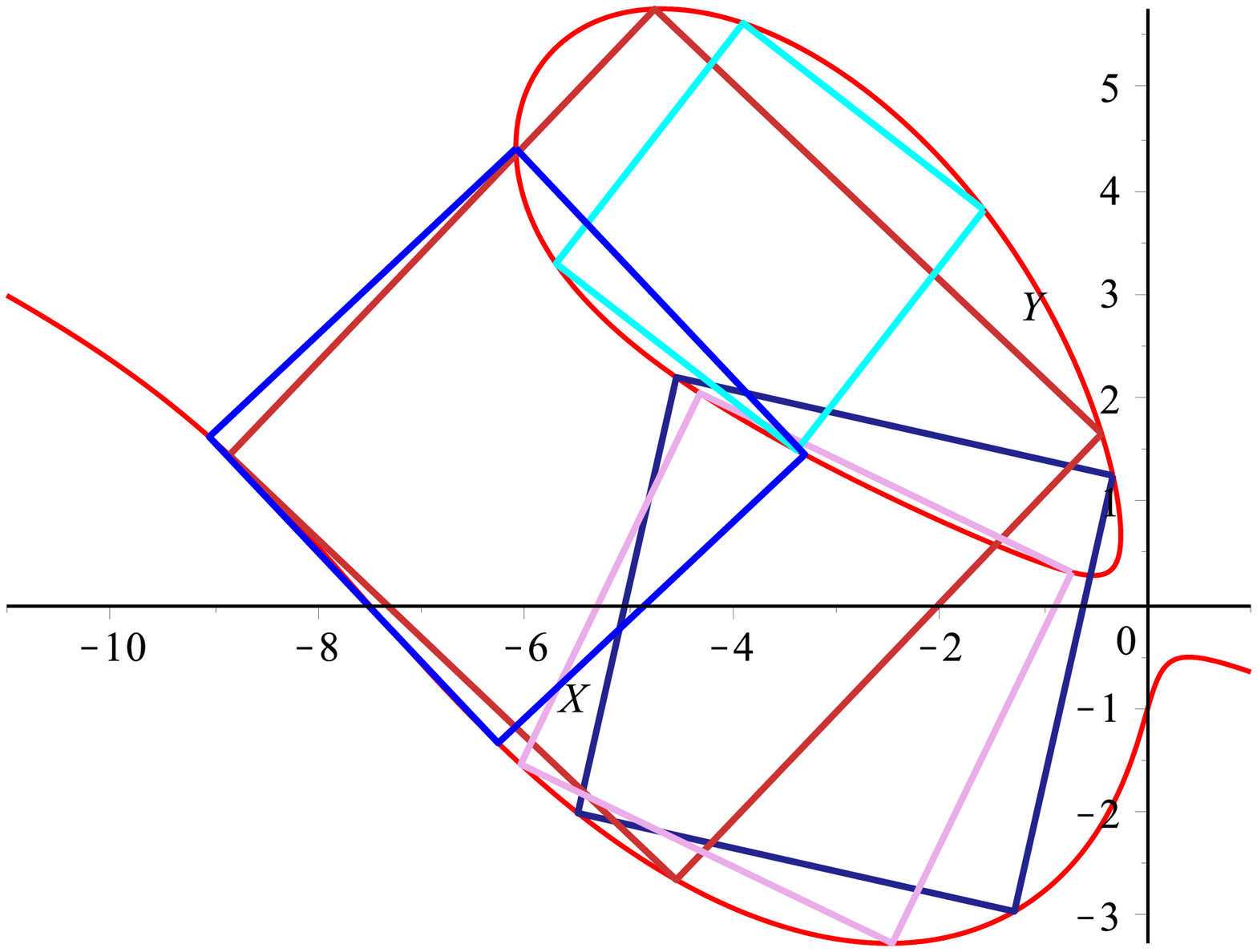}
    \caption{\inscribe{Five}{21}}
  \end{subfigure}
  \begin{subfigure}[b]{0.3\linewidth}
    \includegraphics[width=\textwidth]{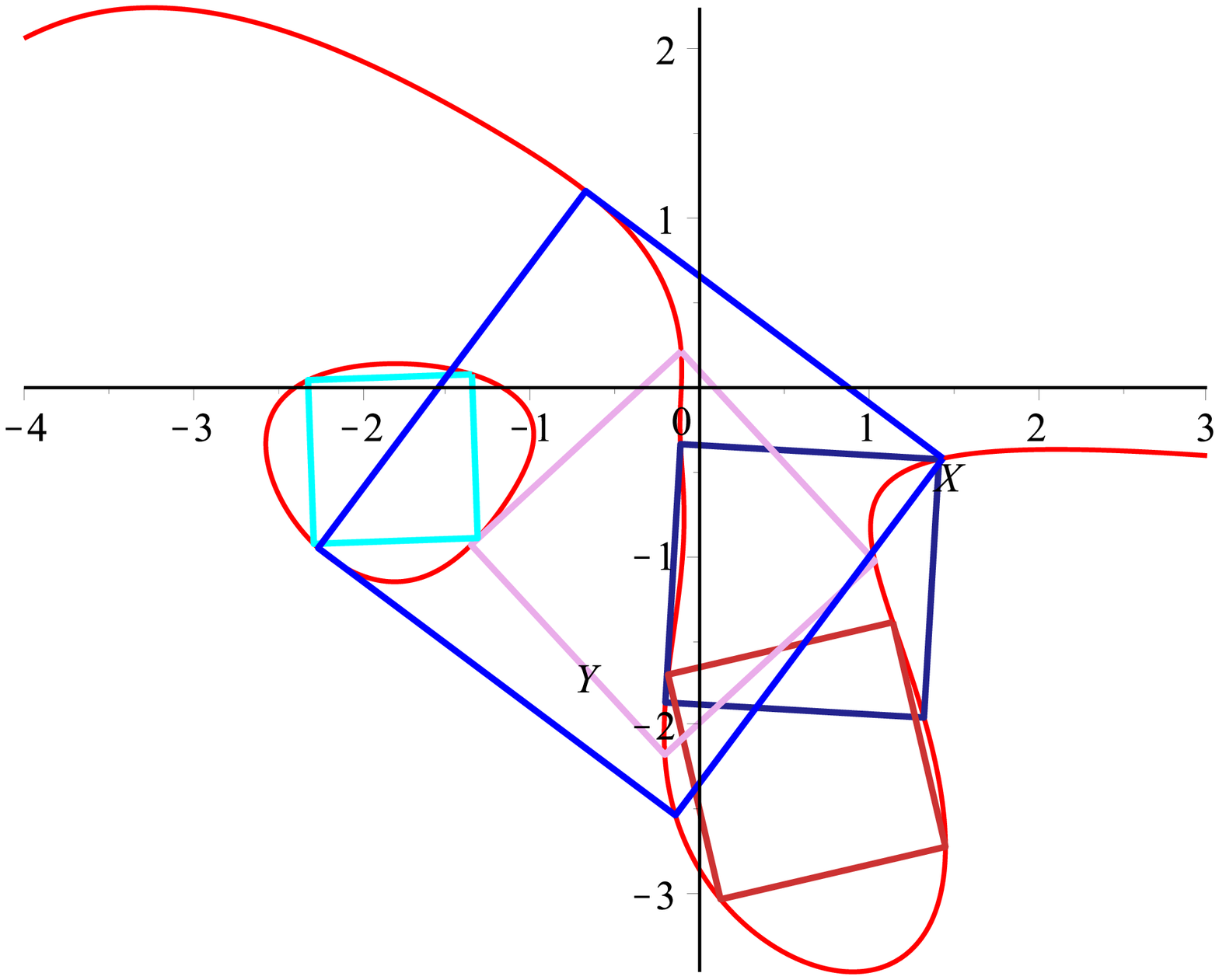}
    \caption{\inscribe{Five}{22}}
  \end{subfigure}
  \begin{subfigure}[b]{0.3\linewidth}
    \includegraphics[width=\textwidth]{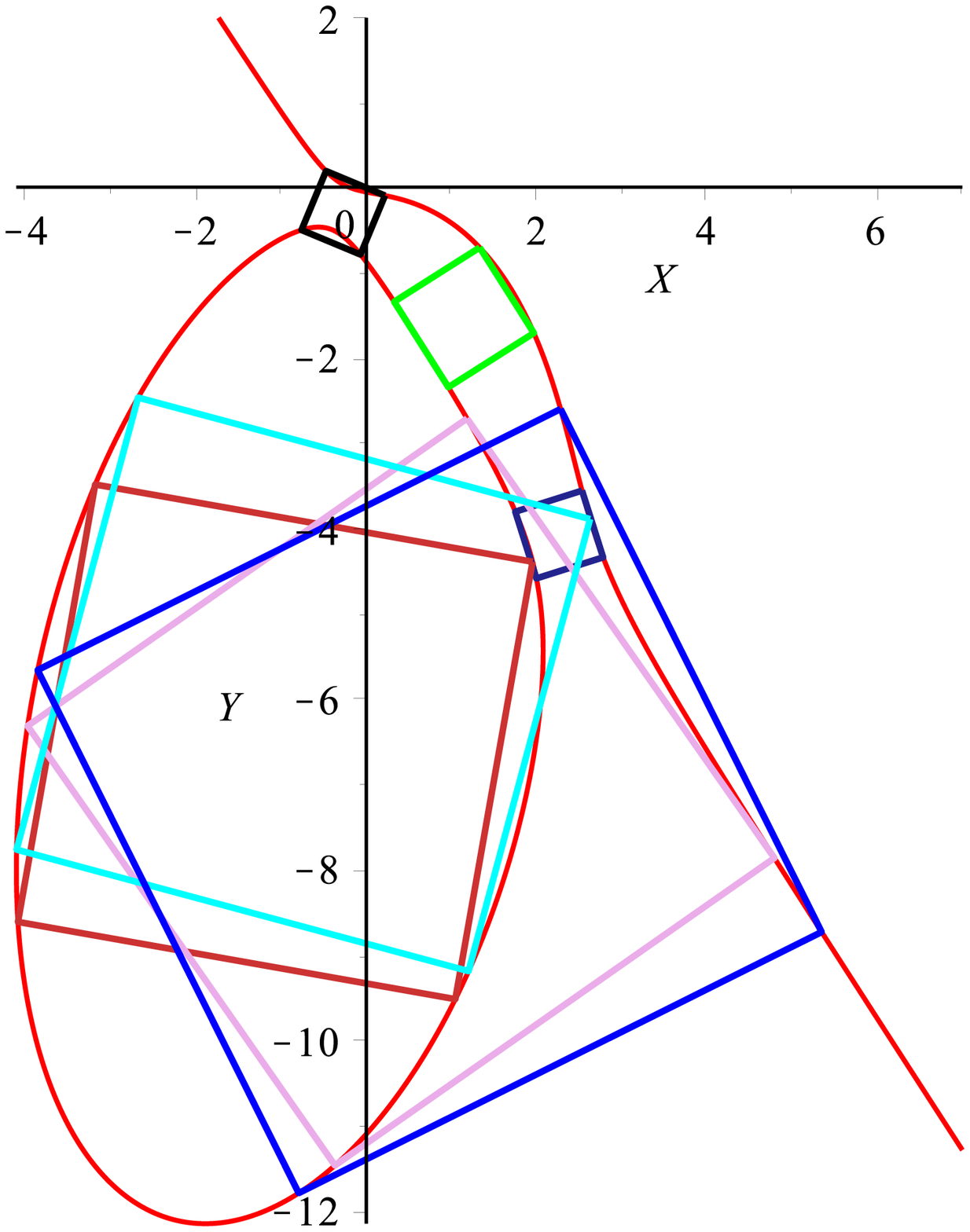}
    \caption{\inscribe{Seven}{23}}
  \end{subfigure}
  \begin{subfigure}[b]{0.3\linewidth}
    \includegraphics[width=\textwidth]{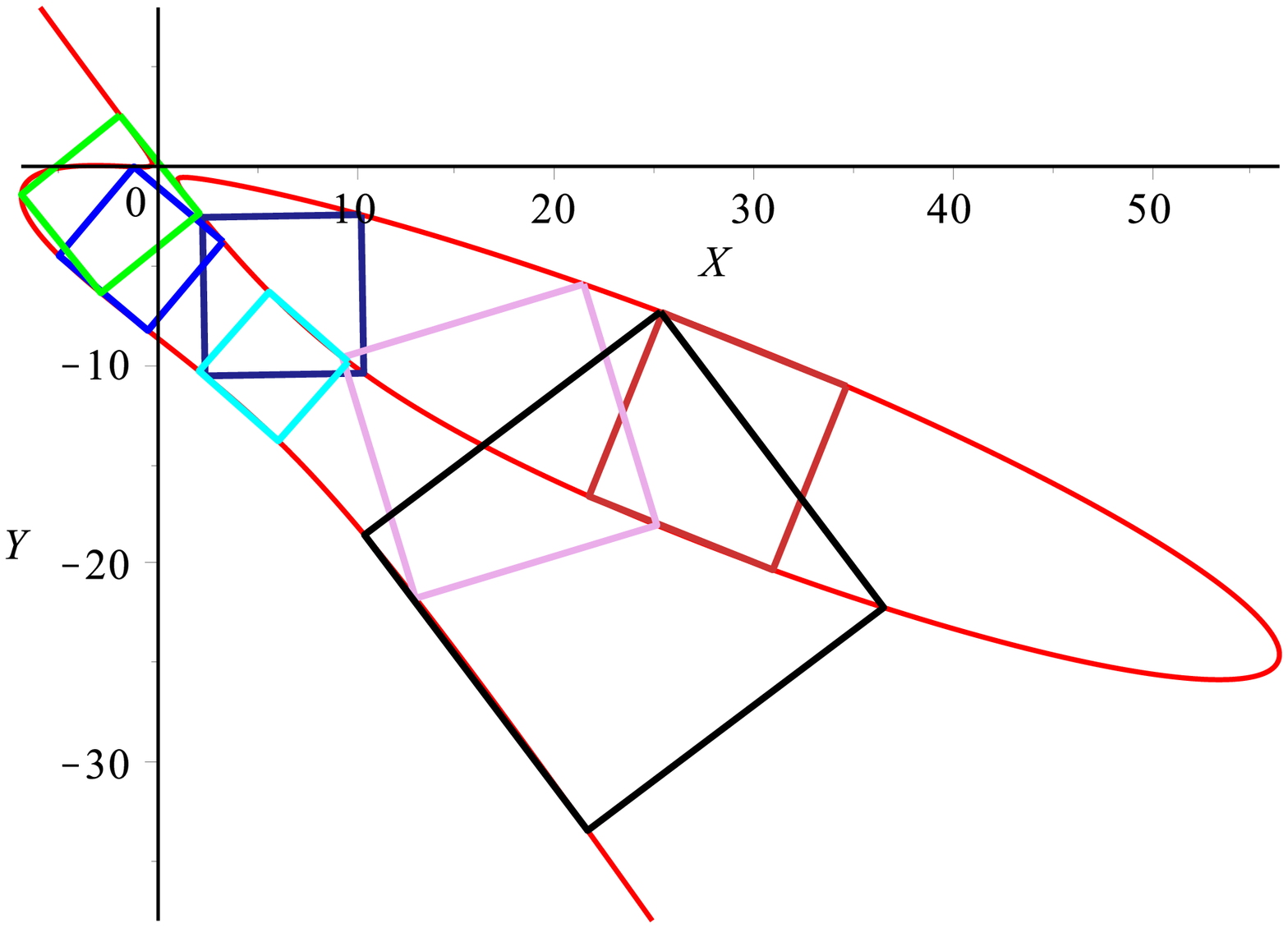}
    \caption{\inscribe{Seven}{24}}
  \end{subfigure}
  \begin{subfigure}[b]{0.3\linewidth}
    \includegraphics[width=\textwidth]{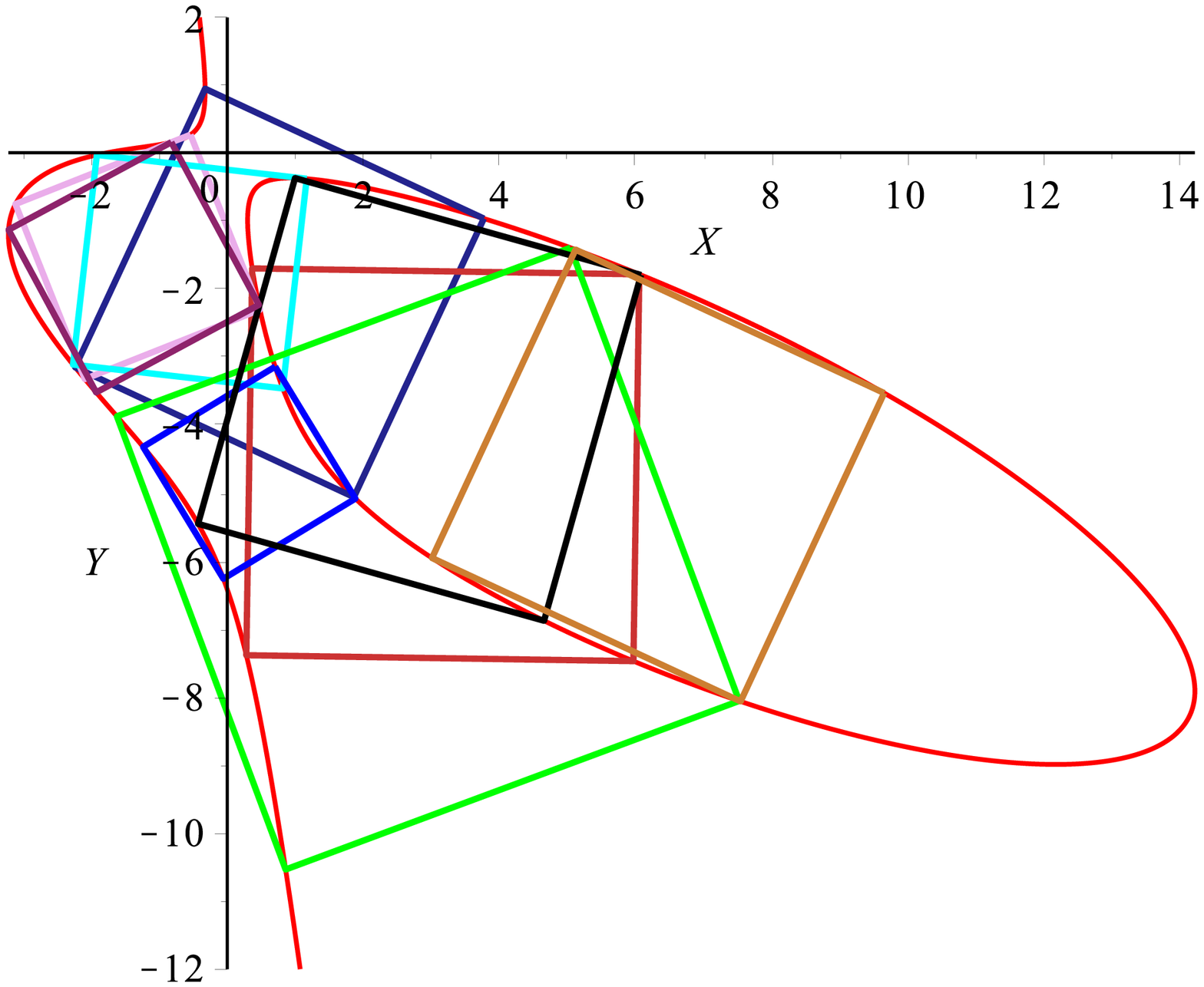}
    \caption{\inscribe{Nine}{25}}
  \end{subfigure}
  \begin{subfigure}[b]{0.3\linewidth}
    \includegraphics[width=\textwidth]{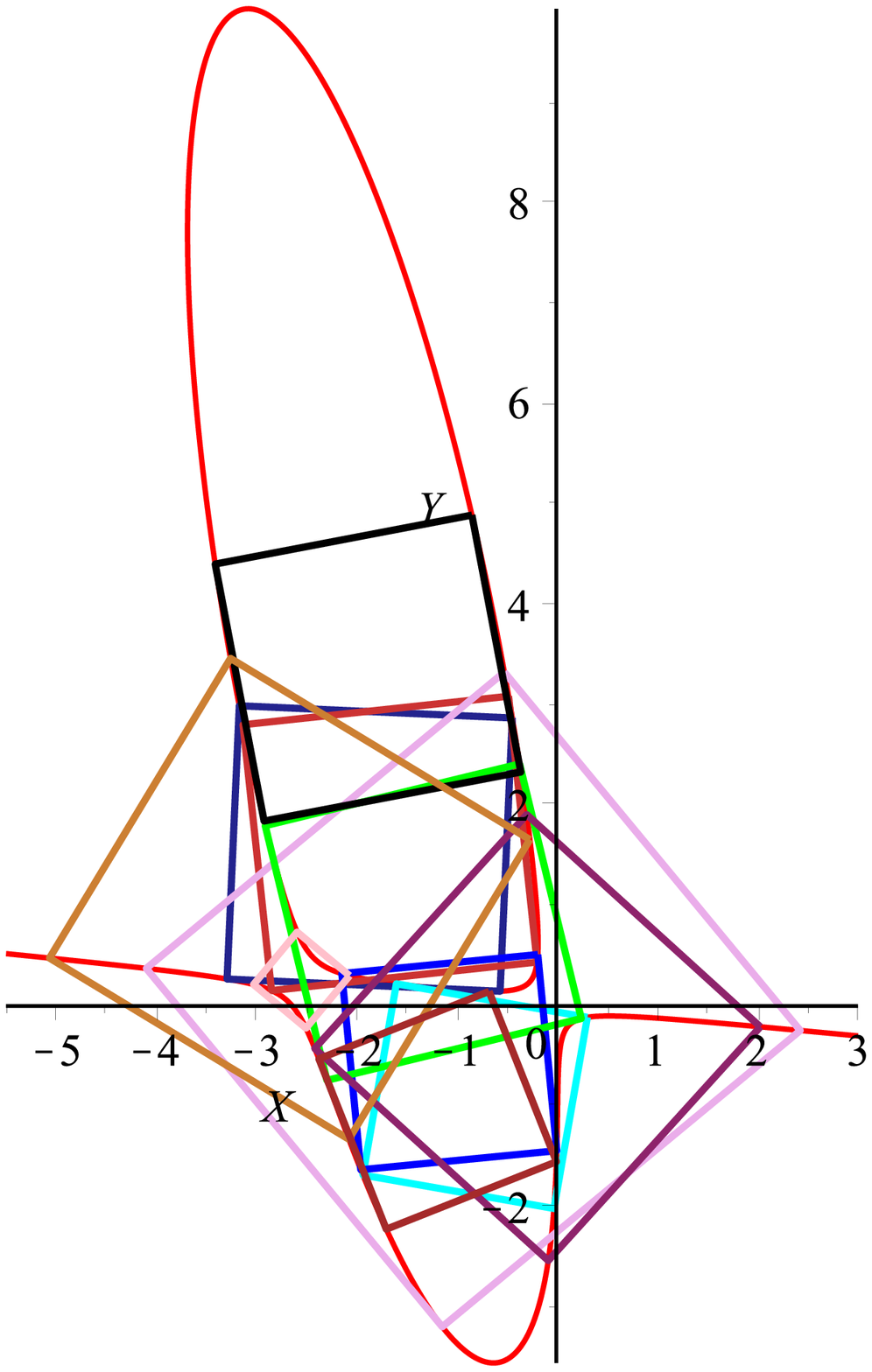}
    \caption{\inscribe{Eleven}{26}}
  \end{subfigure}
  \caption{Squares inscribed on an oval and a line.}
  \label{fig:inscribed-oneOne}
\end{figure}

\subsubsection{Squares inscribed on two lines}

% twoZero = select(allCurves, t -> (2, 0) == (t_3, t_4))
% twoZero / (k -> (length last k, k))
% twoZeroPrime = sort oo / last
% forMapleArray(oo / first, oo / last)

The curves in \Fref{fig:inscribed-twoZero} inscribe one, four, eight,
nine and eleven squares.
\begin{figure}[H]
\centering
  \begin{subfigure}[b]{0.45\linewidth}
    \includegraphics[width=\textwidth]{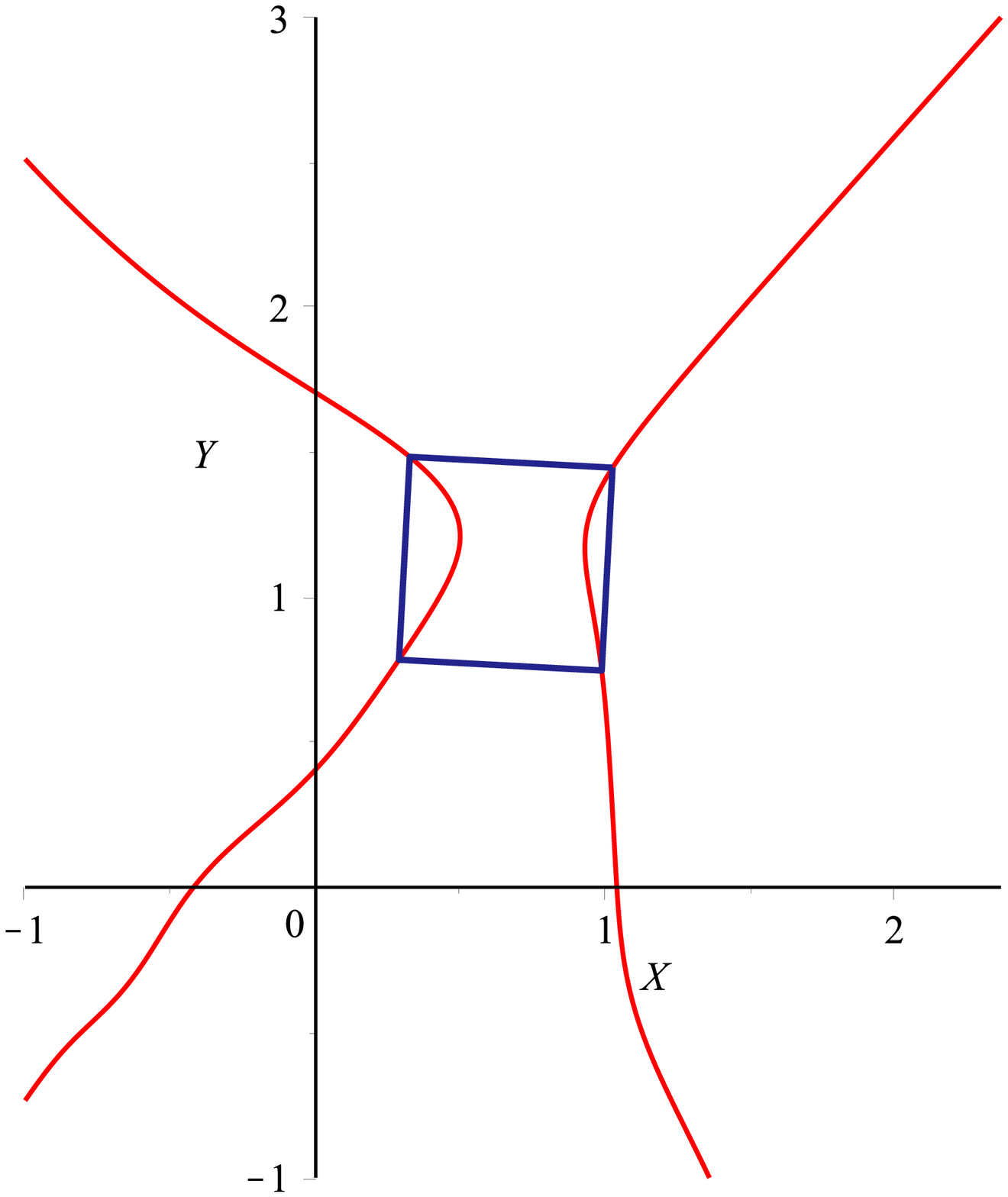}
    \caption{\inscribeOne{One}{13}}
  \end{subfigure}
  \begin{subfigure}[b]{0.45\linewidth}
    \includegraphics[width=\textwidth]{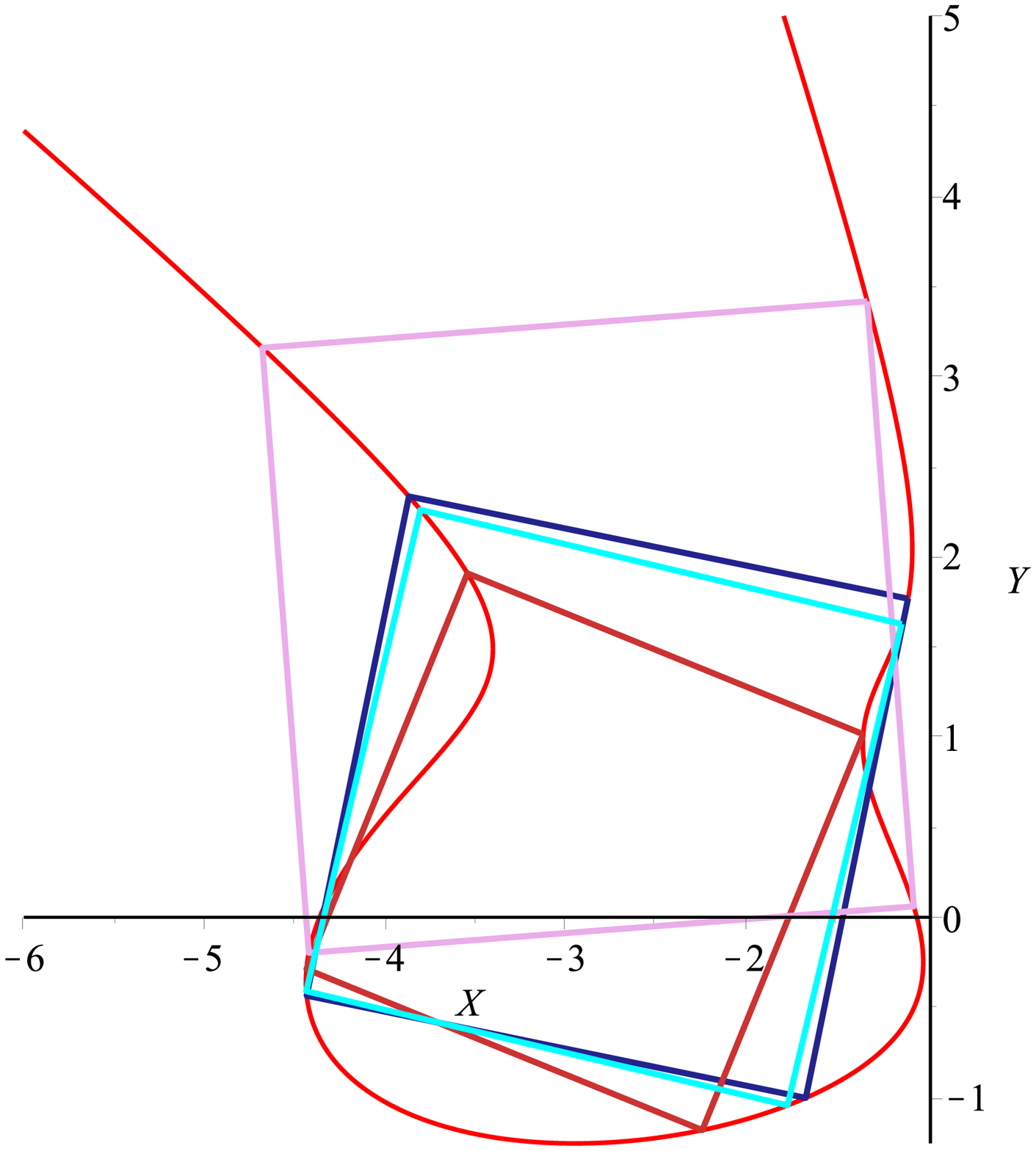}
    \caption{\inscribe{Four}{14}}
  \end{subfigure}
  \begin{subfigure}[b]{0.45\linewidth}
    \includegraphics[width=\textwidth]{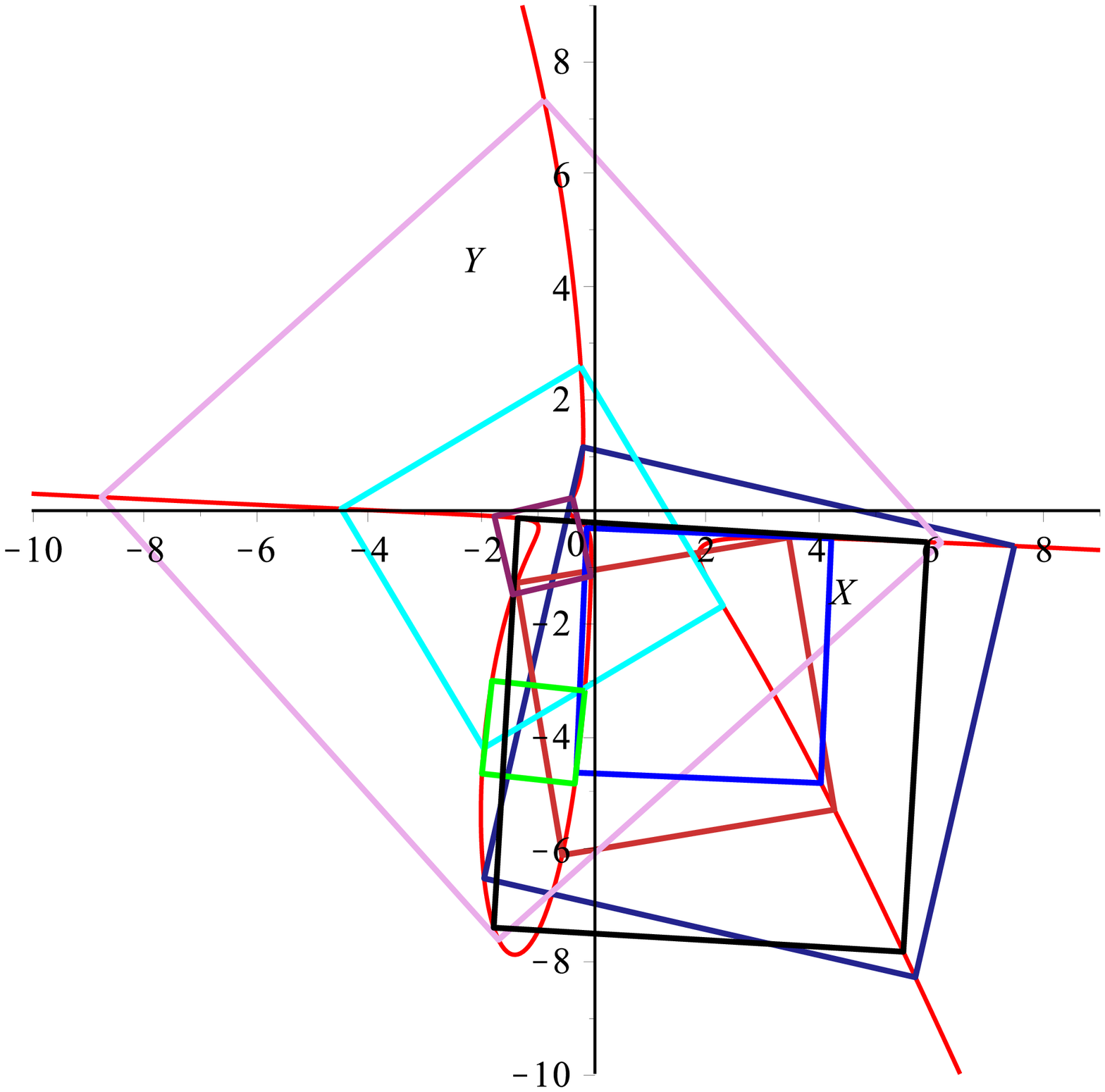}
    \caption{\inscribe{Eight}{15}}
  \end{subfigure}
  \begin{subfigure}[b]{0.45\linewidth}
    \includegraphics[width=\textwidth]{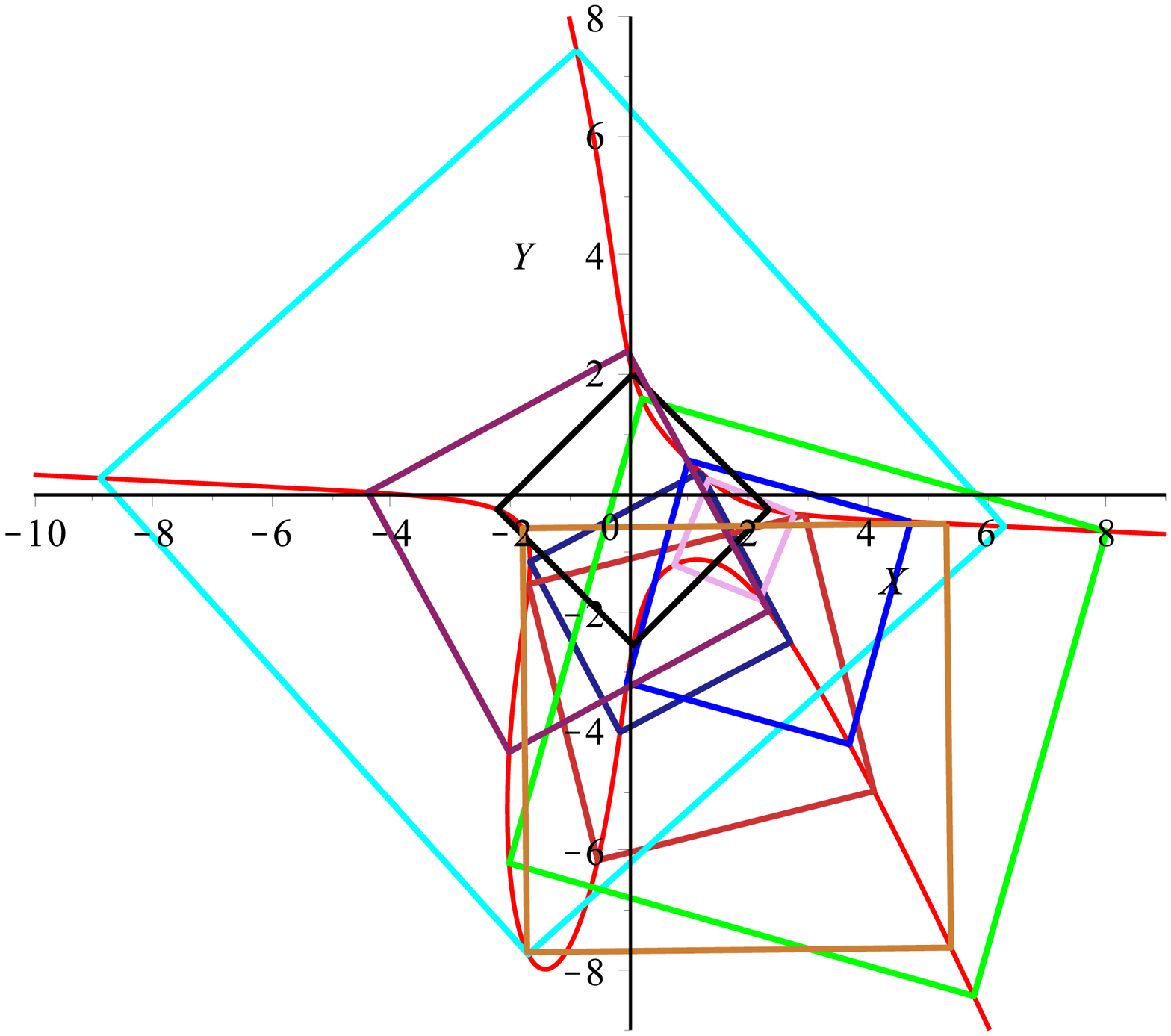}
    \caption{\inscribe{Nine}{16}}
  \end{subfigure}
  \begin{subfigure}[b]{0.45\linewidth}
    \includegraphics[width=\textwidth]{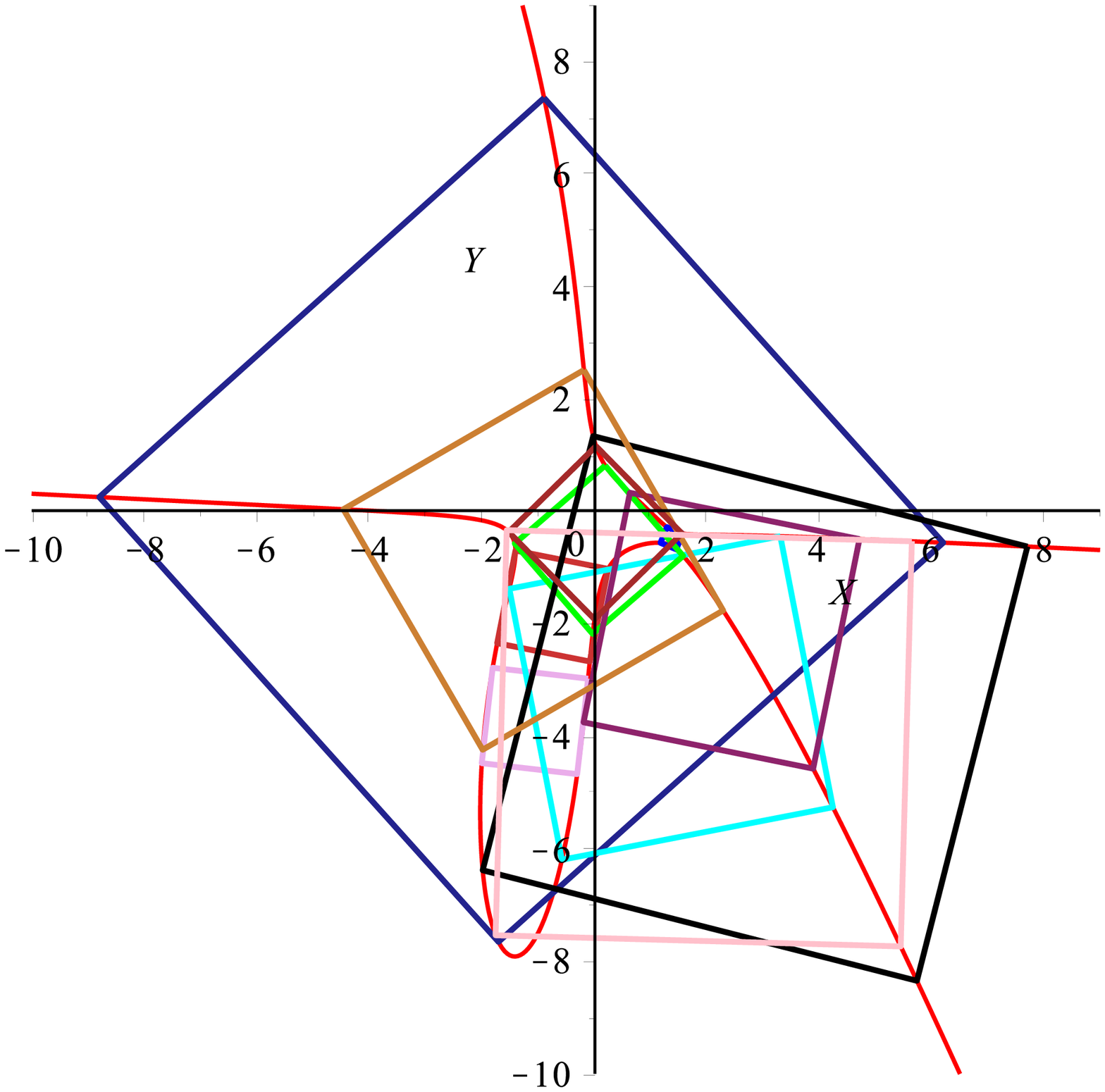}
    \caption{\inscribe{Eleven}{17}}
  \end{subfigure}
  \caption{Squares inscribed on two lines.}
  \label{fig:inscribed-twoZero}
\end{figure}

\subsubsection{Squares inscribed on three lines}
% select(allCurves, t -> (3, 0) == (t_3, t_4))
% forMapleArray(oo / first, oo / last)

The curves in \Fref{fig:inscribed-threeZero} inscribe one, four,
seven, eight, ten and eleven squares.  \Fref{fig:twelve-awesome}
depicts a third degree curve consisting of three lines inscribing the
maximal number of twelve squares.

\begin{figure}[H]
\centering
  \begin{subfigure}[b]{0.3\linewidth}
    \includegraphics[width=\textwidth]{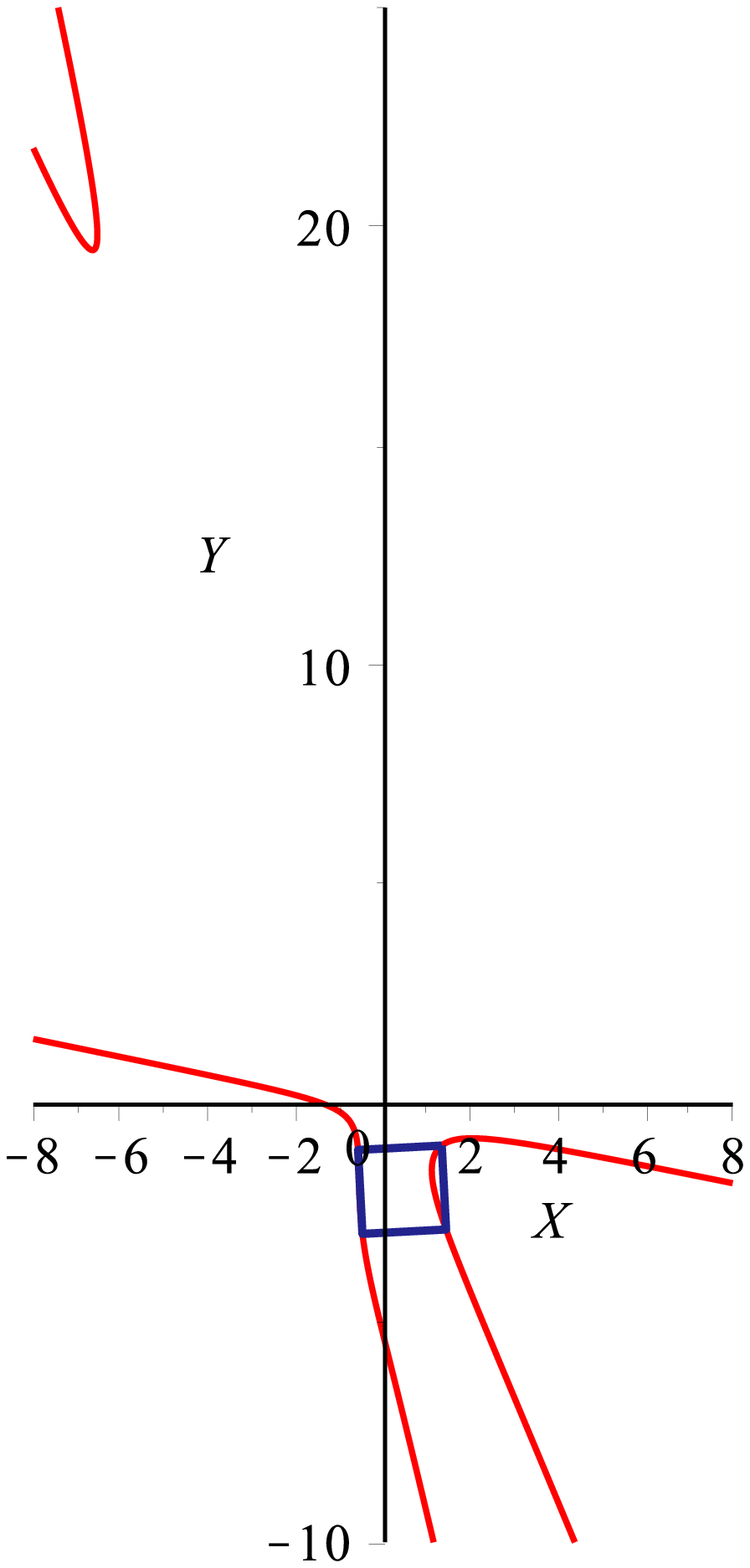}
    \caption{\inscribeOne{One}{8}}
  \end{subfigure}
  \begin{subfigure}[b]{0.3\linewidth}
    \includegraphics[width=\textwidth]{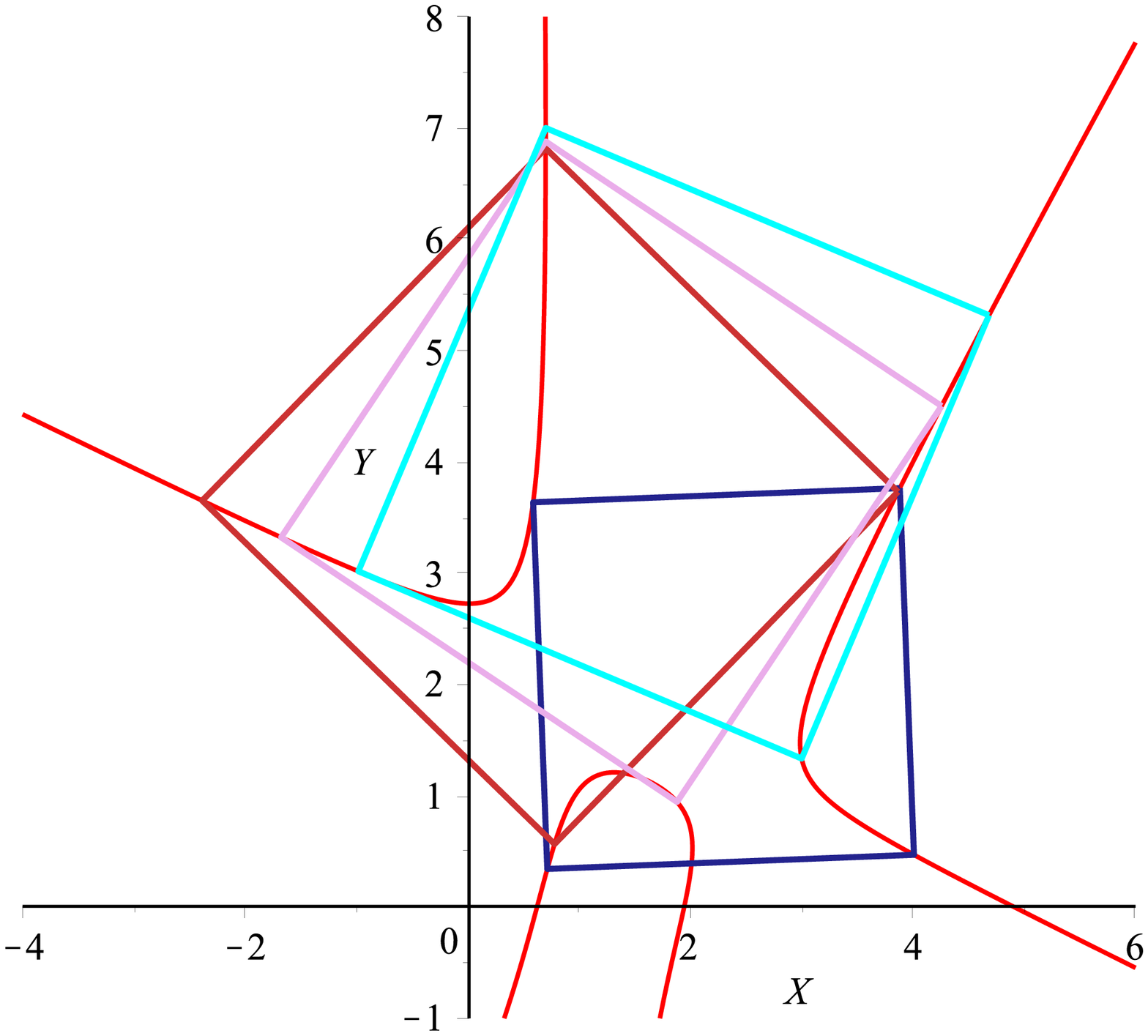}
    \caption{\inscribe{Four}{9}}
  \end{subfigure}
  \begin{subfigure}[b]{0.3\linewidth}
    \includegraphics[width=\textwidth]{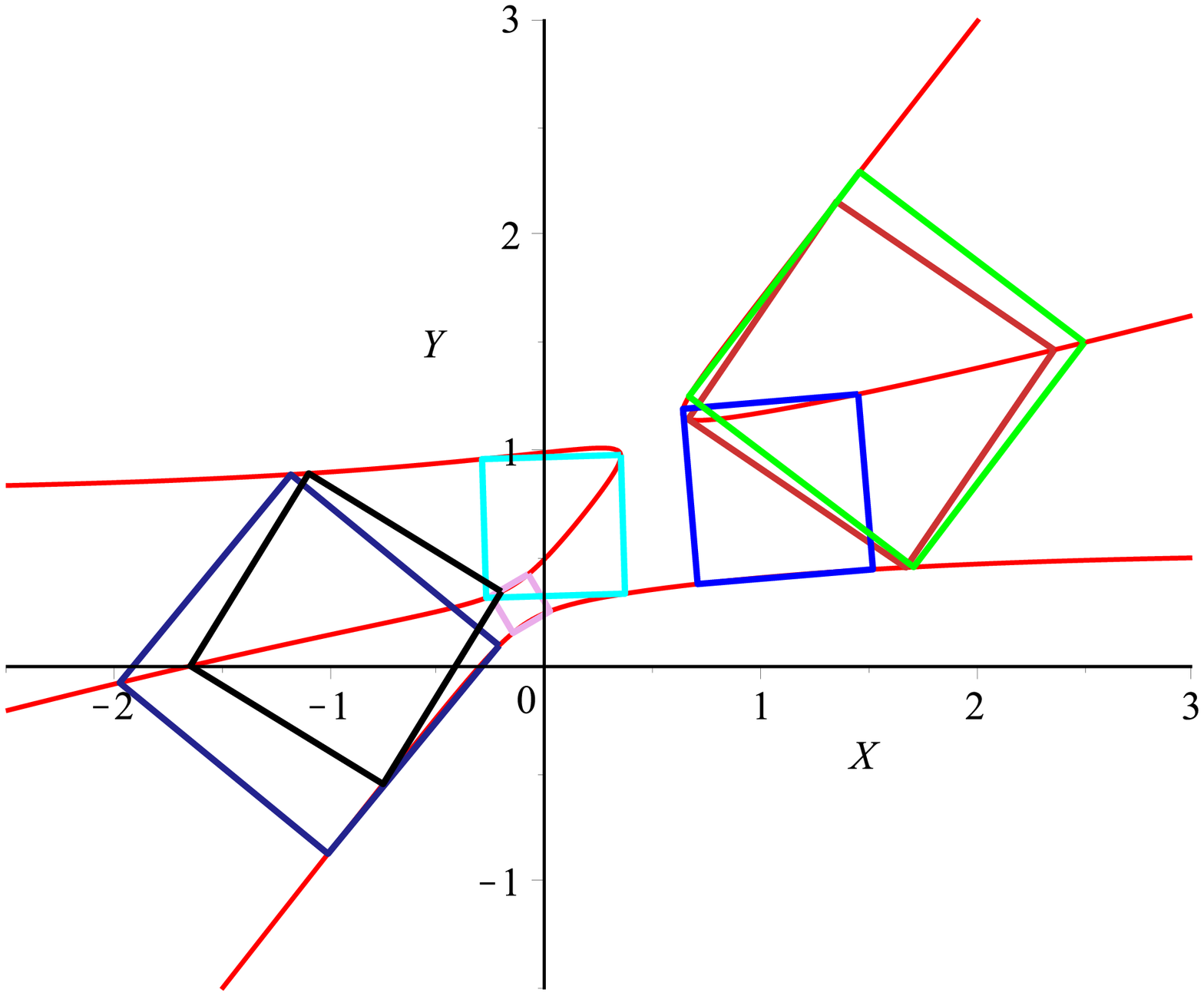}
   \caption{\inscribe{Seven}{10}}
  \end{subfigure}
  \begin{subfigure}[b]{0.3\linewidth}
    \includegraphics[width=\textwidth]{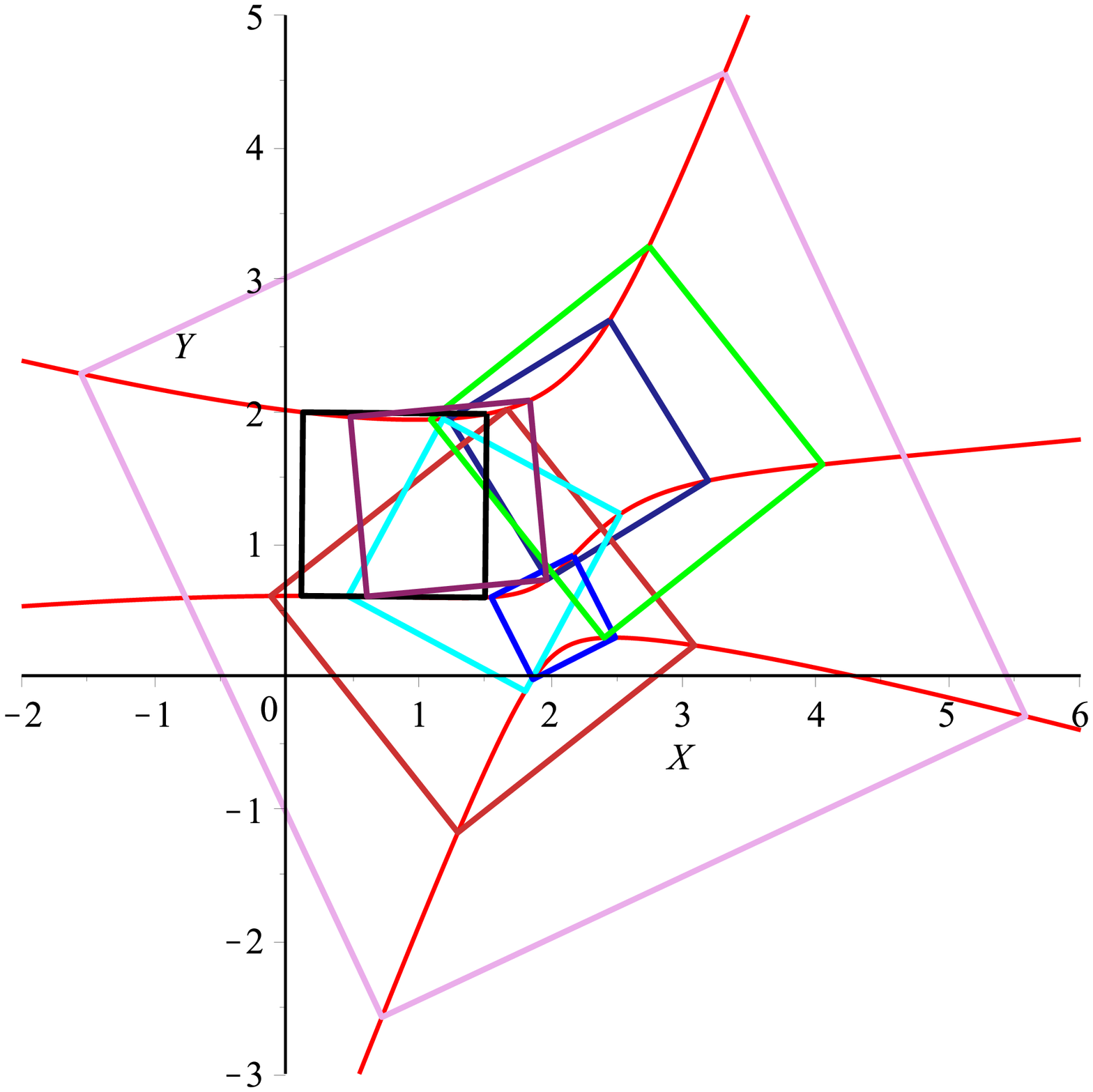}
   \caption{\inscribe{Eight}{11}}
  \end{subfigure}
  \begin{subfigure}[b]{0.3\linewidth}
    \includegraphics[width=\textwidth]{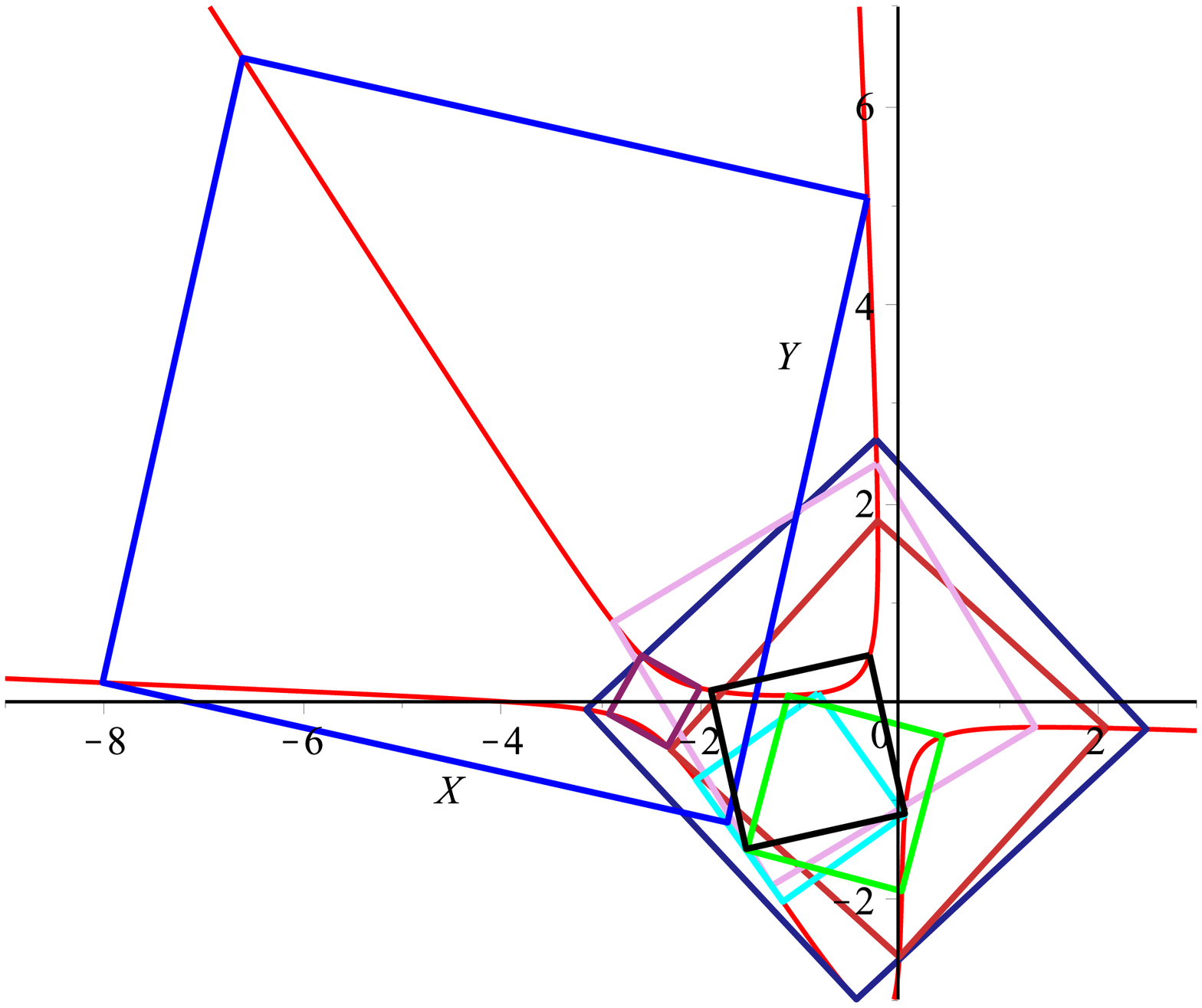}
    \caption{\inscribe{Eight}{38}}
  \end{subfigure}
  \begin{subfigure}[b]{0.3\linewidth}
    \includegraphics[width=\textwidth]{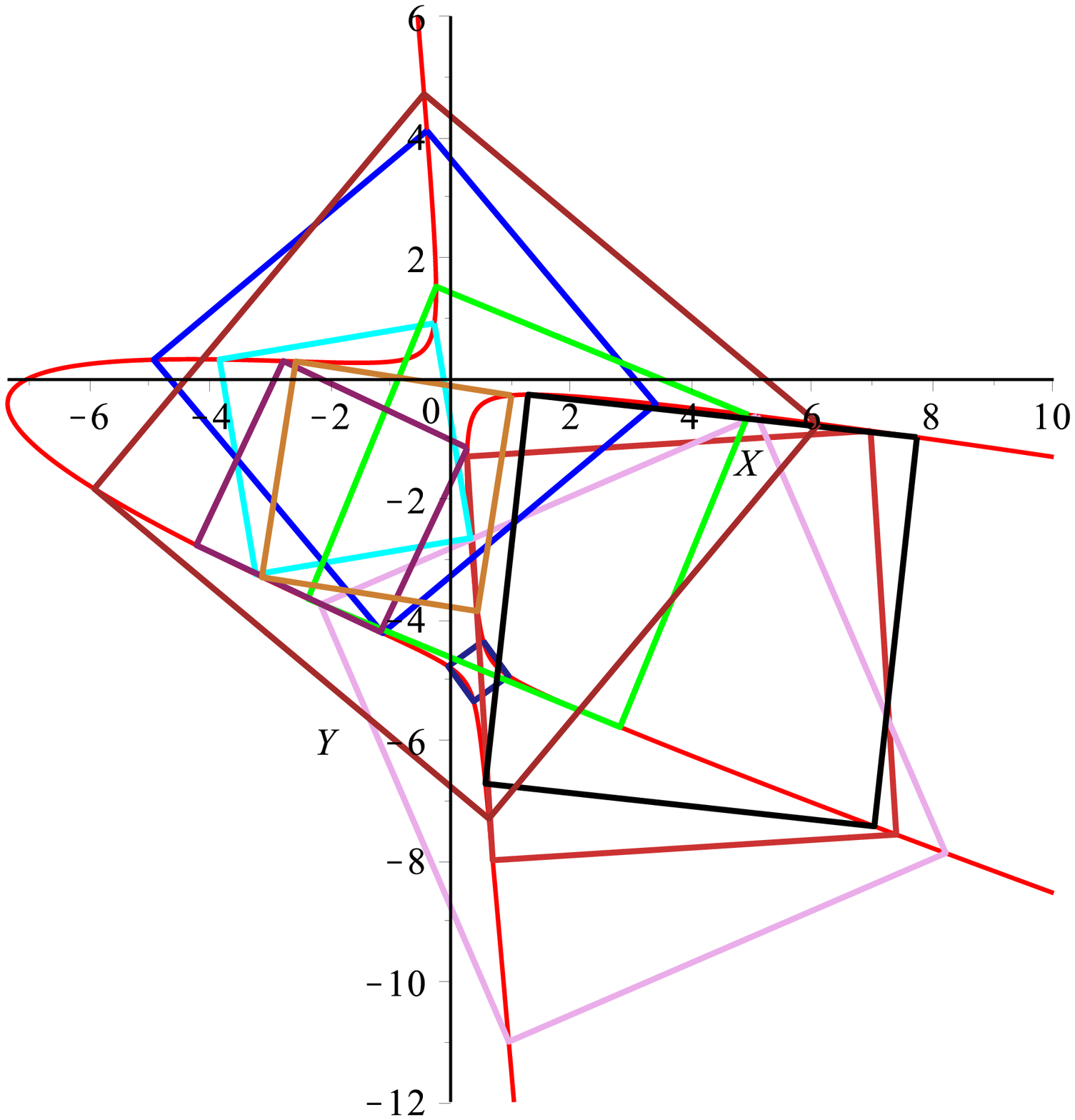}
   \caption{\inscribe{Ten}{39}}
  \end{subfigure}
  \begin{subfigure}[b]{0.3\linewidth}
    \includegraphics[width=\textwidth]{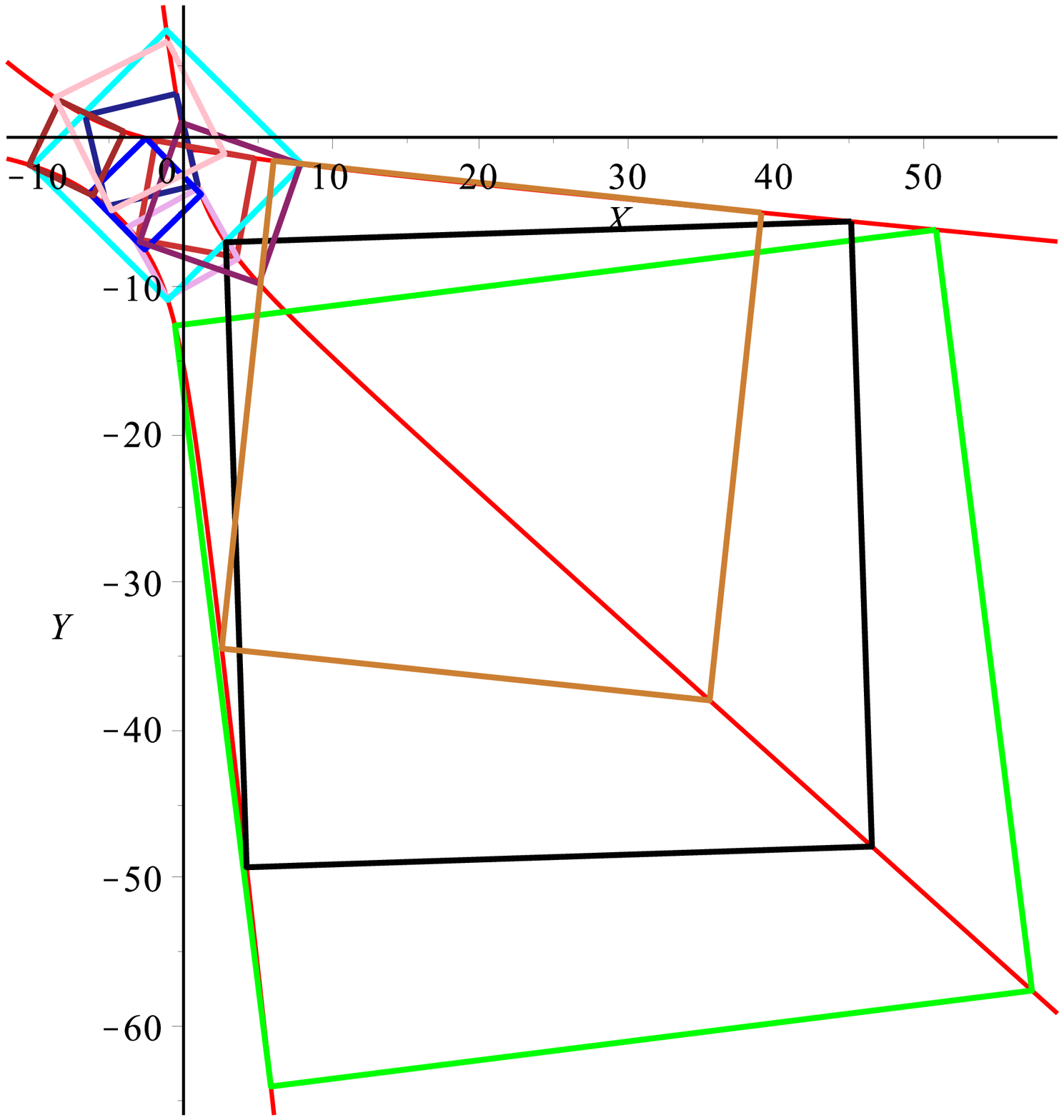}
    \caption{\inscribe{Eleven}{40}}
  \end{subfigure}
  \caption{Squares inscribed on three lines.}
  \label{fig:inscribed-threeZero}
\end{figure}

\begin{figure}[H]
    \includegraphics[width=\textwidth]{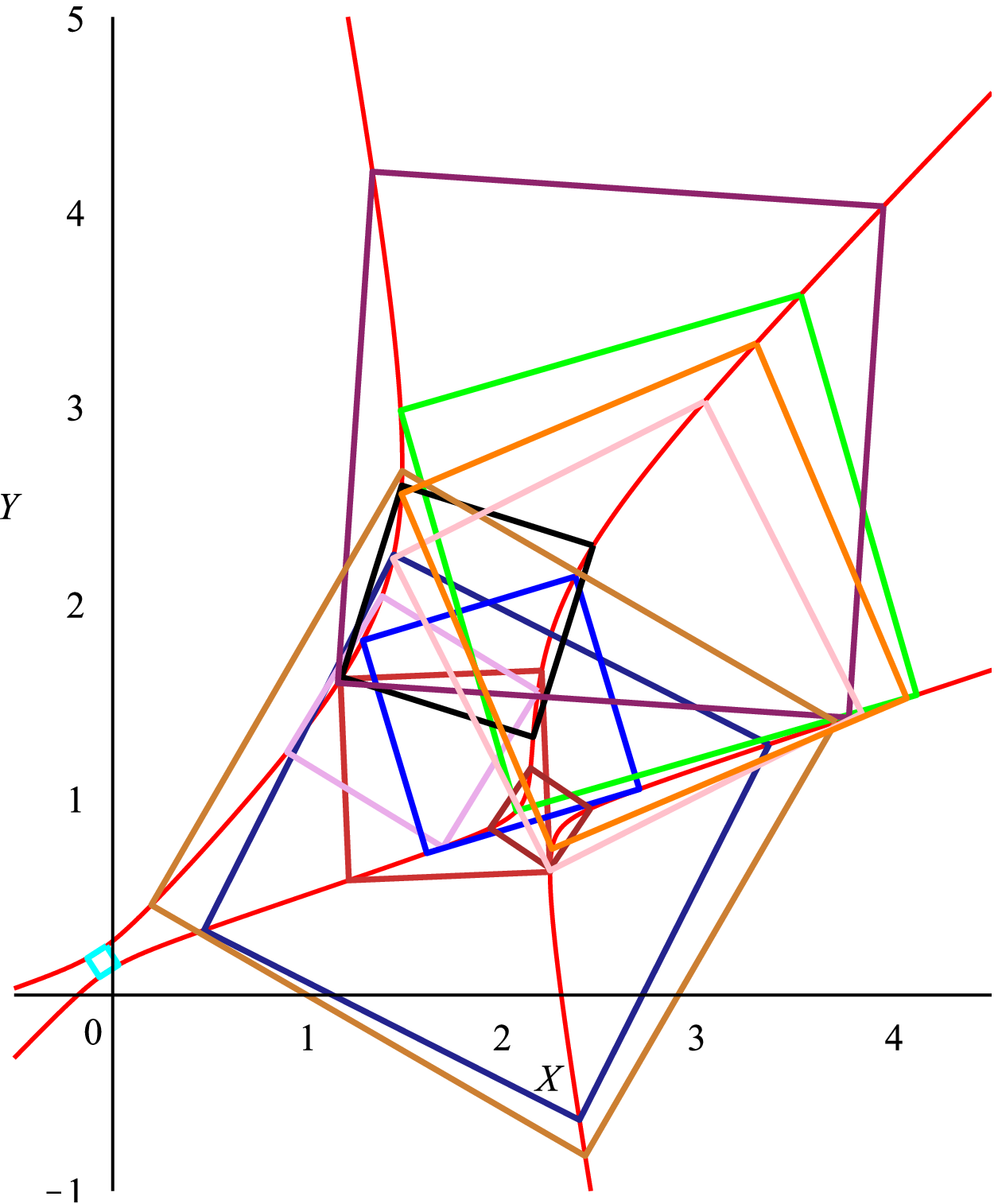}
   \caption{Twelve squares inscribed on
     \protect\hyperlink{poly:f12}{$f_{12}$} in \Fref{tab:long-polys}}
   \label{fig:twelve-awesome}
\end{figure}

\subsubsection{Squares inscribed on an oval and three lines}
% load "classify.m2"
% threeOne = select(allCurves, t -> (3, 1) == (t_3, t_4))
% threeOne / (k -> (length last k, k))
% threeOne = sort oo / last
% forMapleArray(oo / first, oo / last)

The curves in \Fref{fig:inscribed-threeOne} inscribe eight, nine and
eleven squares.
\begin{figure}[H]
%\centering
  \begin{subfigure}[b]{0.5\textwidth}
    \includegraphics[width=\textwidth]{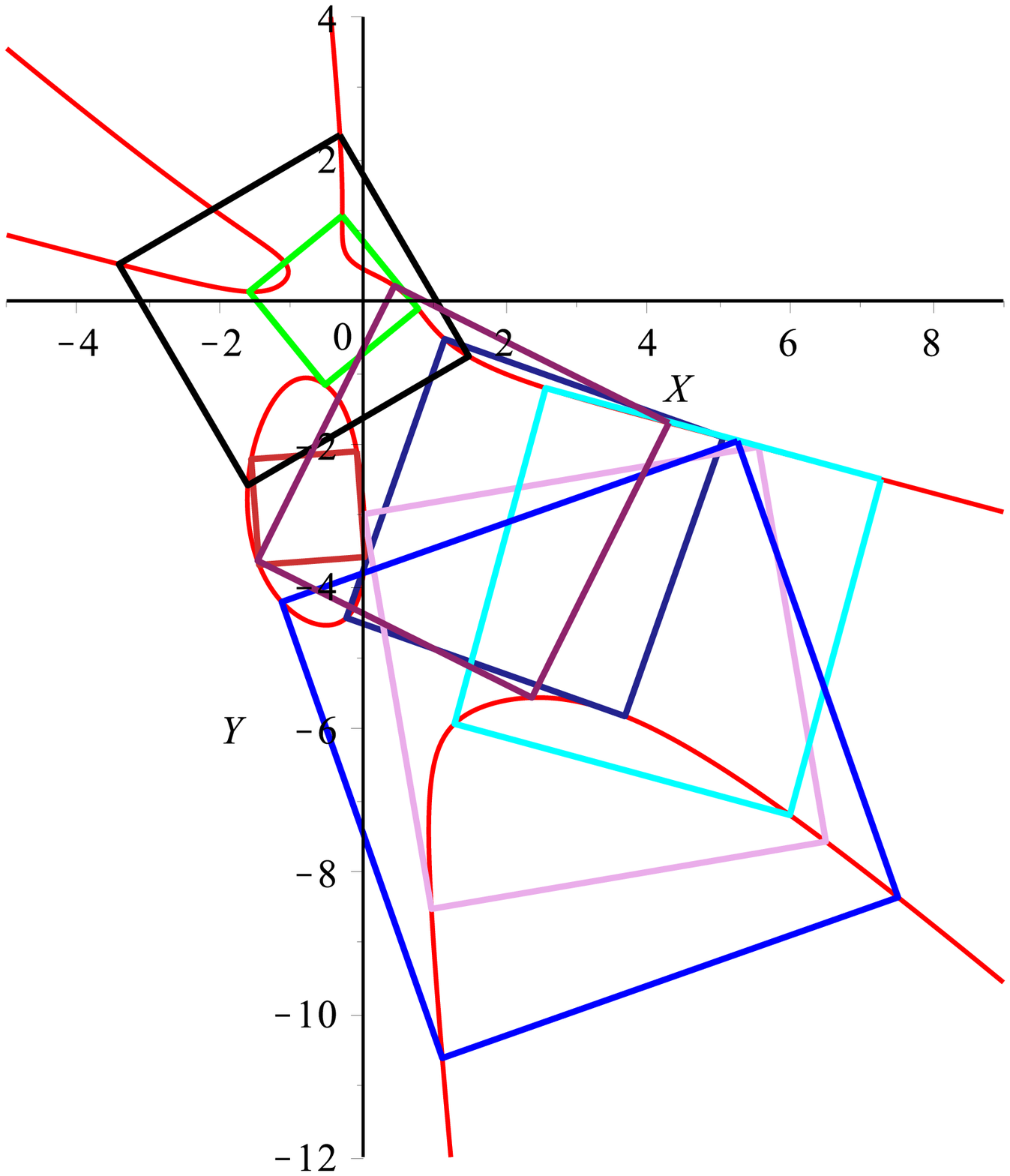}
    \caption{\inscribe{Eight}{27}}
  \end{subfigure}
  \begin{subfigure}[b]{0.5\textwidth}
    \includegraphics[width=\textwidth]{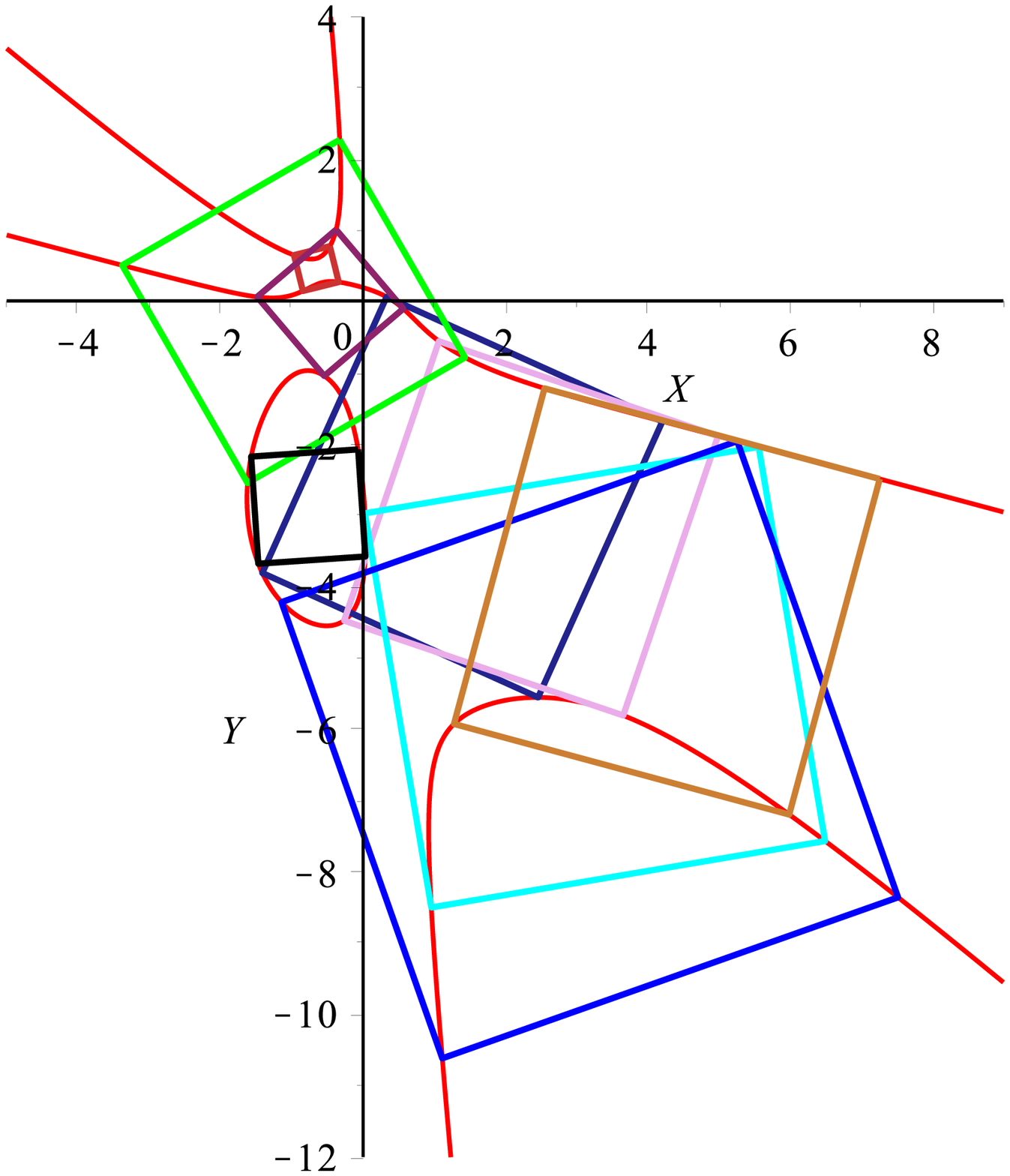}
    \caption{\inscribe{Nine}{28}}
  \end{subfigure}
{\centering
  \begin{subfigure}[b]{0.85\linewidth}
    \includegraphics[width=\textwidth]{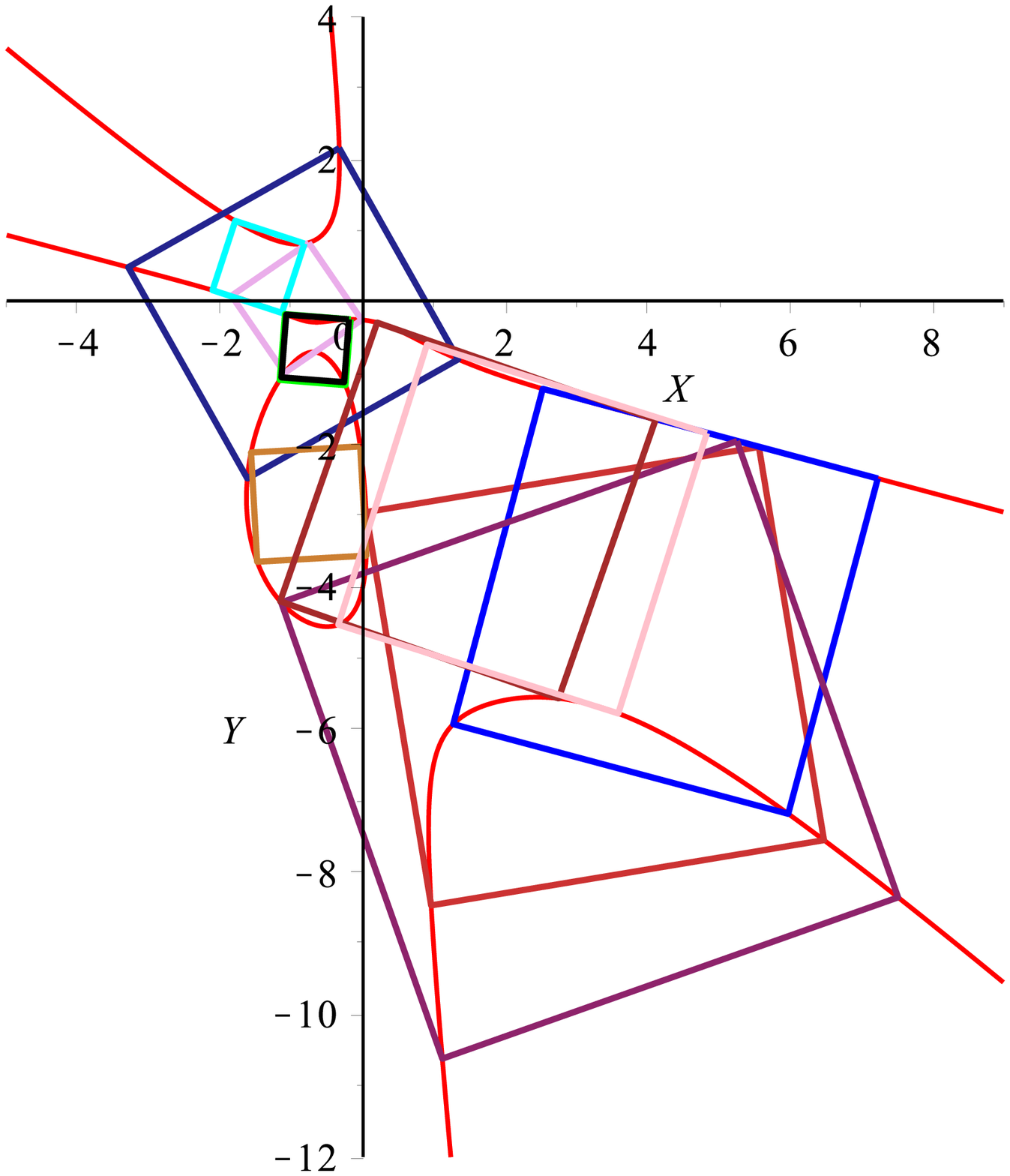}
    \caption{\inscribe{Eleven}{29}}
  \end{subfigure}
}
   \caption{Squares inscribed on an oval and three lines.}
   \label{fig:inscribed-threeOne}
\end{figure}
% select(allCurves, t -> (3, 1) == (t_3, t_4))
% forMapleSequence(oo / first, oo / last)

\section{ Concluding remarks}
\label{sec:conclusions}

The main result of this thesis, \Fref{thm:mine}
in \Fref{sec:upper-bound}, shows that the number of isolated squares inscribed on
a degree $m$ complex algebraic plane curve is at most $(m^4 - 5m^2 +
4m)/4$.   The experimental evidence of \Fref{sec:experimental}
suggests this statement might be strengthened to ``a generic complex
algebraic plane curve inscribes precisely $(m^4 - 5m^2 +
4m)/4$ squares''.  Whether that is true or not, one can ask for any natural
number $m$ what the maximum attainable
number of isolated inscribed squares is on a curve of degree $m$. Can we
construct a curve that attains the theoretical maximum  of $(m^4 - 5m^2 +
4m)/4$?  At least up to degree five any of the curves of
\Fref{tab:experiments} provides a positive answer, but we should aim
for a theoretical argument for all degrees.
Following Rojas~\cite[Section~3.3, p7]{Rojas97toricintersection}, giving the conditions
when the maximum number of solutions is attained might be fruitful.
Intersection theory may also apply to show that the complex squares
from \Fref{tab:experiments} have multiplicity one.
\\

Restricting these questions to real plane curves we can ask again, is
there a real algebraic plane curve that attains the bound of
\Fref{thm:mine}? \Fref{sec:illustrative} includes several positive
examples for degree three.

Certain symmetries in a plane curve give rise to an infinite number of
inscribed squares.  The author is however not aware of a complete
classification of which kinds of curves inscribe an infinitude of squares.

Based on the shaded cells of \Fref{tab:square-topologies} we could
conjecture: Is it true that algebraic plane curves homeomorphic to one of
\begin{enumerate}
  \item the real line
  \item an oval and two lines
  \item two ovals
\end{enumerate}
inscribe respectively an even, odd, and even number of squares?
The other shaded cell corresponds to algebraic Jordan curves, for
which it is already known that this class of curves generically
inscribes an odd number of squares.

Approximating a general Jordan curve with a subclass of curves for
which we know Toeplitz's conjecture to be true may fail to produce an
inscribed square in the limit if the approximating squares degenerate
to a point. Pak~\cite[Section~3.7]{pak} remarks that nonetheless the
limit argument has its use;  for an approximation argument by
algebraic curves we will need to have control over the sizes of the
squares to prevent the squares from degenerating in the limit.

\newpage
\section*{Appendix}
\addcontentsline{toc}{section}{Appendix}
\subsection*{Table of polynomials}
\addcontentsline{toc}{subsection}{Table of polynomials}
\label{sec:tables}

%\begin{landscape}
\begin{longtable}[H]{rp{0.9\textwidth}}
\caption{Polynomials defining curves}
\endhead
\caption{Polynomials defining curves in \Fref{sec:illustrative}.}\label{tab:long-polys} \\
\endfirsthead
\hypertarget{poly:f1}{$f_1$}&
$\scriptstyle(3/8) x^{3}+4 x^{2} y+(10/7) x y^{2}+(2/7) y^{3}+x^{2}+10 x y+(7/9)
      y^{2}+(1/7) x+(4/5) y+10369/300$ \\
\hypertarget{poly:f2}{$f_2$} & $\scriptstyle-(1013346057932523458320374654611/2350924922880000000000)
     x^{3}+(2584640714944881315625401696659/1959104102400000000000) x^{2}
     y-(24370961833016176942717940959039/58773123072000000000000) x
     y^{2}+(495964933561657788423357606871/489776025600000000000)
     y^{3}-(17651791649159643199956179410837/23509249228800000000000)
     x^{2}+(255915596711949314264306252576989/117546246144000000000000) x
     y-(5664920610070897911630510019033/653034700800000000000)
     y^{2}-(45022793169990743253008147707121/11754624614400000000000)
     x+(659705135608555410904182133087481/58773123072000000000000)
     y+12665836021084318920971168631593/11754624614400000000000$ \\
\hypertarget{poly:f3}{$f_3$} &
$\scriptstyle (1/7) x^{5}+(6/7) x^{4} y+(9/5) x^{3} y^{2}+x^{2} y^{3}+7 x y^{4}+10
      y^{5}+x^{4}+(4/5) x^{3} y+(10/7) x^{2} y^{2}+3 x y^{3}+(7/5) y^{4}+(7/6)
      x^{3}+(1/8) x^{2} y+(3/4) x y^{2}+(1/3) y^{3}+(3/10) x^{2}+(4/5) x y+(5/3)
      y^{2}+(5/3) x+(10/9) y+9/4$ \\
\hypertarget{poly:f4}{$f_4$} &
$\scriptstyle (1/2) x^{3}+5 x^{2} y+(2/9) x y^{2}+(5/6) y^{3}+(9/7) x^{2}+9 x y+(1/9)
      y^{2}+(7/5) x+(10/9) y+5/6$ \\
\hypertarget{poly:f5}{$f_5$} & $\scriptstyle (1/3) x^{3}+x^{2} y+(7/9) x y^{2}+9
      y^{3}+(10/9) x^{2}+2 x y+(7/2) y^{2}+(8/7) x+(1/10) y+1/3$ \\
\hypertarget{poly:f6}{$f_6$} & $\scriptstyle (3/8) x^{3}+4 x^{2} y+(10/7) x y^{2}+(2/7) y^{3}+x^{2}+10 x y+(7/9)
      y^{2}+(1/7) x+(4/5) y-19/600$ \\
\hypertarget{poly:f7}{$f_7$} & $\scriptstyle (32357486150754911/3402639576000000)
      x^{3}
      -(14565996465296101997/2143662932880000000) x^{2}y
      +(93487619285326211413/135050764771440000000) xy^{2}
      +(295881163208333/837368333156250)y^{3} -$ \newline
      $\scriptstyle (16455993365369237399/1071831466440000000)x^{2}$ \newline
      $\scriptstyle+(2262751792681121895697/270101529542880000000) xy$ \newline
      $\scriptstyle-(44377450778015156987/16881345596430000000)y^{2}$ \newline
      $\scriptstyle+(483511249013004548209/90033843180960000000)x$ \newline
      $\scriptstyle+(43079601667153982323/33762691192860000000)y
      -9025382297117723393/11254230397620000000$ \\
\hypertarget{poly:f8}{$f_8$} &
$\scriptstyle (4/3)x^5+7x^4y+(7/3)x^3y^2+(1/2)x^2y^3+(1/2)xy^4+(1/10)y^5+(10/7)x^4+(7/3)x^3y+(2/5)x^2y^2+(2/3)xy^3+(5/9)y^4+(3/2)x^3+3x^2y+xy^2+(1/3)y^3+4x^2+(2/3)xy+(8/9)y^2+(8/3)x+(1/10)y+7/5$
\\
\hypertarget{poly:f9}{$f_9$} &
$\scriptstyle (84600046159243700856114369758453/7304069487211315200000000)x^3+(84129864593783714477250895601927/4869379658140876800000000)x^2y-(92678186386758381841697632332217/7304069487211315200000000)xy^2-(985032890300878882041984922489/2921627794884526080000000)y^3-(8432107925141586913574285861810083/97387593162817536000000000)x^2-(42810357305315843166329246331701/1803473947459584000000000)xy+(798870306331587087351224027449571/58432555897690521600000000)y^2+(841046078607802229000433529244096647/5258930030792146944000000000)x-(50130881628172999018538048620781701/5258930030792146944000000000)y-120544249950526645232049562396939597/1752976676930715648000000000$
\\
\hypertarget{poly:f10}{$f_{10}$} &
$\scriptstyle
(17071630870821024280289/127253121732748247040000000)x^3+(44219727353738152825699/5302213405531176960000000)x^2y-(4775926187801988597243641/127253121732748247040000000)xy^2+(2615354993498783429179/108208436847575040000000)y^3\phantom{break
  here god damn it
     }-(218792069736804757977449/38955037265127014400000000)x^2+(303432548905886033642387/6362656086637412352000000)xy-(15987089135911642991445653/381759365198244741120000000)y^2-(86769535959101859196900919/7635187303964894822400000000)x+(1265378561015612782058837/61081498431719158579200000)y-2225833681103904456175739/763518730396489482240000000$
\\
\hypertarget{poly:f11}{$f_{11}$} &
$\scriptstyle-(107666602244268965505153/34359738368000000000000)x^3+(244020905347080929848137/13743895347200000000000)x^2y\phantom{break
  here god damn it}+(3029447197152010641168729/34359738368000000000000)xy^2-(2494391888436262290669501/68719476736000000000000)y^3\phantom{break
  here god damn it}-(6731424554769315405645039/1374389534720000000000000)x^2-(1119679636867415864847621/4294967296000000000000)xy\phantom{break
  here god damn it}-(88162122657769201785657501/1374389534720000000000000)y^2+(1720365306508271453007846519/13743895347200000000000000)x+(5145387047581092010866673443/13743895347200000000000000)y-676235828568952472903449101/3435973836800000000000000$
\\
\hypertarget{poly:f12}{$f_{12}$} &
 $\scriptstyle-(4963493942513921243/65548320768000000)x^3+(326139891975237682121/1123685498880000000)x^2y-(50931413248303191071/299649466368000000)xy^2-(14263797412722377/339738624000000)y^3+(37805850432694119373/327741603840000000)x^2\phantom{break
  here god damn it}-(19179033623835553860379/31463193968640000000)xy\phantom{break
  here god damn it}+(1018795941059176616167/1997663109120000000)y^2\phantom{break
  here god damn it}+(1330205416456247598397/10487731322880000000)x\phantom{break
  here god damn it}-(2843296777056554250263/13983641763840000000)y+95073566433481051/5202247680000000$
 \\
\hypertarget{poly:f13}{$f_{13}$} &
$\scriptstyle 12415x^8+11377x^7y+15240x^6y^2-451x^5y^3+4672x^4y^4+4256x^3y^5+2937x^2y^6-14392xy^7-11440y^8-1118x^7+8649x^6y+9988x^5y^2+15342x^4y^3-13207x^3y^4+4533x^2y^5+13680xy^6+9917y^7-8343x^6-6757x^5y-8308x^4y^2+7606x^3y^3+3138x^2y^4-5358xy^5+11848y^6+12694x^5+181x^4y+3136x^3y^2-12922x^2y^3-14700xy^4+9107y^5+9973x^4+1173x^3y-15433x^2y^2+2406xy^3-13196y^4-8485x^3-8414x^2y-15263xy^2+15206y^3-7714x^2-7243xy+4230y^2-10183x+5303y-3662$\\
\hypertarget{poly:f14}{$f_{14}}$ &
$\scriptstyle (10/9)x^4+(2/7)x^3y+2x^2y^2+5xy^3+(10/7)y^4+5x^3+(10/3)x^2y+(2/5)xy^2+(1/7)y^3+(1/2)x^2+(10/9)xy+(3/2)y^2+(1/7)x+(5/9)y+4$
\\
\hypertarget{poly:f15}{$f_{15}$} &
$\scriptstyle (1/4)x^4+5x^3y+(5/3)x^2y^2+(1/10)xy^3+(1/9)y^4+x^3+(2/3)x^2y+9xy^2+(1/8)y^3+(7/10)x^2+(1/5)xy+(4/5)y^2+(4/5)x+(5/8)y+3/10$
\\
\hypertarget{poly:f16}{$f_{16}$} &
$\scriptstyle (1/4)x^4+5x^3y+(5/3)x^2y^2+(1/10)xy^3+(1/9)y^4+x^3+(2/3)x^2y+9xy^2+(1/8)y^3+(7/10)x^2+(1/5)xy+(4/5)y^2+(4/5)x+(5/8)y-97/10$
\\
\hypertarget{poly:f17}{$f_{17}$} &
$\scriptstyle (1/4)x^4+5x^3y+(5/3)x^2y^2+(1/10)xy^3+(1/9)y^4+x^3+(2/3)x^2y+9xy^2+(1/8)y^3+(7/10)x^2+(1/5)xy+(4/5)y^2+(4/5)x+(5/8)y-27/10$
\\
\hypertarget{poly:f18}{$f_{18}$} &
$\scriptstyle-x^3+y^2+x$ \\
\hypertarget{poly:f19}{$f_{19}$} &
$\scriptstyle-(1/5)x^3+x^2y-(1/5)xy^2+y^3+(8/5)xy-8y^2-(12/5)x+12y+1/100$
\\
\hypertarget{poly:f20}{$f_{20}$} &
$\scriptstyle (3/8)x^3+4x^2y+(10/7)xy^2+(2/7)y^3+x^2+10xy+(7/9)y^2+(1/7)x+(4/5)y+1687/300$
\\
\hypertarget{poly:f21}{$f_{21}$} &
$\scriptstyle (4/9)x^3+(10/7)x^2y+xy^2+(3/4)y^3+(7/2)x^2+8xy+(4/7)y^2+(4/3)x+(1/2)y+5/7$
\\
\hypertarget{poly:f22}{$f_{22}$} &
$\scriptstyle (1/4)x^5+2x^4y+(8/5)x^3y^2+(7/6)x^2y^3+(2/9)xy^4+(1/2)y^5+(3/5)x^4+8x^3y+5x^2y^2+(9/5)xy^3+2y^4+(7/10)x^3+7x^2y+9xy^2+2y^3+3x^2+4xy+(10/9)y^2+(10/3)x+(1/4)y+1/3$
\\
\hypertarget{poly:f23}{$f_{23}$} &
$\scriptstyle (8/3)x^3+(7/8)x^2y+(1/5)xy^2+(1/2)y^3+(1/2)x^2+8xy+6y^2+(5/4)x+5y+1/5$
\\
\hypertarget{poly:f24}{$f_{24}$} &
$\scriptstyle (1/5)x^3+x^2y+(7/4)xy^2+(4/5)y^3+(9/7)x^2+10xy+7y^2+2x+(7/10)y+5/8$
\\
\hypertarget{poly:f25}{$f_{25}$} &
$\scriptstyle (1/2)x^3+(3/2)x^2y+2xy^2+(2/9)y^3+x^2+9xy+(3/2)y^2+(6/7)x+(2/3)y+5/4$
\\
\hypertarget{poly:f26}{$f_{26}$} &
$\scriptstyle (3/8)x^3+4x^2y+(10/7)xy^2+(2/7)y^3+x^2+10xy+(7/9)y^2+(1/7)x+(4/5)y+1/3$\\
\hypertarget{poly:f27}{$f_{27}$} &
$\scriptstyle (1/2)x^5+(9/4)x^4y+(8/5)x^3y^2+(5/7)x^2y^3+(4/3)xy^4+(1/8)y^5+(4/5)x^4+(2/5)x^3y+(8/5)x^2y^2+7xy^3+(2/3)y^4+(5/8)x^3+(3/7)x^2y+(9/7)xy^2+(3/5)y^3+x^2+(6/7)xy+(1/3)y^2+(1/2)x+(5/2)y-4/3$
\\
\hypertarget{poly:f28}{$f_{28}$} &
$\scriptstyle (1/2)x^5+(9/4)x^4y+(8/5)x^3y^2+(5/7)x^2y^3+(4/3)xy^4+(1/8)y^5+(4/5)x^4+(2/5)x^3y+(8/5)x^2y^2+7xy^3+(2/3)y^4+(5/8)x^3+(3/7)x^2y+(9/7)xy^2+(3/5)y^3+x^2+(6/7)xy+(1/3)y^2+(1/2)x+(5/2)y-8/15$ \\
\hypertarget{poly:f29}{$f_{29}$} &
$\scriptstyle (1/2)x^5+(9/4)x^4y+(8/5)x^3y^2+(5/7)x^2y^3+(4/3)xy^4+(1/8)y^5+(4/5)x^4+(2/5)x^3y+(8/5)x^2y^2+7xy^3+(2/3)y^4+(5/8)x^3+(3/7)x^2y+(9/7)xy^2+(3/5)y^3+x^2+(6/7)xy+(1/3)y^2+(1/2)x+(5/2)y+461/750$
\\
\hypertarget{poly:f30}{$f_{30}$} &
$\scriptstyle
(7/9)x^4+(1/2)x^3y+(7/6)x^2y^2+(4/5)xy^3+(4/3)y^4+(2/7)x^3+(4/7)x^2y+(8/3)xy^2+(1/5)y^3+(7/10)x^2+(3/5)xy+(1/6)y^2+5x+(5/7)y+3/10$ \\
\hypertarget{poly:f31}{$f_{31}$} &
$\scriptstyle (3/10)x^4+(5/4)x^3y+(7/5)x^2y^2+(1/5)xy^3+y^4+(9/10)x^3+4x^2y+(2/9)xy^2+y^3+(3/4)x^2+(3/4)xy+y^2+(1/2)x+(9/2)y+9/8$ \\
\hypertarget{poly:f32}{$f_{32}$} &
{$\scriptstyle 4x^4+(1/2)x^3y+(1/9)x^2y^2+2xy^3+(9/7)y^4+9x^3+5x^2y+(5/3)xy^2+(4/3)y^3+(4/3)x^2+(5/2)xy+y^2+(1/3)x+(7/6)y+71/200$} \\
\hypertarget{poly:f33}{$f_{33}$} &
{$\scriptstyle 4x^4+(1/2)x^3y+(1/9)x^2y^2+2xy^3+(9/7)y^4+9x^3+5x^2y+(5/3)xy^2+(4/3)y^3+(4/3)x^2+(5/2)xy+y^2+(1/3)x+(7/6)y+3/8$} \\
\hypertarget{poly:f34}{$f_{34}$} &
$\scriptstyle (9/4)x^4+3x^3y+(1/7)x^2y^2+(2/7)xy^3+(1/3)y^4+(4/5)x^3+(1/5)x^2y+8xy^2+4y^3+2x^2+(10/9)xy+(5/3)y^2+(1/9)x+(1/5)y+2$ \\
\hypertarget{poly:f35}{$f_{35}$} &
$\scriptstyle (1/4)x^4+5x^3y+(5/3)x^2y^2+(1/10)xy^3+(1/9)y^4+x^3+(2/3)x^2y+9xy^2+(1/8)y^3+(7/10)x^2+(1/5)xy+(4/5)y^2+(4/5)x+(5/8)y+33/10$ \\
\hypertarget{poly:f36}{$f_{36}$} &
$\scriptstyle (1/5)x^4+(7/8)x^3y+(1/2)x^2y^2+(5/4)xy^3+y^4+(1/3)x^3+x^2y+8xy^2+y^3+(3/4)x^2+(5/7)xy+(5/9)y^2+(9/8)x+5y+4/3$ \\
\hypertarget{poly:f37}{$f_{37}$} &
$\scriptstyle (1/4)x^4+5x^3y+(5/3)x^2y^2+(1/10)xy^3+(1/9)y^4+x^3+(2/3)x^2y+9xy^2+(1/8)y^3+(7/10)x^2+(1/5)xy+(4/5)y^2+(4/5)x+(5/8)y+13/10$ \\
\hypertarget{poly:f38}{$f_{38}$} &
$\scriptstyle (1/7)x^3+(7/2)x^2y+(7/3)xy^2+(1/10)y^3+(6/7)x^2+9xy+(1/2)y^2+(7/5)x+y+1$ \\
\hypertarget{poly:f39}{$f_{39}$} &
$\scriptstyle (1/8)x^3+x^2y+2xy^2+(1/6)y^3+(6/7)x^2+9xy+(7/9)y^2+(1/9)x+(2/9)y+8/5$\\
\hypertarget{poly:f40}{$f_{40}$} &
$\scriptstyle (1/10)x^3+(7/6)x^2y+(9/7)xy^2+(1/8)y^3+(9/4)x^2+10xy+2y^2+5x+(3/4)y+1/6$
\\
\hypertarget{poly:f41}{$f_{41}$} &
$\scriptstyle (1/4) x^{4}+(17/16) x^{2} y^{2}+(1/4) y^{4}-(5/4) x^{2}-(5/4)
      y^{2}+4382/7225$ \\
\hypertarget{poly:f42}{$f_{42}$} &
 $\scriptstyle 4 x^{4}+(1/2) x^{3} y+(1/9) x^{2} y^{2}+2 x y^{3}+(9/7) y^{4}+9 x^{3}+5
      x^{2} y+(5/3) x y^{2}+(4/3) y^{3}+(4/3) x^{2}+(5/2) x y+y^{2}+(1/3) x+(7/6)
      y+7/8$ \\
\hypertarget{poly:f43}{$f_{43}$} &
$\scriptstyle 4 x^{4}+(1/2) x^{3} y+(1/9) x^{2} y^{2}+2 x y^{3}+(9/7) y^{4}+9 x^{3}+5
      x^{2} y+(5/3) x y^{2}+(4/3) y^{3}+(4/3) x^{2}+(5/2) x y+y^{2}+(1/3) x+(7/6)
      y+27/40$ \\
\hypertarget{poly:f44}{$f_{44}$} &
$\scriptstyle 4 x^{4}+(1/2) x^{3} y+(1/9) x^{2} y^{2}+2 x y^{3}+(9/7) y^{4}+9 x^{3}+5
      x^{2} y+(5/3) x y^{2}+(4/3) y^{3}+(4/3) x^{2}+(5/2) x y+y^{2}+(1/3) x+(7/6)
      y+19/40$ \\
\hypertarget{poly:f45}{$f_{45}$} &
 $\scriptstyle (1/4) x^{4}+(17/16) x^{2} y^{2}+(1/4) y^{4}-(5/4) x^{2}-(5/4)
      y^{2}+40453/43350$ \\
\end{longtable}
%\end{landscape}

\newpage
\subsection*{Code}
\addcontentsline{toc}{subsection}{Code}

\newcounter{listing}
\refstepcounter{listing}
\begin{Verbatim}
-- Calculate the Minkowski volume
--  of m*P1 + l*P2 + g*Delta for degree k
R = QQ[k, e_1..e_2, m_1..m_2][l_1..l_4]

K = (m_1 + m_2 + l_4)*k
L = (m_1*e_1 + m_2*e_2)
M = (l_4 + m_1)*k + m_2*(k - 1)

Vol = (K - L)^3 * (K + 3*L) - (K - M)^3 * (K + 3*M)

volToMvol = (substitutions) ->  (
  vol := sub(Vol, substitutions);
  Mvol := (last coefficients (vol, Monomials => {l_1*l_2*l_3*l_4}))_0_0;
  Mvol = Mvol/4!;  -- Compensate for the volume of the standard simplex
  assert (Mvol == k^4 - 5*k^2 + 4*k); -- Confirm we got the answer we expect
  return Mvol;
)

-- When k is even there is one copy of P1 (even monomials) and two of
--   P2 (odd monomials) and the other way around when k is odd.
({m_1 => l_1,       m_2 => l_2 + l_3, e_1 => 2, e_2 => 1},
 {m_1 => l_2 + l_3, m_2 => l_1,       e_1 => 1, e_2 => 2}) / volToMvol

\end{Verbatim}
\medskip

\refstepcounter{listing}
\begin{Verbatim}
-- For m=2, 3 the polytopes do not have their general shape (and
-- aren't full dimensional either).  However, the Minkowski sum /
-- mixed volume calculation still makes sense.  So just do that for
-- these special cases.
needsPackage "PHCpack"

-- m = 2 case
g4 = (a^2 + b^2 + c^2 + d^2 + 1)
g1 = (c^2 + d^2)
g2 = (c + d)*(1 + a + b)

mv = mixedVolume {g1, g2, g2, g4}
assert (mv == 2^4 - 5*2^2 + 4*2)

-- m = 3 case
R = CC[a, b, c, d]
P4 = newtonPolytope

g4 = (a^3 + b^3 + c^3 + d^3 + 1)
g1 = ((c^2 + d^2)*(1 + a + b))
g2 = ((c + d)*(1 + a^2 + b^2) + (c^3 + d^3))

mv = mixedVolume {g1, g2, g2, g4}
assert (mv == 3^4 - 5*3^2 + 4*3)
\end{Verbatim}
\medskip

\refstepcounter{listing}
\begin{Verbatim}
-- Numerical evidence for sharp BKK bound via degree counting.
S = QQ; load "preamble.m2"; D = 3; degreeSetup(D)

H = new MutableHashTable from {}

coeffs = unique toList apply(1..100, i -> randomCoefficients_D());
curves = coeffs / (c -> sub(abstractCurve_D, c));
fillIn_countSquares_H curves
tally values H
\end{Verbatim}
\medskip

\refstepcounter{listing}
\begin{Verbatim}
S = QQ; load "preamble.m2"; D = 3; degreeSetup(D)

use ring abstractCurve_D
monomialTerms = terms sub(abstractCurve_D, validDegrees_D / (i -> C_i => 1))

curveThroughPoints = (N) -> (
    use ring abstractCurve_D;
    planePoints := toList(apply(1..N, i -> (random(S), random(S))));
    M := matrix (
	{monomialTerms}	| (planePoints /
	    (p -> monomialTerms /
		(t -> sub(t, {X => p_0, Y => p_1})))));
    return determinant M;
    );


H = new MutableHashTable from {};

curves = toList select(apply(1..20, i -> curveThroughPoints(9)), c -> 0 != c)
fillIn_(realSolutions_D @@ curveToCoeff_D)_H curves

pairs H / last / length
tally oo
\end{Verbatim}
\medskip

\refstepcounter{listing}
\begin{Verbatim}
load "realroots.m2"
needsPackage "PHCpack"

W = S[a, b, c, d, MonomialSize => 8];
excess = ideal(c, d);
PHCring = CC[a, b, c, d];


sparseCoeffs = (coeff, localD) -> (
    H := new HashTable from coeff;
    -- Poor mans dict.update(H)
    return for deg in (validDegrees_localD / (d -> C_d)) list
        (if H#?deg then (deg => H#deg) else (deg => 0));
);

zerofy = (squares) -> (
 squares / (square -> for x in square list
   if abs(x) < 1.0e-15 then 0.0 else x))
);

filterReal = (solutions) -> (
  return select(solutions / coordinates,
   j -> all(j, i -> 1.0e-90 > abs imaginaryPart i)) / (s -> s / realPart);
);

forMaple = (D, coeff, solss) -> (
  bounds := {"-10..10", "-10..10"};
  if length solss > 0 then (
     sols := solss / toList;
     Xen  := flatten(sols /
         (s -> {s_0 + s_2, s_0 - s_2, s_0 + s_3, s_0 - s_3} ));
     Yen  := flatten(sols /
         (s -> {s_1 + s_2, s_1 - s_2, s_1 + s_3, s_1 - s_3} ));
     bounds = (Xen, Yen) / (l ->
         toString floor(-2 + min l) | ".." | toString ceiling(2 + max l));
  ) else (
        sols = [];
  );
  return "plotSquaresOnCurve" | toString ("(X, Y) -> " |
    toString sub(abstractCurve_D, coeff),
    " [X=" | bounds_0 | ", Y=" | bounds_1 | ", gridrefine=4] ",
    replace("\\}|\\)", "]", replace("\\{|\\(", "[", toString sols))) | ";";
);

forMapleSimple = (curve, squares) -> (
  return "plotSquaresOnCurve((X, Y) -> " | toString curve | ", opts, " |
    replace("\\}|\\)", "]", replace("\\{|\\(", "[", toString squares)) |")\n";
);

forMapleSequence = (curves, solutions) -> (
  assert(length curves == length solutions);
  contentS := toString(toList(
	    apply(0..length(curves) - 1,
		i -> forMapleSimple(curves_i, solutions_i))));
  return "opts := []; display(" | contentS | ", insequence=true);";
);

forMapleArray = (curves, solutions) ->    (
  assert(length curves == length solutions);
  contentS := toString(toList(
	    apply(0..length(curves) - 1,
		i -> forMapleSimple(curves_i, solutions_i))));
  return "opts := []; display(Array([[" | contentS | "]], transpose));";
);


fillIn = (work, H, curves) -> (
  for curve in curves do (
    if not H #? curve then (
      result := work curve;
      H # curve = result;
    ) else (
      print ("Curve " | toString curve | " already present");
    );
  );
);

countSquares = (curve) -> (
  I := time saturate(makeIdeal_D curveToCoeff_D curve, excess);
  return (dim I, degree I);
);

degreeSetup = (D) -> (
 validDegrees_D = select(toList(
	            set toList(0..D))^**2 / toList, d -> sum(d) <= D);
 R_D = S[apply(validDegrees_D, d -> C_d),
                    MonomialSize => 8][a, b, c, d, MonomialSize => 8];
 T_D = R_D[X, Y];

 curveToCoeff_D = (curve) -> (
    sparseCoeffs(terms curve /
               (j  -> C_(first exponents j) => leadCoefficient j), D);
 );

 use T_D;
 abstractCurve_D = sum(validDegrees_D / (d -> C_d * X^(d_0) * Y^(d_1)));
 use R_D;
 corners_D =  {{ X => a + c, Y => b + d },
   { X => a - c, Y => b - d },
   { X => a + d, Y => b - c },
   { X => a - d, Y => b + c }} / (corner -> sub(abstractCurve_D, corner));
    IJ_D = ideal(
    	corners_D_0 + corners_D_1 - corners_D_2 - corners_D_3,
    	corners_D_0 - corners_D_1,
    	corners_D_2 - corners_D_3,
    	corners_D_3
    );
    -- FIXME: doing the saturation here is perhaps the wrong point.
    -- On the other hand, if we can store this computation, it might speed
    -- things up.

randomCoefficients_D = () -> (
    return apply(validDegrees_D, s -> C_s => random(S))
    );

makeIdeal_D = (coeff) -> (
    use W;
    I := sub(sub(IJ_D, coeff), W);
    J := I;
    return J;
    );

realSolutions_D = (coeff) -> (
    IP := sub(makeIdeal_D(coeff), PHCring);
    use PHCring;  -- this is done to avoid the "key not found"
    complexSols := solveSystem IP_*;
    sols := unique zerofy filterReal complexSols;
    squares := select(sols, s -> s_2 >= 0 and s_3 > 0);
    if (length sols != 4 * length squares) then
    (
	print("Mismatch in solutions and squares " |
	       toString (length sols, length squares));
	sols = unique zerofy filterReal refineSolutions(IP_*, complexSols, 18);
	squares = sort select(sols, s -> s_2 >= 0 and s_3 > 0);
    );
    return squares
    );
)
\end{Verbatim}
\medskip

\refstepcounter{listing}
\begin{Verbatim}[fontsize=\footnotesize,frame=topline,framesep=5mm,label=\fbox{Listing \arabic{listing}: \Large
    drawSquares.mw}\label{lst:drawSquares.mw}]
with(plots):
with(plottools):
with(RAGMaple):
SquarePegs:=module()
option package;
export plotSquare, plotSquaresOnCurve, componentsPoints;
local colorList;

  componentsPoints := (curve) -> (
      seq(point([rhs(P[1]), rhs(P[2])]),
             P in PointsPerComponents([ curve = 0 ], [X, Y]))
  );

  plotSquare := proc(param, kleur)
    local a, b, c, d, p1, p2, p3, p4, line1, line2, line3, line4, plotOpts;
    (a, b, c, d) := op(param);
    plotOpts := thickness=2, color=kleur;
    p1 := [a + c, b + d]:
    p2 := [a - d, b + c]:
    p3 := [a - c, b - d]:
    p4 := [a + d, b - c]:
    display(CURVES([p1, p2, p3, p4, p1]), plotOpts):
  end proc:

 colorList := [
     navy, orange, plum, cyan,
     blue, green, black, maroon,
     gold, brown, pink, coral, magenta,
     khaki
 ];

plotSquaresOnCurve := proc(curve, curveOpts, squares,
                                             showComponents::boolean := true,
                                             showLegend::boolean := true)
  local curvePlot, squaresPlot, setopts, xsX, ysY, passOpts,
          plotList, componentPoints;
  setopts := [seq(lhs(o), o in curveOpts)];
  passOpts := curveOpts;
  if evalb(showComponents) then
       componentPoints := [seq(
          [rhs(P[1]), rhs(P[2])],
          P in PointsPerComponents([curve(X, Y) = 0], [X, Y])
       )];
  else
      componentPoints := [];
  end if;
  if evalb(not X in setopts) then
       xsX := ListTools[Flatten](
           [seq([s[1] + s[3], s[1] + s[4], s[1] - s[3], s[1] - s[4]],
            s in squares)]
       );
       passOpts := [op(passOpts), X=-1+floor(min(xsX, seq(
           P[1], P in componentPoints)))..1
                         +ceil(max(xsX, seq(P[1], P in componentPoints))
       )];
  end if;
  if evalb(not Y in setopts) then
       ysY := ListTools[Flatten]([seq(
           [s[2] + s[3], s[2] + s[4], s[2] - s[3], s[2] - s[4]], s in squares
       )]);
       passOpts := [op(passOpts), Y=-1+floor(min(ysY, seq(
           P[2], P in componentPoints)))..1
                         +ceil(max(ysY, seq(P[2], P in componentPoints))
       )];
  end if;
  if evalb(not gridrefine in setopts) then
     passOpts := [op(passOpts), gridrefine=4];
  end if;
  if evalb(showLegend) then
       curvePlot := implicitplot(curve(X, Y) = 0, op(passOpts),
                     color=red, caption=typeset(curve(x, y), " inscribing ",
                     nops(squares), " squares.")):
  else
       curvePlot := implicitplot(curve(X, Y) = 0, op(passOpts), color=red):
  end if;
  squaresPlot := [seq(plotSquare(squares[1 + i],
              colorList[1 + (i mod nops(colorList))]), i=0..nops(squares) - 1)]:
  if evalb(showComponents) then
      plotList := [curvePlot, op(squaresPlot),
              seq(point(P), P in componentPoints)];
  else
      plotList := [curvePlot, op(squaresPlot)];
  end if;
  display(plotList, scaling=constrained):
end proc:
end module:
\end{Verbatim}

% Local Variables:
% TeX-master: "thesis"
% End:

\end{document}